\documentclass[11pt,a4paper,twoside]{amsart}
\usepackage{amssymb,amsmath,amsthm,graphicx,color
}
\theoremstyle{plain}
\newtheorem{theorem}{Theorem}[section]
\newtheorem{definition}[theorem]{Definition}
\newtheorem{lemma}[theorem]{Lemma}
\newtheorem{corollary}[theorem]{Corollary}
\newtheorem{proposition}[theorem]{Proposition}

\theoremstyle{remark}
\newtheorem{remark}[theorem]{Remark}

\usepackage{hyperref}

\numberwithin{equation}{section}

\newcommand{\C}{\mathbb{C}}
\newcommand{\R}{\mathbb{R}}
\newcommand{\Z}{\mathbb{Z}}
\newcommand{\N}{\mathbb{N}}

\newcommand{\F}{\mathcal{F}}

\newcommand{\I}{\infty}
\newcommand{\abs}[1]{\left\lvert #1\right\rvert}
\newcommand{\norm}[1]{\left\lVert #1\right\rVert}
\newcommand{\Lebn}[2]{\left\lVert #1 \right\rVert_{L^{#2}}}

\newcommand{\Mn}[5]{\left\lVert #1 \right\rVert_{\dot{M}^{#2}_{{#3},{#4}}({#5})}}
\newcommand{\LMn}[4]{\left\lVert #1 \right\rVert_{L^{{#2},2}(\R,\dot{M}^{#3}_{{#4},{2}})}}
\newcommand{\LMnI}[5]{\left\lVert #1 \right\rVert_{L^{{#2},2}({#5},\dot{M}^{#3}_{{#4},{2}})}}

\newcommand{\LrMn}[5]{\left\lVert #1 \right\rVert_{L^{{#2},{#5}}(\R,\dot{M}^{#3}_{{#4},2})}}
\newcommand{\M}[4]{\dot{M}^{#1}_{{#2},{#3}}({#4})}
\newcommand{\Ms}[3]{\dot{M}^{#1}_{{#2},{#3}}}
\newcommand{\Mne}[5]{\left\lVert |#5|^s M(-{#5}){#1} \right\rVert_{\dot{B}^{#2}_{{#3},{#4}}}}
\newcommand{\hSobn}[2]{\left\lVert #1 \right\rVert_{\dot{H}^{#2}}}
\newcommand{\hBn}[4]{\left\lVert #1 \right\rVert_{\dot{B}^{#2}_{{#3},{#4}}}}
\newcommand{\Jbr}[1]{\left\langle #1 \right\rangle}

\newcommand{\IN}{\quad\text{in }}
\newcommand{\wIN}{\quad\text{weakly in }}
\newcommand{\Xtn}[4]{\left\lVert #1 \right\rVert_{\dot{X}^{#2}_{#3}({#4})}}
\newcommand{\LXtn}[4]{\left\lVert #1 \right\rVert_{L^{{#2},2}(\R,\dot{X}^{#3}_{#4})}}
\newcommand{\Xt}[3]{\dot{X}^{#1}_{#2}({#3})}
\newcommand{\Xts}[2]{\dot{X}^{#1}_{#2}}
\newcommand{\Xtne}[4]{\left\lVert |#4|^s M(-{#4}){#1} \right\rVert_{\dot{H}^{#2}_{#3}}}

\def\({\left(}
\def\){\right)}
\def\<{\left\langle}
\def\>{\right\rangle}
\def\le{\leqslant}
\def\ge{\geqslant}

\def\d{{\partial}}

\newcommand{\al}{\alpha}

\newcommand{\eps}{\varepsilon}

\newcommand{\pst}{p_{\mathrm{St}}}

\DeclareMathOperator{\supp}{supp}
\newcommand{\FHsc}{\F \dot{H}^{s_c}}

\begin{document}
\title[Mass-subcritical NLS]
{A sharp scattering condition\\ for focusing mass-subcritical nonlinear
Schr\"odinger equation}
\author
{Satoshi Masaki}
\address{Laboratory of Mathematics\\
Institute of Engineering\\
Hiroshima University\\
Higashihiroshima Hirhosima, 739-8527, Japan}
\email{masaki@amath.hiroshima-u.ac.jp}
\date{}
\maketitle
\vskip-5mm
\begin{abstract}
This article is concerned with time global behavior of solutions to
focusing mass-subcritical nonlinear Schr\"odinger equation
of power type with data in a critical homogeneous weighted $L^2$ space.
We give a sharp sufficient condition for scattering by
proving existence of a threshold solution
 which does not scatter at least for one time direction 
and of which initial data attains minimum value of a norm of the weighted $L^2$ space
in all initial value of non-scattering solution.
Unlike in the mass-critical or -supercritical case, ground state is not a threshold.
This is an extension of previous author's result to
the case where the exponent of nonlinearity is below so-called Strauss number.
A main new ingredient is a stability estimate in a Lorenz-modified-Bezov type spacetime norm.
\end{abstract}

\section{Introduction}
In this article, we continue our study in \cite{Ma}
on time global behavior of solutions to
the following focusing mass-subcritical nonlinear Schr\"odinger equation
\begin{equation}\tag{NLS}\label{eq:NLS}
	i \d_t u + \Delta u =- |u|^{p-1} u, \quad (t,x)\in \R^{1+N}
\end{equation}
with initial condition
\begin{equation}\tag{IC}\label{eq:IC}
	u(0)=u_0 \in \FHsc,
\end{equation}
where $N\ge 1$, the exponent of the nonlinearity is in the range
\begin{equation}\label{cond:p}
	\max\(1+\frac2N, 1+\frac4{N+2} \) < p < 1+\frac4{N},
\end{equation}
and $s_c := \frac2{p-1} -\frac{N}2$.
Here, the function space in which initial data lie is 
homogeneous weighted $L^2$ space
\[
	\FHsc=\FHsc(\R^N) :=\{ f \in S^\prime \ |\ \F^{-1} f \in \dot{H}^{s_c} (\R^N) \}
\]
with norm
$\norm{f}_{\FHsc} = \Lebn{|x|^{s_c}f}2$, where $\F$ stands for usual Fourier transform in $\R^N$
and a homogeneous Sobolev space $\dot{H}^{s}(\R^N)$ is defined as
 \[
	\dot{H}^{s}(\R^N) = \{ f \in L^{\frac{2N}{N-2s}}(\R^N)\ |\ \norm{|\nabla|^s f}_{L^2} <\I \}
 \]
for $0<s<N/2$.
Remark that $0< s_c <\min (1,N/2)$ as long as \eqref{cond:p}.
The equation \eqref{eq:NLS} has the following scaling;
if $u(t,x)$ is a solution to \eqref{eq:NLS} then
for all $\lambda>0$
\begin{equation}\label{eq:scaling}
	u_{[\lambda]} (t,x) = \lambda^{\frac{2}{p-2}}u(\lambda^2 t, \lambda x)
\end{equation}
is also a solution to \eqref{eq:NLS}. 
The initial space $\FHsc$ is scaling critical
in such a sense that $\norm{u_{0\{\lambda\}}}_{\FHsc} = \norm{u_0}_{\FHsc}$
holds for any $\lambda>0$, where
\begin{equation}\label{eq:scaling2}
	f_{\{\lambda\}}(x)=\lambda^{\frac{2}{p-2}}f(\lambda x).
\end{equation}
For solutions to \eqref{eq:NLS}, mass
\[
	M(u) := \int_{\R^N} |u|^2 dx,
\]
and energy
\begin{equation}\label{def:energy}
	E(u) := \int_{\R^N} \(\frac12 |\nabla u|^2 + \frac1{p+1}|u|^{p+1}\) dx,
\end{equation}
are well-known conservative quantities.
The case $p=1+\frac4N$ is referred to as a mass-critical case since
the scaling \eqref{eq:scaling2} leaves the mass invariant.
Similarly, we call the case $p=1+\frac4{N-2}$ ($N\ge3$) an energy-critical case.
A pseudo conformal energy
\begin{equation}\label{def:cenergy}
	\int_{\R^N} \( |(x+2it\nabla) u|^2 + \frac{8t^2}{p+1} |u|^{p+1} \) dx,
\end{equation}
is known to be useful to obtain a priori decay estimate of solution,
while it is conserved only in the mass-critical case $p =1+\frac4N$.
In our setting $u_0 \in \FHsc$, 
none of the above quantities make sense in general.

Our aim here is analysis of time global behavior of solutions to \eqref{eq:NLS}-\eqref{eq:IC}
in mass-subcritical range $p<1+\frac4N$.
Nakanishi and Ozawa show in \cite{NO} that, under an assumption on $p$ weaker than \eqref{cond:p},
if an initial data belongs to $\FHsc \cap L^2$ and if it is small with respect to the $\FHsc$ norm,
then the solution to \eqref{eq:NLS}-\eqref{eq:IC} 
exists globally in time and behaves like a linear solution near $t=\pm\I$ (see also \cite{CW1,GOV}).
This behavior, which is referred to as \emph{scattering}, occurs when the
linear dispersive effect becomes dominant for large time. 
Remark that smallness of datum gives that of corresponding solution, which
is closely related to weakness of nonlinear effect relative to linear effect.
On the other hand, there exists a non-scattering solution.
A typical example is a standing-wave solution $e^{it} \phi(x)$, where $\phi$ is a solution to 
the elliptic equation
\begin{equation}\label{eq:sp}
	-\Delta \phi + \phi = |\phi|^{p-1} \phi.
\end{equation}
There exists a unique positive radial solution $Q(x)$ to \eqref{eq:sp},
provided $1<p < 1 + \frac{4}{N-2}$ ($1<p<\I$ if $N=1,2$),
see \cite{CazBook} and references therein.
It is called a ground state and is
characterized as the function minimizing the energy $E(u)$
among all nontrivial $H^1$ functions under mass constraint.
Another behavior is blowup (in finite time).
Fibich shows in \cite{Fi} that a self-similar blowup solution exists in the 
mass-subcritical case $1<p<1+4/N$.

Recently, remarkable progresses are made on classification of initial space
according to global behavior of corresponding solution.
In the energy-critical case,
global behavior of
$\dot{H}^1$-solutions of energy less than that of $W=(1+\frac{|x|^2}{N(N-2)})^{-\frac{N-2}2}$,
which is a solution to an elliptic equation $\Delta W + |W|^{\frac4{d-2}}W=0$,
is precisely analyzed.
It is shown that initial space consists of two disjoint sets, say $S$ and $B$,
and that if an initial data belongs to $S$ then a corresponding solution of \eqref{eq:NLS}
scatters for both time directions,
while if data belongs to $B$ (and if some further assumption is satisfied)
then a solution blows up in finite time.
This is first given by Kenig and Merle \cite{KM} for $3\le N\le 5$
under radial assumption, and is extended by Killip and Visan \cite{KV}
to non-radial $N\ge5$ case.
In \cite{AN,HR}, a similar classification is given in $H^1$ framework
in the energy-subcritical and mass-supercritical case $1+\frac4N<p< 1+\frac4{N-2}$
($1+\frac4N<p<\I$ if $N=1,2$).
See also \cite{DHR,FXC}.
In \cite{NS-CVPDE}, Nakanishi and Schlag consider \eqref{eq:NLS} with $p=N=3$
and radial $H^1$ data of energy at most slightly above that of the ground state $Q$,
and show that initial space splits into nine nonempty disjoint sets
which are characterized by distinct behavior of solution.
For the mass-critical case,
Dodson shows in \cite{Do} that if an $L^2$-solution $u(t)$ satisfies
$\Lebn{u(0)}2<\Lebn{Q}2$ then the solution exists globally in time and
scatters in $L^2$ for both positive and negative time (see also \cite{W-CMP}).

As for this subject, in contrast, the mass-subcritical case $p<1+4/N$ is less understood.
In the previous paper \cite{Ma},
the case $\pst<p<1+4/N$ and $u_0 \in \F H^1:=\{f\in \mathcal{S}' \ |\ \F^{-1} f \in H^1(\R^N)\}$ is treated,
where
\begin{equation}\label{def:pst}
	\pst := 1+ \frac{2-N+\sqrt{N^2+12N+4}}{2N} 
\end{equation}
is a Strauss exponent,
and a sharp sufficient condition for scattering is given by
demonstrating that there exists a special solution $\widetilde{u}_c(t)$ such that
$\widetilde{u}_c(t)$ does not scatter either for positive or negative time,
and that if $u_0\in \F H^1$
satisfies $\ell_{\F H^1}(u_0) < \ell_{\F H^1} (\widetilde{u}_c(0))$
then the corresponding solution $u(t)$ to \eqref{eq:NLS}-\eqref{eq:IC} scatters for both positive and negative time,
where $\ell_{\F H^1}$ is defined as
\begin{equation}\label{def:lfh1}
	\ell_{\F H^1}(f) := \norm{f}_{L^2(\R^N)}^{1-s_c} \norm{|x|f}_{L^2(\R^N)}^{s_c} 
	= 2^{-\frac12}\inf_{\lambda>0} \norm{f_{\{\lambda\}}}_{\F H^1}
\end{equation}
for $f\in \F H^1$ with $\norm{f}_{\F H^1}^2 = \Lebn{f}2^2+\Lebn{|x|f}2^2$.
It is obvious that $\ell_{\F H^1}(\cdot )$ is invariant under \eqref{eq:scaling2}.
This criterion is similar to that in the mass critical case.
However, it is also proven that standing wave solutions are not
the threshold $u_c(t)$ any more.

In this article, we improve this result in the following two directions.
Firstly, the restriction $p>\pst$ is removed and the possible range of $p$ is extended as in \eqref{cond:p}.
Secondly, we take an initial data from $\FHsc$ which is wider than $\F H^1$. 

The Strauss exponent $\pst$, which is a positive root of $N p^2 -(N+2)p -2 =0$, appears
at least in the following two contexts.
One is time global behavior of defocussing case where the right hand side of 
\eqref{eq:NLS} has the opposite sign, $+|u|^{p-1}u$.
It is known that asymptotic completeness holds in $\F H^1$ for $p\ge\pst$.
More precisely, if $p\ge \pst$ then a priori decay estimate of $\F H^1$-solution, 
which follows from a priori estimate on the pseudo conformal energy \eqref{def:cenergy}, is sufficient
to show that every $\F H^1$-solution scatters both for positive and negative times
(see \cite{CW1,NO,Tsu2}).
Another is validity of a stability type estimate called ``long-time perturbation''.
When $p>\pst$, we are able to obtain a stability type estimate
with a Lebesgue-Lebesgue type space-time norm $L^\rho_t L^q_x$ (see \cite{Ka,Ma}).

The restriction $p>\pst$ in the previous study \cite{Ma}  comes from the latter point.
To remove this, we establish a new stability estimate 
with respect to a Lorentz-modified-Bezov type space. 
It is known that spatial decay of initial data yields time decay of
solution to free Schr\"odinger equation (see \cite{CazBook,GOV,NO}).
As long as the linear dispersive effect is dominant, nonlinear solution inherits this property.
The Lorentz-modified-Besov space, introduce in \cite{NO}, enables us to
take time decay effect due to fractionally weighted data into account efficiently and
to treat fractional weights in a way very similar to that for fractional derivatives.
To obtain the stability estimate, we need
the restriction $p>1+\frac{4}{N+2}$ for $N\ge 3$ (see Remark \ref{rmk:lbd_of_p2}).
The other restriction $p>1+\frac2N$ is not a technical one.
It is known that if $p\le 1+ \frac2N$ then any non-trivial 
$L^2$-solution does not scatter (see \cite{Ba,St}).

We would emphasize that, in our setting, finite-mass assumption $u_0\in L^2$ is removed.
The finite mass property is closely related to blowup phenomenon.
In the mass-subcritical case $1<p<1+4/N$, 
a solution with finite mass always exists globally in time due to 
the mass conservation law (see \cite{Tsu1}),
while blowup phenomenon can occur without the finite mass assumption
as shown in \cite{Fi}.
Note that, however, it is not clear so far whether or not
our assumption $u_0 \in \FHsc$ admits a finite-time blowup solution.
It is worth mentioning that local well-posedness in $\FHsc$ is not found in the previous results
\cite{CW1,GOV,Ma,NO}.
The difference is that continuous dependence property was based on 
the finite mass assumption $u_0 \in L^2$.
Our stability estimate is a key for removing this assumption.
For $N\ge 3$ and $p$ such that $0<s_c<1/2$, 
local well-posedness holds true in $\dot{H}^{-s_c} \supset \FHsc$ under radial assumption
(see \cite{Hi})

\subsection{Main results}
Before presenting our main results,
we introduce several notations.
The notion of an $\FHsc$-solution to \eqref{eq:NLS}-\eqref{eq:IC} is given in Definition \ref{def:sol} below.
Here, we merely note that an $\FHsc$-solution $u(t)$ is such that $t \mapsto U(-t)u(t)$ is
continuous $\R \supset I \to \FHsc$, where $U(t):= e^{it\Delta}$ is the free Schr\"odinger group
and $I$ is an time interval. In what follows, we say $u(t)$ scatters as $t\to\I$ (resp. $t\to-\I$) if
$I$ contains $(\tau, \I)$ (resp. $(-\I,\tau)$) for some $\tau \in \R$
and the limit $\lim_{t\to\I} U(-t) u(t)$
(resp. $\lim_{t\to-\I} U(-t) u(t)$) exists in $\FHsc$ sense, unless otherwise stated.
Let
\[
	S_\pm := \left\{u_0 \in \FHsc \ \Big|\ 
	\begin{aligned}
	&\text{solution}\ u(t) \text{ to }\eqref{eq:NLS} \text{ with} \\
	&u(0)= u_0 \text{ scatters as } t \to \pm \I. 
	\end{aligned}
	\right\}.
\]
We also set a scattering set $S=S_+ \cap S_-$.
Let us define
\begin{equation}\label{def:lc}
	\ell_{c}
	:= \inf\{ \norm{f}_{\FHsc} \ |\ f \in \FHsc \setminus S \}.
\end{equation}
Then, we have $\ell_c \in (0,\I)$.
Indeed, by small data scattering,
there exists a constant $\eta>0$ such that $\ell_c \ge \eta$
(see Remark \ref{rmk:lc_low}).
On the other hand, $\ell_c \le \ell (Q)$
follows from the fact that the ground state solution $e^{it}Q(x)$ is a non-scattering solution.
Remark that $Q \in \FHsc$, that is, $\ell (Q)<\I$ since
$Q(x)$ is smooth and decays exponentially as $|x|\to \I$.

Our main theorems are as follows.
\begin{theorem}\label{thm:main1}
Suppose \eqref{cond:p}. Then, the Cauchy problem \eqref{eq:NLS}-\eqref{eq:IC} is locally well-posed in $\FHsc$.
Further, there exists
a initial data $u_{0,c} \in \FHsc \setminus S_+$ such that
$\norm{u_{0,c}}_{\FHsc} = \ell_c$.
\end{theorem}
\begin{theorem}\label{thm:main2}
Suppose \eqref{cond:p}. If $u_0 \in \dot{H}^1 \cap \FHsc$ and
if $E(u_0)<0$ then $u_0 \not\in S_+\cup S_-$ and $\norm{u_0}_{\FHsc} > \ell_c$.
In particular, a solution of which initial data is $u_{0,c}$ (given in Theorem \ref{thm:main1}) is not 
a standing wave solution.
\end{theorem}
\begin{remark}
By definition of $\ell_c$, we obtain the following criterion for scattering;
if $u_0 \in \FHsc$ is such that $\norm{u_0}_{\FHsc} < \norm{u_{0,c}}_{\FHsc}$ then $u_0 \in S$.
Notice that Theorem \ref{thm:main1} implies that this criterion is sharp.
Namely, the above criterion actually fails when $\norm{u_0}_{\FHsc} = \norm{u_{0,c}}_{\FHsc}$. 
\end{remark}
\begin{remark}
In Theorem \ref{thm:main1}, the assumption that the equation is focusing is used only 
for justifying $\ell_c <\I$.
Namely, the conclusion of Theorem \ref{thm:main1} holds true
also for the defocussing equation, provided we assume $\ell_c <\I$.
As for the short-range defocussing equations, a conjecture is that asymptotic
completeness holds, that is, $\ell_c=+\I$ or equivalently $S=\FHsc$.
This conjecture is true when $\pst \le p < 1+4/N$ and initial space is $\F H^1 $
(see \cite{CW1,NO,Tsu2}).
In the case $1+2/N < p < \pst$, it is shown in \cite{TY} that 
if initial data is taken from $H^1 \cap \F H^1 $ then a solution globally in time and
scatters in $L^2$ for both time directions.
Theorem \ref{thm:main1} reduces the conjecture to nonexistence of a critical element.
If existence of the critical element $u_{0,c}$, the conclusion of Theorem \ref{thm:main1},
yields any contradiction, then we obtain the conjecture.
\end{remark}
\begin{remark}
In the mass-critical case, the ground state $ Q(x)$ plays the role of $u_{0,c}$ (\cite{Do}).
Theorem \ref{thm:main2} says that the situation is completely different in the mass-subcritical case.
It seems reasonable because the ground state solution $e^{it}Q(x)$ is 
orbitally stable in $H^1$ in the mass subcritical case (see \cite{CazBook} and references therein).
\end{remark}

\begin{remark}
Rigorously speaking, the statement in
Theorems \ref{thm:main1} and \ref{thm:main2} is not a direct extension of that in \cite{Ma}.
It is because the minimizing problem with respect to $\norm{\cdot}_{\FHsc}$ is 
different from that with respect to $\ell_{\F H^1}(\cdot)$.
One sees this from the fact that there exist functions $\psi_1,\psi_2 \in \F H^1$ such that
$\norm{\psi_1}_{\FHsc}<\norm{\psi_2}_{\FHsc}$ and $\ell_{\F H^1}(\psi_1)>\ell_{\F H^1}(\psi_2)$
(such example will be given in Remark \ref{rmk:FHscFH1}).
Nevertheless, the result of \cite{Ma} can be also extended
to the case \eqref{cond:p} as mentioned in Remark \ref{rmk:FH1extension}.
In other words, our argument enables us to consider the both minimizing problems.
In this sense, Theorems \ref{thm:main1} and \ref{thm:main2} are ``extensions''
of results of \cite{Ma}.
\end{remark}

If an initial data has a rapid quadratic oscillation, then a solution is global and scatters.
This kind of result is known for $p>\pst$ and $u_0 \in \F H^1$ (see \cite{CW1,Ma}). 
As a byproduct of our main theorems,
we improve this result as follows.
\begin{theorem}\label{thm:main3}
Suppose \eqref{cond:p}. For any $\psi \in \FHsc$, there exists $b_0>0$ such that
if $|b|>b_0$ then $e^{ib |x|^2} \psi \in S$.
\end{theorem}
\begin{remark}
Theorem \ref{thm:main3} holds true also for $\psi \in \F H^1$:
An $\F H^1$-solution of \eqref{eq:NLS} with data $e^{ib |x|^2} \psi$ scatters in $\F H^1$-sense
for both time directions, provided $|b|$ is sufficiently large (see Remark \ref{rmk:odsFH1}).
This is a direct extension of the result of \cite{CW1,Ma} to the case \eqref{cond:p}.
\end{remark}
Theorem \ref{thm:main3} yields an unboundedness property of the scattering set $S$.
\begin{corollary}
Suppose \eqref{cond:p}. It holds that
\[
	\sup_{u_0 \in S} \inf_{x_0 \in \R^N} \norm{u_0(\cdot + x_0)}_{\FHsc} = +\I.
\]
\end{corollary}
Unboundedness of $S$ itself is trivial in view of Galilean transform.
Indeed, if $u_0 \in S\cap L^2_{\mathrm{loc}}$ then,  for any $x_0\in \R^N$,
$u(t,x-x_0)$ is also a solution with data $u_0(x-x_0)$ and 
further it scatters since $U(-t)(u(t,\cdot -x_0))(x) = (U(-t)u(t))(x-x_0)$.
Thus, $u_0(\cdot - x_0) \in S$ for any $x_0\in \R^N$.
The unboundedness of $S$ then follows from the fact that $\norm{u_0(\cdot-x_0)}_{\FHsc}\to \I$ as $|x_0|\to\I$.
However, the unboundedness above follows by a completely different cause, 
a rapid oscillation of initial data.
Theorem \ref{thm:main3} also suggest that some quantity other than $\ell(u_0)$ should be
taken into account for complete classification of behavior of solutions.

The rest of the paper is organized as follows.
We first define function spaces which we work with and
collect several preliminary facts and estimates in Section \ref{sec:2}.
We then begin our study with local well-posedness in $\FHsc$ (Theorem \ref{thm:lwp}) and
a stability type estimate, long time perturbation,
in a Lorentz-modified-Bezov space (Theorem \ref{thm:lpt}) in Section \ref{sec:3}.
A necessary and sufficient condition for scattering is also given in this section (Proposition \ref{prop:nscond}).
Section \ref{sec:4} is devoted to 
proof of concentration compactness
for sequences bounded in $\FHsc$ (Proposition \ref{prop:pd}).
Smallness due to rapid oscillation of initial data is considered there (Proposition \ref{prop:ods}).
We then prove our main theorems in Sections
\ref{sec:6}, \ref{sec:7}, and \ref{sec:5}.

\section{Preliminaries}\label{sec:2}
\subsection{Function spaces}
For our analysis, we prepare several notations and function spaces.

We first define a function space $\dot{X}^s_{q}(t)$ defined by
the following norm:
\begin{equation}\label{def:Xn}
	\Xtn{f}sqt := \norm{U(t)|x|^sU(-t)f}_{L^q(\R^N)}
\end{equation}
for $s,t \in \R$ and $1\le q\le\I$, where $U(t)=e^{it\Delta}$.
For $t\neq 0$,
let $M(t)$ be a $t$-dependent multiplication operator defined by
$M(t):=e^{i\frac{|x|^2}{4t}} \times$
and let $D(t)$ be a dilation operator defined by
	$D(t)f:= (2\pi i t)^{-N/2} f\(\frac{x}{2t}\)$.
One deduces from the well-known factorization
$U(t)=M(t)D(t)\F M(t)$ that
the identity 
\begin{equation}\label{eq:prelim1}
	U(t)[\phi U(-t)f]  = M(t) [\phi(2it\nabla) (M(-t)f)]
\end{equation}
holds for a suitable function $\phi$.
Hence, we have the following equivalent norm
for $t\neq0$:
\begin{equation}\label{eq:Xn_alt}
	\Xtn{f}sqt \sim \Xtne{f}sqt.
\end{equation}
The space $\dot{X}^s_{q}(t)$ is regarded as a modified Sobolev space in this sense.

Let a sequence $\{ \varphi_j \}_{j\in\Z}$ be a Littlewood-Paley decomposition, that is,
$\varphi_j(\xi) = \varphi_0 (2^{-j}\xi) \in C_0^\I(\R^N)$,
$\varphi_0\ge0$, $\supp \varphi_0 \subset \{2^{-1}<|\xi|<2\}$,
and $\sum_{j\in\Z} \varphi_j(\xi)=1$ for $\xi\neq0$.
We use a function space defined by the following norm
\begin{equation}\label{def:Mn}
	\Mn{f}sqrt := \norm{2^{js} U(t) \varphi_j U(-t)f}_{\ell_j^r(\Z,L^q_x)}
\end{equation}
for $s,t\in \R$, and $1\le q,r \le \I$.
By \eqref{eq:prelim1}, we have the following equivalent norm
\begin{equation}\label{eq:Mn_alt}
	\Mn{\psi}sqrt \sim \Mne{\psi}sqrt
\end{equation}
for $t\neq0$, where $\hBn{\cdot}sqr$ is a norm of usual homogeneous Besov space.
We refer the space $\M{s}qrt$ to as a modified Besov space.

The following fundamental properties of $\dot{X}^s_q(t)$ and $\dot{M}^s_{p,q}(t)$
are summarized as follows. 
\begin{lemma}\label{lem:embedding1}
Let $s,t \in \R$, $1\le  q  \le \I$, and $1\le r_1 \le r_2 \le \I$.
Then, the followings hold.
\begin{enumerate}
\item $\dot{X}^0_q(t) =L^q$.
\item $\dot{M}^s_{q,r_1}(t) \hookrightarrow \dot{M}^s_{q,r_2}(t)$.
\item $\dot{M}^s_{q,\min(q,2)}(t) \hookrightarrow \dot{X}^s_q(t) \hookrightarrow \dot{M}^s_{q,\max(q,2)}(t)$ if $t\neq0$ and $1<  q  < \I$.
In particular, $\dot{M}^{s}_{2,2}(t) =\dot{X}^s_2(t)$ and $\dot{M}^0_{q,\min(q,2)}(t)\hookrightarrow L^q$.
\item $\dot{M}^s_{q,q}(0)=\dot{X}^s_{q}(0)$.
In particular, $\dot{M}^{s}_{2,2}(0) =\dot{X}^s_2(0)= \F \dot{H}^s$.
\end{enumerate}
\end{lemma}
The first and the last properties are obvious by definition.
The rest are easily derived from properties of Sobolev and Besov spaces (see \cite{BL-Book})
by means of \eqref{eq:Xn_alt} and \eqref{eq:Mn_alt}.

We also use the space-time function space of the form $L^{\rho,\gamma}(I, \M{s}qrt)$ for an interval $I$,
where $1\le \rho \le \I$ and $1\le \gamma \le \I$ and $L^{\rho,\gamma}(I)$ is the Lorentz space.
Here, the case $\rho=\I$ and $\gamma<\I$ is excluded.
For simplicity, we sometimes omit $(t)$ in this kind of norm if it is clear from the context.
One of the fundamental tool for analysis in Lorentz space is 
the following.
\begin{proposition}[Generalized H\"older's inequality \cite{Na}]\label{prop:gH}
Let $1\le \rho,\rho_1,\rho_2<\I$ and $1\le \gamma,\gamma_1,\gamma_2\le\I$
satisfy $1/\rho=1/\rho_1+1/\rho_2$ and $1/\gamma=1/\gamma_1+1/\gamma_2$.
Then,
\[
	\norm{fg}_{L^{\rho,\gamma}}
	\le C\norm{f}_{L^{\rho_1,\gamma_1}} \norm{g}_{L^{\rho_2,\gamma_2}}.
\]
\end{proposition}
The following embedding is useful for our analysis.
\begin{lemma}\label{lem:LMn_inclusion}
Let $\rho_1, \rho_2 \in [1,\I)$, $s_1,s_2 \in \R$, $q_1,q_2 \in [1,\I]$ satisfy
\[
	s_2-s_1 = \frac1{\rho_1}-\frac1{\rho_2} = \frac{N}{q_2}-\frac{N}{q_1} \ge0.
\]
Then, we have
\[
	\LMn{f}{\rho_1}{s_1}{q_1} \le C \LMn{f}{\rho_2}{s_2}{q_2}
\]
\end{lemma}
This is a special case of the following interpolation type inequality.
\begin{lemma}\label{lem:interpolation}
Let $\rho,\rho_1, \rho_2 \in [1,\I)$, $s,s_1,s_2 \in \R$, $q,q_1,q_2 \in [1,\I]$, and $r, r_1,r_2 \in (0.\I]$ satisfy
\[
	(\theta s_1+(1-\theta)s_2)-s = \frac1{\rho}-\(\frac\theta{\rho_1}+\frac{1-\theta}{\rho_2}\) 
	= N\(\frac{\theta}{q_1}+\frac{1-\theta}{q_2}\)-\frac{N}{q} \ge0
\]
and $\frac1r \ge \frac\theta{r_1}+\frac{1-\theta}{r_2}$
for some $\theta\in[0,1]$. 
Then, we have
\[
	\LrMn{f}{\rho}{s}{q}r \le C \LrMn{f}{\rho_1}{s_1}{q_1}{r_1}^\theta
	\LrMn{f}{\rho_1}{s_1}{q_1}{r_2}^{1-\theta}.
\]
\end{lemma}
\begin{proof}
Since $1/q \le \theta/{q_1} + (1-\theta)/{q_2}$, we 
can take $ a_i\in [q_i,\I]$ ($i=1,2$) so that
$1/q=\theta/{a_1} + (1-\theta)/{a_2}$.
Then, H\"older's inequality gives us
\[
	\Lebn{\hat{\varphi}_j*f}q 
	\le \Lebn{\hat{\varphi}_j*f}{a_1}^\theta
	\Lebn{\hat{\varphi}_j*f}{a_2}^{1-\theta}.
\]
By definition, $\hat{\varphi}_j(x)=2^{jN}\hat{\varphi}_0(2^jx)$ and so
$\Lebn{\hat{\varphi}_j}p = 2^{jN(1-\frac1p)}\Lebn{\hat{\varphi}_0}p$ for $p\in[1,\I]$.
Since $(\varphi_{j-1}+\varphi_j+\varphi_{j+1})\varphi_j=\varphi_j$, we have
\begin{align*}
	\Lebn{\hat{\varphi}_j*f}{a_i} 
	={}& \Lebn{(\hat{\varphi}_{j-1}+\hat{\varphi}_j+\hat{\varphi}_{j+1})*(\hat{\varphi}_j*f)}{a_i}\\
	\le {}& C 2^{jN(1-\frac1{r_i})} \Lebn{\hat{\varphi}_j*f}{q_i}
\end{align*}
for $i=1,2$,
where $r_i \in [1,\I]$ is such that
\[
	\frac1{a_i} = \frac1{r_i}+\frac1{q_i}-1.
\]
Remark that such $r_i$ exists because $a_i\ge q_i\ge1$.
Therefore,
\[
	\Lebn{\hat{\varphi}_j*f}q 
	\le C 2^{jN(\frac{\theta}{q_1}+\frac{1-\theta}{q_2})-j\frac{N}q} \Lebn{\hat{\varphi}_j*f}{q_1}^\theta
	\Lebn{\hat{\varphi}_j*f}{q_2}^{1-\theta}.
\]
By the assumption $(\theta s_1+(1-\theta)s_2)-s = N\(\frac{\theta}{q_1}+\frac{1-\theta}{q_2}\)-\frac{N}{q}$,
\[
	2^{js} \Lebn{\hat{\varphi}_j*f}q 
	\le C  (2^{js_1}\Lebn{\hat{\varphi}_j*f}{q_1})^\theta
	(2^{js_2}\Lebn{\hat{\varphi}_j*f}{q_2})^{1-\theta}.
\]
Taking $\ell^2$ norm in $j$ and using H\"older's inequality, we obtain
\[
	\hBn{f}sq2 \le \hBn{f}{s_1}{q_1}2^\theta \hBn{f}{s_2}{q_2}2^{1-\theta}.
\]
From this estimate,
\begin{align*}
	\Mn{f}sq2t \sim{}& |t|^s \hBn{M(-t) f}sq2\\
	\le{}& |t|^s \hBn{M(-t)f}{s_1}{q_1}2^\theta \hBn{M(-t)f}{s_2}{q_2}2^{1-\theta}\\
	\sim{}& |t|^{s-(\theta s_1+(1-\theta)s_2)} \Mn{f}{s_1}{q_1}2t^\theta \Mn{f}{s_2}{q_2}2t^{1-\theta}.
\end{align*}
Then, the result follows from the generalized H\"older inequality
(Proposition \ref{prop:gH}).
We only note that $|t|^{-\al} \in L^{1/\al,\I}(\R)$ for $\al\ge0$.
\end{proof}

\begin{proposition}\label{prop:LM_approx}
Let $1<\rho_1,\rho_2<\I$, $s_1,s_2 \in \R$, and $1 \le q_1,q_2, r_1, r_2 <\I$.
Let $v \in L^{\rho_1,2}(\R, \Ms{s_1}{q_1}{r_1}) \cap L^{\rho_2,2}(\R, \Ms{s_2}{q_2}{r_2})$.
For any number $\eps>0$ there exists a function $\widetilde{v}(t,x)$ defined on $\R^{1+N}$ such that
$\supp \widetilde{v} \subset \{ (t,x) \in \R^{1+N}\ |\ \delta\le |t| \le M,\, |x| \le R \}$
for some positive numbers $\delta,M,R$, and that
\[
	\sum_{j=1,2}\norm{v- \widetilde{v}}_{L^{\rho_j,2}(\R, \Ms{s_j}{q_j}{r_j})} \le \eps.
\]
\end{proposition}
\begin{proof}
Since $\rho_1,\rho_2<\I$, we can take $\delta,M>0$ so that
\[
	\sum_{j=1,2}\norm{v}_{L^{\rho_j,2}((-\I,-M)\cup (-\delta,\delta)\cup(M,\I), \Ms{s_j}{q_j}{r_j})}
 \le \frac{\eps}2.
\]
Then, it suffices to approximate $v(t) \in \M{s_1}{q_1}{r_1}t \cap \M{s_2}{q_2}{r_2}t$
by a function which has a compact support, for each fixed $t \in [-M,-\delta]\cup[\delta,M]$.
By the equivalent representation \eqref{eq:Mn_alt},
approximation of $w:=M(-t)v(t)$ in $\dot{B}^{s_1}_{q_1,r_1} 
\cap \dot{B}^{s_2}_{q_2,r_2}$ is sufficient.
This is an immediate consequence of the atomic decomposition.
Indeed, for $w\in \dot{B}^{s_1}_{q_1,r_1}$, we have the
following decomposition;
\begin{equation}\label{eq:Mnalt1}
	w = \sum_{j\in \Z} \sum_{\ell(Q)=2^{-j}} s_Q a_Q,
\end{equation}
where $Q$'s are dyadic cubes, $s_Q$'s are constants satisfying
\begin{equation}\label{eq:Mnalt2}
	\(\sum_{j\in \Z} \(\sum_{\ell(Q)=2^{-j}} |s_Q|^{q_1}\)^{r_1/q_1}\)^{1/r_1} \le C
	\norm{w}_{\dot{B}^{s_1}_{q_1,r_1}}
\end{equation}
and $a_Q$'s are the $(s_1,q_1)$-atoms.
Remark that we can construct $\{s_Q\}$ and $\{a_Q\}$ 
independently of $s$, $q$, and $r$ (see \cite[Theorem 2.6]{FJ}).
Therefore, $w\in \dot{B}^{s_2}_{q_2,r_2}$ implies that the above decomposition holds
in $\dot{B}^{s_2}_{q_2,r_2}$ with the same $a_Q$ and $s_Q$.
Now, a linear combination of sufficiently large (but finite) number of $s_Qa_Q$
is a desired approximation of $w$ since the $\dot{B}^{s_1}_{q_1,r_1}$ norm
is equivalent to the infimum of the left hand side of \eqref{eq:Mnalt2} in all
decompositions of the form \eqref{eq:Mnalt1} and since $q_1,q_2, r_1,r_2 <\I$.
\end{proof}

\subsection{Strichartz' estimates}
We use the notation $\delta(q)=N(\frac12-\frac1q)$.
Set $2^{*}=\I$ for $N=1,2$ and $2^*=2N/(N-2)$ for $N\ge3$.
\begin{definition}
A pair $(\rho,q)$ is said to be \emph{acceptable} if
$1<\rho<\I$, $2<q<2^*$, and $\rho \delta(q)>1$.
We say a pair $(\rho,q)$ is \emph{admissible} if $(\rho,q)$ is acceptable and
$\rho \delta(q)=2$.
\end{definition}
We emphasize that, in our terminology, the end points $(\rho,q)=(\I,2)$, $(1,2^*)$
($(\rho,q)=(\I,2)$, $(2,\I)$ if $N=1$) are excluded from admissible pairs.
Figure 1 shows the range of acceptable pair for $N\ge3$.
Here, $A=(1/2^*,0)$, $B=(1/2,0)$, $C=(1/2^*,1)$, and $D=(1/2^*,1/2)$.
A pair $(\rho, q)$ is acceptable
if $(1/q,1/\rho)$ is in interior of triangle $ABC$.
The lines $BC$ and $BD$ correspond to $\frac1\rho-\delta(q)=0$ and
$\frac2\rho-\delta(q)=0$, respectively.
Therefore, a pair $(\rho, q)$ is admissible if  $(1/q,1/\rho)$ lies on
the line segment $BD$.
Let $F=(1/2^*,(1-k)/2)$ and $H=(1/2^*,(1+k)/2)$ for some $0<k<1$.
The lines $EF$ and $GH$ are parallel to $BD$.
Then, $EF$ corresponds to $\frac2\rho-\delta(q)=k$, and
$GH$ to $\frac2\rho-\delta(q)=-k$.
\begin{center}
\unitlength 0.1in
\begin{picture}( 22.0000, 22.5000)(  0.0000,-23.0000)
%
{\color[named]{Black}{%
\special{pn 13}%
\special{pa 200 2300}%
\special{pa 200 100}%
\special{fp}%
\special{sh 1}%
\special{pa 200 100}%
\special{pa 180 168}%
\special{pa 200 154}%
\special{pa 220 168}%
\special{pa 200 100}%
\special{fp}%
\special{pa 0 2100}%
\special{pa 2200 2100}%
\special{fp}%
\special{sh 1}%
\special{pa 2200 2100}%
\special{pa 2134 2080}%
\special{pa 2148 2100}%
\special{pa 2134 2120}%
\special{pa 2200 2100}%
\special{fp}%
}}%
%
{\color[named]{Black}{%
\special{pn 8}%
\special{pa 2000 2300}%
\special{pa 2000 100}%
\special{dt 0.045}%
\special{pa 2200 300}%
\special{pa 0 300}%
\special{dt 0.045}%
\special{pa 0 1200}%
\special{pa 2200 1200}%
\special{dt 0.045}%
\special{pa 600 2300}%
\special{pa 600 100}%
\special{dt 0.045}%
}}%
%
{\color[named]{Black}{%
\special{pn 8}%
\special{pa 600 300}%
\special{pa 600 2100}%
\special{pa 2000 2100}%
\special{pa 600 300}%
\special{pa 600 2100}%
\special{fp}%
}}%
%
{\color[named]{Black}{%
\special{pn 4}%
\special{pa 600 360}%
\special{pa 626 334}%
\special{fp}%
\special{pa 600 420}%
\special{pa 652 368}%
\special{fp}%
\special{pa 600 480}%
\special{pa 680 402}%
\special{fp}%
\special{pa 600 540}%
\special{pa 706 436}%
\special{fp}%
\special{pa 600 600}%
\special{pa 732 470}%
\special{fp}%
\special{pa 600 660}%
\special{pa 758 502}%
\special{fp}%
\special{pa 600 720}%
\special{pa 784 536}%
\special{fp}%
\special{pa 600 780}%
\special{pa 810 570}%
\special{fp}%
\special{pa 600 840}%
\special{pa 836 604}%
\special{fp}%
\special{pa 600 900}%
\special{pa 862 638}%
\special{fp}%
\special{pa 600 960}%
\special{pa 890 672}%
\special{fp}%
\special{pa 600 1020}%
\special{pa 916 706}%
\special{fp}%
\special{pa 600 1080}%
\special{pa 942 740}%
\special{fp}%
\special{pa 600 1140}%
\special{pa 968 772}%
\special{fp}%
\special{pa 600 1200}%
\special{pa 994 806}%
\special{fp}%
\special{pa 600 1260}%
\special{pa 1020 840}%
\special{fp}%
\special{pa 600 1320}%
\special{pa 1046 874}%
\special{fp}%
\special{pa 600 1380}%
\special{pa 1072 908}%
\special{fp}%
\special{pa 600 1440}%
\special{pa 1100 942}%
\special{fp}%
\special{pa 600 1500}%
\special{pa 1126 976}%
\special{fp}%
\special{pa 600 1560}%
\special{pa 1152 1010}%
\special{fp}%
\special{pa 600 1620}%
\special{pa 1178 1042}%
\special{fp}%
\special{pa 600 1680}%
\special{pa 1204 1076}%
\special{fp}%
\special{pa 600 1740}%
\special{pa 1230 1110}%
\special{fp}%
\special{pa 600 1800}%
\special{pa 1256 1144}%
\special{fp}%
\special{pa 600 1860}%
\special{pa 1282 1178}%
\special{fp}%
\special{pa 600 1920}%
\special{pa 1310 1212}%
\special{fp}%
\special{pa 600 1980}%
\special{pa 1336 1246}%
\special{fp}%
\special{pa 600 2040}%
\special{pa 1362 1280}%
\special{fp}%
\special{pa 600 2100}%
\special{pa 1388 1312}%
\special{fp}%
\special{pa 660 2100}%
\special{pa 1414 1346}%
\special{fp}%
\special{pa 720 2100}%
\special{pa 1440 1380}%
\special{fp}%
\special{pa 780 2100}%
\special{pa 1466 1414}%
\special{fp}%
\special{pa 840 2100}%
\special{pa 1492 1448}%
\special{fp}%
\special{pa 900 2100}%
\special{pa 1520 1482}%
\special{fp}%
\special{pa 960 2100}%
\special{pa 1546 1516}%
\special{fp}%
\special{pa 1020 2100}%
\special{pa 1572 1550}%
\special{fp}%
\special{pa 1080 2100}%
\special{pa 1598 1582}%
\special{fp}%
\special{pa 1140 2100}%
\special{pa 1624 1616}%
\special{fp}%
\special{pa 1200 2100}%
\special{pa 1650 1650}%
\special{fp}%
\special{pa 1260 2100}%
\special{pa 1676 1684}%
\special{fp}%
\special{pa 1320 2100}%
\special{pa 1702 1718}%
\special{fp}%
\special{pa 1380 2100}%
\special{pa 1730 1752}%
\special{fp}%
\special{pa 1440 2100}%
\special{pa 1756 1786}%
\special{fp}%
\special{pa 1500 2100}%
\special{pa 1782 1820}%
\special{fp}%
\special{pa 1560 2100}%
\special{pa 1808 1852}%
\special{fp}%
\special{pa 1620 2100}%
\special{pa 1834 1886}%
\special{fp}%
\special{pa 1680 2100}%
\special{pa 1860 1920}%
\special{fp}%
\special{pa 1740 2100}%
\special{pa 1886 1954}%
\special{fp}%
\special{pa 1800 2100}%
\special{pa 1912 1988}%
\special{fp}%
}}%
%
{\color[named]{Black}{%
\special{pn 4}%
\special{pa 1860 2100}%
\special{pa 1940 2022}%
\special{fp}%
\special{pa 1920 2100}%
\special{pa 1966 2056}%
\special{fp}%
\special{pa 1980 2100}%
\special{pa 1992 2090}%
\special{fp}%
}}%
%
{\color[named]{Black}{%
\special{pn 13}%
\special{pa 2000 2100}%
\special{pa 600 1200}%
\special{fp}%
\special{pa 1580 2100}%
\special{pa 600 1470}%
\special{fp}%
\special{pa 600 930}%
\special{pa 1580 1560}%
\special{fp}%
}}%
\put(6.4000,-22.5000){\makebox(0,0)[lb]{$A$}}%
\put(18.2000,-22.7000){\makebox(0,0)[lb]{$B$}}%
\put(6.4000,-2.5000){\makebox(0,0)[lb]{$C$}}%
\put(5.5000,-12.5000){\makebox(0,0)[rt]{$D$}}%
\put(5.4000,-14.2000){\makebox(0,0)[rt]{$F$}}%
\put(5.4000,-8.8000){\makebox(0,0)[rt]{$H$}}%
\put(14.4000,-22.6000){\makebox(0,0)[lb]{$E$}}%
\put(16.0000,-15.2000){\makebox(0,0)[lb]{$G$}}%
\put(21.3000,-22.6000){\makebox(0,0)[lb]{$1/q$}}%
\put(2.6000,-1.8000){\makebox(0,0)[lb]{$1/\rho$}}%
\end{picture}%
\\
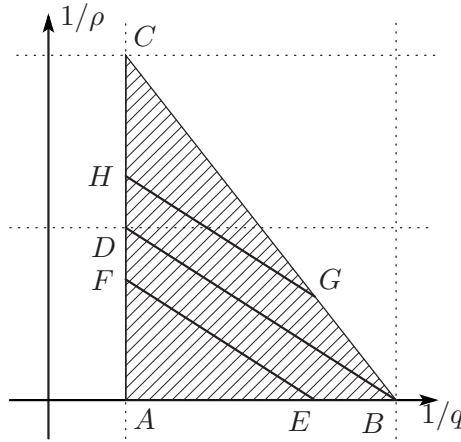
Figure 1. range of acceptable pair $(\rho,q)$. 
\end{center}
The following is a Strichartz' estimate adopted to the Lorentz-modified-Besov and
Lorentz-modified-Sobolev spaces.
\begin{proposition}[Strichartz' estimate \cite{NO}]\label{prop:Strichartz}
Let $s\in \R$ and $t_0 \in I \subset \R$.
\begin{enumerate}
\item For any admissible pair $(\rho,q)$, we have
\begin{multline*}
	\norm{U(t)f}_{L^\I(\R,\Ms{s}22)} + 
	\LMn{U(t)f}{\rho}sq \\
	+ \norm{U(t)f}_{L^\I(\R,\Xts{s}2)} + 
	\LXtn{U(t)f}{\rho}sq
	\le C_{\rho} \norm{f}_{\F \dot{H}^s}.
\end{multline*}
\item For any admissible pairs $(\rho_1,q_1)$ and $(\rho_2,q_2)$,
\begin{multline*}
	\norm{ \int_0^t U(t-s) f(s) ds}_{L^\I(I,\dot{M}^s_{2,2})} + 
	\LMnI{\int_0^t U(t-s) f(s) ds}{\rho_1}s{q_1}I \\
	\le C \LMnI{f}{\rho_2^\prime}s{q_2^\prime}I,
\end{multline*}
\begin{multline*}
	\norm{ \int_0^t U(t-s) f(s) ds}_{L^\I(I,\dot{X}^s_{2})} + 
	\norm{\int_0^t U(t-s) f(s) ds}_{L^{\rho_1,2}(I;\dot{X}^s_{q_1})}\\
	\le C \norm{f}_{L^{\rho_2',2}(I;\dot{X}^s_{q_2'})}.
\end{multline*}
\end{enumerate}
\end{proposition}
An improved version of inhomogeneous Strichartz' estimate (the second estimate)
for the Lorentz-modified-Sobolev space will be required in our analysis.
Before stating it,
we introduce an immediate consequence of Strichartz' estimate
and Lemma \ref{lem:LMn_inclusion}.
\begin{proposition}\label{prop:LMn_bdd}
For pair $(\rho,q)$ satisfying $0< \delta(q)-(s_c-s) < \min(1,N/2)$
and $\frac2\rho-\delta(q)+s=s_c$
for some $s\le s_c$, it holds that
\[
	\LMn{U(t)f}\rho{s}q \le C \norm{f}_{\FHsc}.
\]
\end{proposition}
\begin{proof}
Take $\widetilde{\rho}$ and $\widetilde{q}$ so that
\[
	\frac1{\widetilde{\rho}} = \frac1\rho - (s_c-s),\qquad
	\frac{N}{\widetilde{q}}= \frac{N}q +(s_c-s).
\]
One verifies that
\[
	\frac2{\widetilde{\rho}}= \delta(\widetilde{q})= \delta(q) - (s_c-s)\in (0,\min(1,N/2)),
\]
which implies that
$(\widetilde{\rho},\widetilde{q})$ is an admissible pair.
Hence, $\LMn{U(t)f}{\widetilde{\rho}}{s_c}{\widetilde{q}} \le C\norm{f}_{\FHsc}$
by the Strichartz estimate.
Hence the desired result is an immediate consequence of
Lemma \ref{lem:LMn_inclusion},
\end{proof}
\begin{remark}
Notice that a pair $(\rho,q)$ is not always acceptable in the preceding proposition.
Indeed, when $N\ge3$ the case $\I> q\ge 2^*$, which is equivalent to $\delta(q)\ge N/2$, is allowed.
\end{remark}
We now state the inhomogeneous Strichartz' estimate for the Lorentz-modified-Sobolev space
with non-admissible pairs.
See \cite{Fo,Ka,Ko,Vi} for this kind of extension with respect to $L^\rho_t L^q_x$ type norm.
\begin{proposition}\label{prop:imprvStrLS}
Let $s\in \R$.
Let $(\rho_i,q_i)$ ($i=1,2$) be two acceptable pairs.
If
\[
	\frac2{\rho_1} - \delta(q_1) + \frac2{\rho_2} - \delta(q_2) =0 
\]
then it holds that
\begin{equation}\label{eq:imprvStrLS}
	\norm{\int_0^t U(t-s) f(s) ds}_{L^{\rho_1,2}(I;\dot{X}^s_{q_1})}
	\le C \norm{f}_{L^{\rho_2',2}(I;\dot{X}^s_{q_2'})}.
\end{equation}
\end{proposition}
\begin{remark}
The condition of Proposition \ref{prop:imprvStrLS} reads as follows:
In figure 1,
one of $(1/q_1,1/\rho_1)$, $(1/q_2,1/\rho_2)$ lies in the line segment $GH$ for some $0<k<1$,
and the other in $EF$ for the same $k$.
\end{remark}
\begin{proof}
The estimate is known
if either $1/{q_1}+1/q_2=1$ or $2/\rho_i-\delta(q_i)=0$.

Let us first consider the case $q_1\le q_2$.
Set $Kf=\int^t_0 U(t-s) f(s) ds$.
We decompose the integral as $K=K_j f$ with
$K_j f = \int U(t-\tau)I_j(t-\tau)f(\tau)d\tau$,
where $I_j$ denotes the characteristic function of $[2^j,2^{j+1})$.
We begin with the well-known estimate
\[
	\Lebn{U(s)\varphi}q \le C|s|^{-\delta(q)} \Lebn{\varphi}{q'}
\]
for $2\le q \le \I$. By definition of $\dot{X}^s_{q}$, we have
\begin{align*}
	\Xtn{U(\tau)\varphi}sqt &{}= \norm{U(\tau)U(t-\tau)|x|^s U(-(t-\tau))\varphi}_{L^q}\\
	&{} \le C|\tau|^{-\delta(q)}\norm{U(t-\tau)|x|^s U(-(t-\tau))\varphi}_{L^{q^\prime}}\\
	&{}=C|\tau|^{-\delta(q)} \Xtn{\varphi}s{q^\prime}{t-\tau}.
\end{align*}
Apply this estimate to obtain
\begin{align*}
	\Xtn{K_j f(t)}sqt 
	&{}\le	C \norm{I_j(t-\tau)\Xtn{U(t-\tau)f(\tau)}sqt }_{L^1_\tau}\\
	&{}\le C \norm{|t-\tau|^{-\delta(q)}I_j(t-\tau)\Xtn{f(\tau)}s{q'}{\tau} }_{L^1_\tau}\\
	&{}\le C 2^{-j\delta(q)} \norm{I_j(t-\tau)f(\tau)}_{L^1_\tau (\Xt{s}{q'}\tau)}
\end{align*}
for any fixed $t\in \R$.
Therefore, we have
\begin{equation}\label{eq:imprvStrLS1}
	\Xtn{K_j f(t)}s{q_2}t \le C 2^{-j\delta(q_2)} \norm{I_j(t-\tau)f(\tau)}_{L^1_\tau (\Xt{s}{q_2^\prime}\tau)}.
\end{equation}

On the other hand, if $2/r=\delta(q) \in (0,\min(1,N/2))$,
\begin{align*}
	&\Xtn{\int U(t-\tau)g(\tau)d\tau}s2t^2
		=  \Lebn{ \int |x|^s U(-\tau)g(\tau)d\tau}2^2\\
	&{}=  \iint \Jbr{|x|^s U(-\tau_1)g(\tau_1),|x|^s U(-\tau_2)g(\tau_2)}_{L^2_x}d\tau_1d\tau_2 \\
	&{}= \int 
	\Jbr{ (U(\tau_2)|x|^s U(-\tau_2)) \int U(\tau_2-\tau_1)g(\tau_1)d\tau_1,U(\tau_2)|x|^s U(-\tau_2)g(\tau_2)}_{L^2_x}d\tau_2\\
	&{}\le C \LXtn{\int U(\tau_2-\tau) g(\tau)d\tau}rsq
	\norm{g}_{L^{r'}(\R,\dot{X}^s_{q'})}\le C \norm{g}_{L^{r'}(\R,\dot{X}^s_{q'})}^2
\end{align*}
where we have applied \eqref{eq:imprvStrLS} and
used the inclusion $L^{r,2} \subset L^r$ and $L^{r'} \subset L^{r',2}$.
From this inequality, we have
\begin{equation}\label{eq:imprvStrLS2}
	\Xtn{K_j f(t)}s2t \le C \norm{I_j(t-\tau)f(\tau)}_{L^{\frac2{2-\delta(q_2)}}_\tau (\Xt{s}{q_2^\prime}\tau)}.
\end{equation}
Interpolating \eqref{eq:imprvStrLS1} and \eqref{eq:imprvStrLS2},
one sees that
\begin{equation}\label{eq:imprvStrLS3}
	\Xtn{K_j f(t)}s{q_1}t \le C 2^{-j\delta(q_1)}\norm{I_j(t-\tau)f(\tau)}_{L^{\frac2{2+\delta(q_1)-\delta(q_2)}}_\tau (\Xt{s}{q_2^\prime}\tau)}
\end{equation}
(recall that $2<q_1 < q_2$).

As shown in \cite{NO}, it follows that the operator
\[
	\mathcal{X}: f \mapsto I_j(t-\tau) f(\tau)
\]
is bounded as a map
\begin{equation}\label{eq:imprvStrLS4}
	\mathcal{X}: L^{\frac1\al,2} \to \ell^{\al-\beta-\gamma,2}_j L_t^{\frac1\beta} L_\tau^{\frac1\gamma}
\end{equation}
as long as $0\le \beta<\al<\gamma$, where $\ell^{s,2}$ is a weighted $\ell^2$ space
defined by the norm $\norm{a_j}_{\ell^{s,2}}:=\norm{2^{js}a_j}_{\ell^2}$.

We are now in a position to finish the proof.
For  $\lambda \in (0,\I)\setminus\{1\}$,
by the discrete $J$-functional method, we have
\begin{align*}
	\norm{\varphi}_{(A,B)_{\theta,q}} \sim_\lambda \inf_{\varphi=\sum_j \varphi_j}	
	\norm{\lambda^{j\theta}\norm{\varphi_j}_A + \lambda^{-j(1-\theta)}\norm{\varphi_j}_B }_{\ell^2}.
\end{align*}
Now, we fix $\rho\in(1,\rho_1)$ and chose $\lambda = 2^{-1/\rho}$ and $\varphi_j=K_jf$
to obtain
\begin{align*}
	\LXtn{Kf}{\rho_1}s{q_1} \sim_\rho {}& \norm{Kf}_{(L^\I(\dot{X}^s_{q_1}),
	L^\rho(\dot{X}^s_{q_1}))_{\frac{\rho}{\rho_1},2}} \\
	\le {}& C \norm{2^{-\frac{j}\rho_1}\norm{K_j f}_{L^\I(\dot{X}^s_{q_1})}
	+ 2^{j(\frac1\rho-\frac{1}{\rho_1})}\norm{K_j f}_{L^\rho(\dot{X}^s_{q_1})}}_{\ell^2}
\end{align*}
By means of \eqref{eq:imprvStrLS3} and \eqref{eq:imprvStrLS4}, we conclude that
\begin{align*}
	\LXtn{Kf}{\rho_1}s{q_1} 
	\le {}&C \norm{\mathcal{X}f}_{\ell^{-\delta(q_1)-\frac1{\rho_1},2} L^\I_t
	L^{\frac2{2+\delta(q_1)-\delta(q_2)}}_\tau (\Xt{s}{q_2^\prime}\tau)}\\
	&{}+ C \norm{\mathcal{X}f}_{\ell^{-\delta(q_1)+\frac1\rho-\frac1{\rho_1},2} L^\rho_t
	L^{\frac2{2+\delta(q_1)-\delta(q_2)}}_\tau (\Xt{s}{q_2^\prime}\tau)}\\
	\le{}& C\LXtn{f}{\rho_2^\prime}s{q_2^\prime}
\end{align*}
as long as
\[
	0 < \frac1{\rho_2^\prime} < 1 + \frac{\delta(q_1)}2- \frac{\delta(q_2)}2
\]
is valid. This condition is equivalent to $\rho_1 \delta(q_1)>1$.
\end{proof}
Similarly, we have a inhomogeneous estimate non-admissible for the Lorentz-modified-Sobolev space.
\begin{proposition}\label{prop:imprvStr}
Let $(\rho_i,q_i)$ ($i=1,2$) be two acceptable pairs.
If $s\ge0$ and
\[
	\frac2{\rho_1} - \delta(q_1) + \frac2{\rho_2} - \delta(q_2) =0 
\]
then it holds that
\begin{equation}\label{eq:imprvStr}
	\LMn{\int_0^t U(t-s) f(s) ds}{\rho_1}s{q_1}
	\le C \LMn{f}{\rho_2^\prime}s{q_2^\prime}.
\end{equation}
\end{proposition}
The proof is done in the essentially same way.
We omit details.
\subsection{Specific choice of function spaces}
We introduce several explicit spaces which is used throughout this paper.
The choice depends on whether $p\ge2$ or not.
Recall that we are concerned with the case \eqref{cond:p}.
Hence, $p\ge 2$ occurs only when $N\le 3$.
Similarly, $p<2$ implies $N\ge3$.
We introduce pairs
\begin{equation*}
	(\rho(W_1),q(W_1))=
	\left\{
	\begin{aligned}
	&\(\frac{2(p+1)}{N(p-1)-2}, \frac{2N(p+1)}{2N+4-N(p-1)}\) &p\ge 2,\\
	&\(\frac8{(N-2)(p-1)},\frac{4N}{2N-(N-2)(p-1)}\) & p<2,
	\end{aligned}
	\right.
\end{equation*}
\begin{equation*}
	(\rho(W_2),q(W_2))=
	\left\{
	\begin{aligned}
	&(\rho(W_1),q(W_1)) &p\ge 2,\\
	&\(\frac{4(p-1)}{(N+2)(p-1)-4}, \frac{N(p-1)}{3-p}\)& p<2,
	\end{aligned}
	\right.
\end{equation*}
\begin{equation*}
	(\rho(L),q(L))=
	\left\{
	\begin{aligned}
	&\(\frac{p^2-1}{4-(N-1)(p-1)},\frac{N(p^2-1)}{2N(p-1)-4} \)&p\ge 2,\\
	&\(\frac2{s_c+1},2^*\)& p<2,
	\end{aligned}
	\right.
\end{equation*}
and
\begin{equation*}
	(\rho(F),q(F)) := (\rho(W_1)',q(W_1)')=
	\(\frac{\rho(W_1)}{\rho(W_1)-1}, \frac{q(W_1)}{q(W_1)-1}\).
\end{equation*}
And, for an interval $I \subset \R$, we let
\begin{equation}\label{def:spaces}
\begin{aligned}
	W_1(I):={}& L^{\rho(W_1),2}(I,\Ms{s_c}{q(W_1)}2),\\
	W_2(I):={}& L^{\rho(W_2),2}(I,\Ms{s_c}{q(W_2)}2),\\
	W  (I):={}& W_1(I) \cap W_2(I),\\
	F  (I):={}& L^{\rho(F),2}(I,\Ms{s_c}{q(F)}2),\\
	L  (I):={}& L^{\rho(L),\I}(I,L^{q(L)}),\\
	S(I):={}& L^{\rho(L)}(I,L^{q(L)}).
\end{aligned}
\end{equation}
We omit $(I)$ when it is clear from the context.
When $p\ge2$, $W_1(I)=W_2(I)=W(I)$.
\begin{remark}
The pairs $(\rho(W_1),q(W_1))$ and $(\rho(W_2),q(W_2))$
are admissible.
The pair $(\rho(L),q(L))$ is acceptable, and satisfies
$\frac{2}{\rho(L)}-\delta(q(L))=s_c$.
Due to Lemma \ref{lem:LMn_inclusion} and basic property of Lorentz space,
the embedding
\begin{equation}\label{eq:W2L}
W_2(I)\hookrightarrow S(I) \hookrightarrow L(I)
\end{equation}
holds. 
Further, we have an identity
\begin{equation}\label{eq:expWLF}
	\frac1{\rho(W_1)}+\frac{p-1}{\rho(L)} = \frac1{\rho(F)},\quad
	\frac1{q(W_1)}+\frac{p-1}{q(L)} = \frac1{q(F)}.
\end{equation}
\end{remark}
\begin{remark}\label{rmk:linest}
By means of Strichartz' estimate (Proposition \ref{prop:Strichartz}),
Proposition \ref{prop:LMn_bdd}, and Lemma \ref{lem:embedding1} (3),
we have $\norm{U(t)f}_{L(\R)} + \norm{U(t)f}_{W(\R)} \le C \norm{f}_{\FHsc}$.
\end{remark}

When $p<2$, we use two more function spaces.
Let $s_0=s_0(p,N)$ be a real number given by
\begin{equation}\label{def:s_0}
	s_0=
	\begin{cases}
	\frac3{2(p-1)}+\frac38p-\frac74 &N=3,\,1+\frac2{3+2}<p<2,\\
	\frac2{p-1}+\frac{N(N-2)}8p-\frac{N^2+6N-8}{8} &N\ge4,\,1+\frac4{N+2}<p<1+\frac4N.
	\end{cases}
\end{equation}
Notice that $\frac1{p-1}-\frac{N-2}4 p < s_0 < (p-1)s_c$
as long as \eqref{cond:p} with $N\ge3$ and $p<2$.
We now set
\begin{align*}
	(\rho(X), q(X)) ={}& \( \(\frac{1}{p-1}-\frac{N-2}4-s_0\)^{-1} , \frac{2N}{N-2+2s_0} \) \\
	(\rho(Y), q(Y)) ={}& \( \(\frac{p}{p-1}-\frac{N-2}4p-s_0\)^{-1} , \frac{2N}{(N-2)p+2s_0} \)
\end{align*}
and define two more spaces;
\begin{equation}\label{def:spaces2}
	X(I):= L^{\rho(X),2}(I,\dot{X}^{s_0}_{q(X)}), \quad
	Y(I):= L^{\rho(Y),2}(I,\dot{X}^{s_0}_{q(Y)}).
\end{equation}
Remark that the relation
\begin{equation}\label{eq:expXLY}
	\frac1{\rho(X)}+\frac{p-1}{\rho(L)} = \frac1{\rho(Y)},\quad
	\frac1{q(X)}+\frac{p-1}{q(L)} = \frac1{q(Y)}.
\end{equation}
and the embedding $W_2(I) \hookrightarrow X(I) \hookrightarrow L(I)$ hold. 

\subsection{Estimates on the nonlinearity}
In this subsection, we collect several 
estimate on the nonlinearity.
\begin{lemma}\label{lem:lpt_nonlinearest1}
Let $I\subset \R$ be an interval.
There exists a constant $C$ such that
\begin{equation}\label{eq:lpt_nlest1}
	\norm{|u|^{p-1}u}_{F(I)}
	\le C \norm{u}_{L(I)}^{p-1} \norm{u}_{W_1(I)}
\end{equation}
whenever the right hand side makes sense.
\end{lemma}
\begin{lemma}\label{lem:lpt_nonlinearest2}
Let $I\subset \R$ be an interval.
There exists a constant $C$ such that
\begin{equation}\label{eq:lpt_nlest2}
\begin{aligned}
	\norm{|u_1|^{p-1}u_1-|u_2|^{p-1}u_2}_{F(I)}
	\le{}& C \norm{u_1-u_2}_{W_1(I)}\norm{u_1}_{L(I)}^{p-1} \\&{} +
	C \norm{u_1-u_2}_{L(I)}^{p-1}\norm{u_2}_{W_1(I)}
\end{aligned}
\end{equation}
if $p<2$ and
\begin{multline}\label{eq:lpt_nlest22}
	\norm{|u_1|^{p-1}u_1-|u_2|^{p-1}u_2}_{F(I)}
	\le C \norm{u_1-u_2}_{W_1(I)}\norm{u_1}_{L(I)}^{p-1} \\ +
	C \norm{u_1-u_2}_{L(I)}(\norm{u_1}_{L(I)}+\norm{u_2}_{L(I)})^{p-2}\norm{u_2}_{W_1(I)}
\end{multline}
if $p\ge2$,
whenever the right hand side makes sense.
\end{lemma}
By \eqref{eq:W2L}, the $L(I)$ norms can be replaced by $S(I)$ norm
or $W_2(I)$ norm in the above lemmas.
To prove these lemmas, we introduce difference representation
of Besov norm (see \cite{BL-Book}).
\begin{lemma}\label{lem:B_equivalence}
Let $s >0$ and $q,r \in [1,\I]$.
Let $M$ be an integer larger than $s$.
Let $\delta=\delta_a$ be a difference operator, i.e.\ $\delta f (x) = f(x+a)-f(x)$.
Then, we have the following equivalent expression of the Besov norm:
\[
	\norm{u}_{\dot{B}^s_{q,r}} 
	\sim \norm{2^{js} \sup_{|a| \le 2^{-j}} \norm{\delta^M u}_{L^q} }_{\ell^r(\Z)}.
\]
\end{lemma}

\begin{proof}[Proof of Lemma \ref{lem:lpt_nonlinearest1}]
Put
$g_1(z) := \frac{\d}{\d z}(|z|^{p-1}z)$
and 
$g_2(z) := \frac{\d}{\d \overline{z}}(|z|^{p-1}z)$.
It is obvious that $|g_j(z)|\le C |z|^{p-1}$. 
For $a\in \R^N$, we use the notation $[f](x)=|f(x)|+|f(x+a)|$.

Set $v=M(-t)u$. Notice that $M(-t)|u|^{p-1}u=|v|^{p-1}v$.
By an elementary computation,
\begin{multline}\label{eq:lpt_nonlinearest1}
		\delta_a (|v|^{p-1} v) = (\delta_a v) \int_0^1 g_1(tv(x+a) +(1-t)v(x)) dt\\
	+ \overline{\delta_a v} \int_0^1 g_2(tv(x+a) +(1-t)v(x)) dt.
\end{multline}
Hence, 
\[
	|\delta_a(M(-t)|u|^{p-1}u)| = |\delta_a (|v|^{p-1} v)|\le C |\delta_a v| [v]^{p-1}.
\]
We deduce from the second relation of \eqref{eq:expWLF} that
\[
	\Lebn{\delta_a (M(-t)|u|^{p-1} u)}{q(F)}\le C \Lebn{u}{q(L)}^{p-1}
	\Lebn{\delta_a (M(-t)u)}{q(W_1)}.
\]
Then, taking $\sup_{|a|\le 2^{-j}}$, multiplying by $2^{s_c j}$,
and then taking $\ell^2$-norm in $j$, we see from
Lemma \ref{lem:B_equivalence} that
\[
	\norm{M(-t)|u|^{p-1}u}_{\dot{B}^{s_c}_{q(F),2}}
	\le C \Lebn{u}{q(L)}^{p-1}
	\norm{M(-t)u}_{\dot{B}^{s_c}_{q(W_1),2}}
\]
The lemma now follows from \eqref{eq:Mn_alt} and the
generalized H\"older inequality
in time.
\end{proof}
\begin{proof}[Proof of Lemma \ref{lem:lpt_nonlinearest2}]
We keep the notation $g_1$ and $g_2$.
Let us first consider the case $p<2$.
It is clear that $g_1$ and $g_2$ 
are H\"older continuous of order $p-1$.
Set $v_j=M(-t)u_j$.
We deduce from \eqref{eq:lpt_nonlinearest1} that
\begin{align*}
	&\delta (|v_1|^{p-1} v_1 - |v_2|^{p-1} v_2 ) \\
	={} & \delta (|v_1|^{p-1} v_1) - \delta (|v_2|^{p-1} v_2) \\
	={} & (\delta v_1 - \delta v_2) \int_0^1 g_1(tv_1(x+a) +(1-t)v_1(x)) dt\\
	&{} + \overline{\delta v_1 -\delta v_2 } \int_0^1 g_2(tv_1(x+a) +(1-t)v_1(x)) dt \\
	&{} + \delta v_2 \int_0^1 g_1(tv_1(x+a) +(1-t)v_1(x)) - g_1(tv_2(x+a) +(1-t)v_2(x)) dt \\
	&{} + \overline{\delta v_2 } \int_0^1 g_2(tv_1(x+a) +(1-t)v_1(x))- g_2(tv_2(x+a) +(1-t)v_2(x)) dt.
\end{align*}
Since $g_j$ is H\"older continuous of order $p-1$, we obtain
\[
	|\delta (|v_1|^{p-1} v_1 - |v_2|^{p-1} v_2 )|
	\le C |\delta (v_1-v_2)| [v_1]^{p-1} + C|\delta v_2| [v_1-v_2]^{p-1}.
\]
Arguing as in the previous lemma, we obtain \eqref{eq:lpt_nlest2}.

On the other hand, when $p\ge2$ there exists a constant $C$ such that
\[
	|g_1(z_1)-g_1(z_2)| + |g_2(z_1)-g_2(z_2)|
	\le C |z_1-z_2|(|z_1|+|z_2|)^{p-2}
\]
for any $z_1,z_2\in \C$. Hence, we obtain
\[
	|\delta (|v_1|^{p-1} v_1 - |v_2|^{p-1} v_2 )|
	\le C |\delta (v_1-v_2)| [v_1]^{p-1} + C|\delta v_2| 
	|v_1-v_2|(|v_1|+|v_2|)^{p-2},
\]
from which \eqref{eq:lpt_nlest22} follows.
\end{proof}

The following is the key estimates of the analysis for $p<2$.
\begin{lemma}\label{lem:lpt_nonlinearest3}
Let $I\subset \R$ be an interval.
There exists a constant $C$ such that
\begin{equation}\label{eq:lpt_nlest3}
	\norm{|u_1|^{p-1}u_1-|u_2|^{p-1}u_2}_{Y(I)}
	\le C \norm{u_1-u_2}_{X(I)}\sum_{i=1,2}\norm{u_i}_{W_2(I)}^{p-1}
\end{equation}
whenever the right hand side makes sense.
\end{lemma}
To prove this lemma, we need following two lemmas.
\begin{lemma}\label{lem:lpt_product}
Suppose $1<p,p_1,p_2,p_3,p_4<\I$ satisfy $\frac1p=\frac1p_1+\frac1p_2=\frac1p_3+\frac1p_4$.
Let $s\in (0,1]$. Then,
\[
	\norm{uv}_{\dot{H}^s_p} \le C (\norm{u}_{L^{p_1}}\norm{v}_{\dot{H}^s_{p_2}}
	+\norm{u}_{\dot{H}^s_{p_3}}\norm{v}_{L^{p_4}})
\]
whenever the right hand makes sense.
\end{lemma}
\begin{lemma}[\cite{V}]\label{lem:lpt_fractional}
Let $1<p<2$.
Suppose $g$ is H\"older continuous of order $p-1$.
Let $0<\sigma<p-1$, $1<q<\I$, and $\frac\sigma{p-1}<s<1$.
Then, we have
\[
	\norm{g(u)}_{\dot{H}^\sigma_q} \le C \norm{|u|^{p-1-\frac\sigma{s}}}_{L^{q_1}}
	\norm{u}_{\dot{H}^s_{\frac\sigma{s}q_2}}^{\frac\sigma{s}}
\]
for $\frac1q=\frac1q_1+\frac1q_2$ with $(p-1-\frac\sigma{s})q_1>p-1$
whenever the right hand side makes sense.
\end{lemma}
\begin{proof}[Proof of Lemma \ref{lem:lpt_nonlinearest3}]
As in the proof of \eqref{eq:lpt_nonlinearest1}, it holds that
\begin{multline*}
	|u_1|^{p-1} u_1 - |u_2|^{p-1} u_2 = (u_1-u_2) \int_0^1 g_1(tu_1+(1-t)u_2) dt\\
	+ \overline{(u_1-u_2)} \int_0^1 g_2(tu_1+(1-t)u_2) dt
\end{multline*}
Indeed, one can see this by replacing $u(x+a)$ and $u(x)$ by $u_1$ and $u_2$, respectively.
Hence, by Lemma \ref{lem:lpt_product},
\begin{align*}
	&\norm{|u_1|^{p-1} u_1 - |u_2|^{p-1} u_2}_{\dot{H}^{s_0}_{q(Y)}}\\
	\le {}&C \norm{u_1-u_2}_{\dot{H}^{s_0}_{q(X)}} \sum_{j=1,2}\int_0^1
	\norm{g_j(tu_1+(1-t)u_2)}_{L^{\frac{q(L)}{p-1}}} dt \\
	&{} + C \norm{u_1-u_2}_{L^{q(L)}} \sum_{j=1,2}\int_0^1
	\norm{g_j(tu_1+(1-t)u_2)}_{\dot{H}^{s_0}_{q_0}} dt,
\end{align*}
where $1/q_0=1/q(Y)-1/q(L)$. Lemma \ref{lem:lpt_fractional} then yields
\begin{equation}\label{eq:lpt_nlest3_1}
	\norm{g_j(v)}_{\dot{H}^{s_0}_{q_0}}
	\le C \norm{v}_{L^{q(L)}}^{p-1-\frac{s_0}{s_c}}
	\norm{v}_{\dot{H}^{s_c}_{q(W_2)}}^{\frac{s_0}{s_c}}.
\end{equation}
Applying this inequality and the embedding $\dot{B}^{s_c}_{q(W_2),2} \hookrightarrow
\dot{B}^{s_0}_{q(X),2}\hookrightarrow\dot{H}^{s_0}_{q(X)}\hookrightarrow L^{q(L)}$,
we conclude that
\[
	\norm{|u_1|^{p-1} u_1 - |u_2|^{p-1} u_2}_{\dot{H}^{s_0}_{q(Y)}}
	\le C\norm{u_1-u_2}_{\dot{H}^{s_0}_{q(X)}}
	\sum_{i=1,2} \norm{u_i}_{\dot{B}^{s_c}_{q(W_2),2}}^{p-1}. 
\]
We replace $u_i$ with $M(-t)u_i(t)$ to obtain
\begin{multline*}
	\norm{(|u_1|^{p-1} u_1 - |u_2|^{p-1} u_2)(t)}_{\dot{X}^{s_0}_{q(Y)}(t)}\\
	\le C\norm{(u_1-u_2)(t)}_{\dot{X}^{s_0}_{q(X)}(t)}
	\sum_{i=1,2} \norm{u_i(t)}_{\M{s_c}{q(W_2)}2t}^{p-1}
\end{multline*}
for $t\neq 0$. Thus, the desired estimate follows from the 
generalized H\"older inequality.
\end{proof}

\begin{remark}\label{rmk:lbd_of_p}
The lower bound $p>1+\frac{4}{N+2}$ of our main theorem comes from this lemma.
Indeed, it is necessary that the exponent $s_0$, which denotes the order of weight of $X(I)$ and $Y(I)$,
obeys
\[
	\frac2{p-1} -\frac{Np}2 + 1 < s_0 \le s_c
\]
to choose $\rho(W_1)$, $\rho(W_2)$, $\rho(F)$, $\rho(L)$, $\rho(X)$, $\rho(Y)$,
$q(W_1)$, $q(W_2)$, $q(L)$, $q(F)$, $q(X)$, and $q(Y)$ so that
$(\rho(W_i),q(W_i))$ and $(\rho(F)',q(F)')$ are admissible pairs;
$(\rho(X),q(X))$ and $(\rho(Y)',q(Y)')$ are acceptable pairs 
which fulfill the assumption of Proposition \ref{prop:imprvStr};
and that the relations \eqref{eq:expWLF} and \eqref{eq:expXLY} and
the embedding $W_2(I) \hookrightarrow X(I) \hookrightarrow L(I)$ hold true.
On the other hand,
since $|z|^{p-1}z$ is H\"older continuous of order $p$,
to estimate the left hand side of \eqref{eq:lpt_nlest3},
the exponent $s_0$ must be smaller than $p-1$.
Further, in view of \eqref{eq:lpt_nlest3_1}, we need the relation $s_0< s_c(p-1)$.
These restrictions yield the bound $s_c<1$, which is nothing but $p>1+\frac{4}{N+2}$.
\end{remark}

\subsection{An estimate on Lorentz space}
In this subsection, we prove the following.
\begin{proposition}\label{prop:lpt-division}
Let $I$ be an interval.
Let $1<\rho<\I$ and $1<r<\I$. Suppose $f \in L^{\rho,r}(I)$ and $\norm{f}_{L^{\rho,r}(I)}=M>0$.
For any $\delta>0$, there exists a subdivision $\{I_j\}_{j=1}^k$ of $I$, i.e. 
$I_j=[t_{j-1},t_j]$ with $\inf I =t_0 < t_1 < \dots < t_k = \sup I$, such that
 $k \le 1+C(M/\delta)^{\max(\rho,r)}$ and
\[
	\norm{f}_{L^{\rho,r}(I_j)} \le \delta
\]
holds for all $j\in[1,k]$.
\end{proposition}
For the proof of this proposition, we need the following.
\begin{lemma}\label{lem:lpt-division}
Let $I$ be an interval.
Let $1<\rho<\I$ and $1<r<\I$. Let $f\in L^{\rho,r}(I)$.
Let $\{I_j\}_{j=1}^k$ be a subdivision of $I$.
Then, there exists a constant $C>0$ independent of $k$ and $f$ such that
\[
	\norm{f}_{L^{\rho,r}(I)} \ge C k^{\min(0,\frac{1}{\rho}-\frac{1}{r})}
	\(\sum_{j=1}^k \norm{f}_{L^{\rho,r}(I_j)}^r \)^{\frac1r}
\]
\end{lemma}
\begin{proof}
Take $\eta=\frac12\min(\frac1\rho,\frac1\rho,\frac1r.1-\frac1r)>0$. Define
$\rho_{\pm} \in (1,\I)$ and $r_\pm \in (1,\I)$ by
\[
	\frac1{\rho_\pm} = \frac1\rho \pm \eta, \quad
	\frac1{r_\pm} = \frac1r \pm \eta.
\]
Introduce a linear operator $T$ by
\[
	T: f(x) \mapsto f(x) {\bf 1}_{I_j}(x),
\]	
where ${\bf 1}_{I_j}$ is a characteristic function of $I_j$.
We now claim that $T$ is a bounded mapping from $L^{\rho_+}$ to $\ell_j^{r_+} L^{\rho_+}$
with norm $K:=\max(1,k^{\frac1r-\frac1\rho}) $.
Indeed, by the H\^older inequality,
\[
	\norm{f {\bf 1}_{I_j}}_{\ell^{r_+}L^{\rho_+}}
	\le \norm{1}_{\ell^{(\frac1{r_+}-\frac1{\rho_+})^{-1}}}
	\norm{f {\bf 1}_{I_j}}_{\ell^{\rho_+}L^{\rho_+}}
\]
if $r_+<\rho_+$.
Since $\norm{1}_{\ell^{(\frac1{r_+}-\frac1{\rho_+})^{-1}}}=k^{\frac1r-\frac1\rho}$
and
\[
	\(\sum_{j=1}^k \norm{f{\bf 1}_{I_j}}_{L^{\rho_+}(I)}^{\rho_+} \)^{\frac1{\rho_+}}
	= \norm{f}_{L^{\rho_+}(I)},
\]
we have
$\norm{Tf}_{\ell^{r_+}L^{\rho_+}} \le k^{\frac1r-\frac1\rho} \norm{f}_{L^{\rho_+}(I)}$.
When $r_+\ge \rho_+$, the embedding $\ell^{\rho_+}\hookrightarrow \ell^{r_+}$ gives
$\norm{Tf}_{\ell^{r_+}L^{\rho_+}} \le \norm{f}_{L^{\rho_+}(I)}$. Thus, $T: L^{\rho_+} \to \ell^{r_+}L^{\rho_+}$.
The same argument shows 
$
	T: L^{\rho_-} \to \ell^{r_-}L^{\rho_-}
$
with the same norm $K=\max(1,k^{\frac1r-\frac1\rho})$.
By the real interpolation method, we conclude that
\[
	T: L^{\rho.r} \to (\ell^{r_+}L^{\rho_+},\ell^{r_-}L^{\rho_-})_{\frac12,r}=
	\ell^r L^{\rho,r}
\]
with norm $C\max(1,k^{\frac1r-\frac1\rho})$, 
which yields the desired estimate.
\end{proof}

\begin{proof}[Proof of Proposition \ref{prop:lpt-division}]
Suppose $M>\delta$, otherwise there is nothing to prove.
Take $t_1\in I$  so that $t_1>t_0$ and
$\norm{f}_{L^{\rho,r}(I_1)} = \delta$.
Similarly, as long as $\norm{f}_{L^{\rho,r}(t_{j-1},\sup I)} > \delta$, 
we define $t_j \in I$ so that $t_j>t_{j-1}$ and
$\norm{f}_{L^{\rho,r}(I_j)} = \delta$.
Let us now show that, under this procedure, 
the assumption fails in finite steps,
that is, there exists $k_0$ such that
\begin{equation}\label{eq:lpt-division1}
	\norm{f}_{L^{\rho,r}(t_{k_0-1},\sup I)} \le \delta.
\end{equation}
Indeed, take $k\ge 1$ and suppose that we are able to take $t_j$  so that
$\norm{f}_{L^{\rho,r}(I_j)} = \delta$ for $1\le j \le k$.
Then, one verifies from Lemma \ref{lem:lpt-division} that
\begin{align*}
	M=\norm{f}_{L^{\rho,r}(I)}\ge \norm{f}_{L^{\rho,r}(t_0,t_k)}
	&{}\ge C k^{\min(0,\frac1\rho-\frac1r)}\(\sum_{j=1}^k \norm{f}_{L^{\rho,r}(I_j)}^r \)^{\frac1r}\\
	&{} = C k^{\min(\frac1r,\frac1\rho)} \delta,
\end{align*}
which implies $k$ obeys the estimate
\[
	k \le C (M/\delta)^{\max(\rho,r)}.
\]
Thus, such $k$ is bounded from above and so
there exists $k_0$ satisfying \eqref{eq:lpt-division1}.
Put $t_{k_0}= \sup I$.
\end{proof}

\subsection{Some other estimates}
We finish this section with several useful estimates.
\begin{lemma}
Let $1 \le q<\I$ and $1\le r \le 2$ satisfy $r\le q$.
Let $s\in (0,N/q)$.
Then, for any function $\chi(x) \in \mathcal{S}(\R^N)$,
a multiplication operator $\chi \times$ is bounded on $\dot{B}^s_{q,r}$.
\end{lemma}
\begin{proof}
Take $f \in \dot{B}^s_{q,r}$.
Let $M$ be an integer such that $s<M \le s+1$.
By Lemma \ref{lem:B_equivalence}, 
\[
	\hBn{\chi f}sqr \sim \norm{ 2^{sj} \sup_{|a| \le 2^{-j} } \Lebn{\delta^M_a (\chi f)}q }_{\ell^r},
\]
where $\delta_a f(x)=f(x+a)-f(x)$. One verifies that
\[
	\delta_a^M (\chi f) (x) = 
	\sum_{k=0}^M \binom{M}{k} \delta_a^k \chi (x) \delta_a^{M-k}f(x+ka).
\]
Therefore,
\[
	\Lebn{\delta_a^M (\chi f)}q \le C_M \sum_{k=0}^M \Lebn{\delta_a^{M-k} f}{q_k}
	\Lebn{\delta_a^{k}\chi}{p_k},
\]
where $p_k=\frac{NM}{sk} \in (q,\I]$ and
$q_k^{-1} = q^{-1} - p_k^{-1} \in (0,q^{-1}]$. Since $\chi \in \mathcal{S}$, we have
\[
	\Lebn{\delta_a^{k}\chi}{p_k} \le C_{M,q} \min (|a|^{k} ,1)
\]
for $k\in [0,M]$. One then sees that
\begin{multline*}
	2^{sj}\sup_{|a| \le 2^{-j} } \Lebn{\delta^M_a (\chi f)}q \\
	\le C_{M,q} \sum_{k=0}^M \min (2^{-\frac{k(M-s)}{M}j},2^{\frac{ks}{M}j}) \( 2^{s_kj} \sup_{|a| \le 2^{-j} }\Lebn{\delta_a^{M-k} f}{q_k}\),
\end{multline*}
where $s_k=\frac{M-k}{M}s$.
Since $s>0$ and $M-s>0$ and since $s_k < M-k$, taking $\ell^r$ norm with respect to $j$, we obtain
\[
	\hBn{\chi f}sqr \le C_{s,q} \sum_{k=0}^{M-1} \hBn{f}{ s_k }{q_k}{r}
	+ C_{s,q,r} \Lebn{f}{q_M}.
\]
The embedding 
$$\dot{B}^s_{q,r} \hookrightarrow \dot{B}^{s_1}_{q_1,r} \hookrightarrow \dot{B}^{s_2}_{q_2,r}
\hookrightarrow \cdots \hookrightarrow \dot{B}^{s_{M-1}}_{q_{M-1},r} \hookrightarrow \dot{B}^{0}_{q_{M},r}\hookrightarrow L^{q_M}$$
follows by assumption $r\le 2$ and by definitions of $s_k$ and $q_k$. Thus,
\[
	\hBn{\chi f}sqr \le C \hBn{f}sqr,
\]
which completes the proof.
\end{proof}
\begin{corollary}\label{cor:W_cutoff}
For any $\chi \in \mathcal{S}$, a multiplication operator
$\chi \times$ is bounded on $W_1(I)$.
\end{corollary}
\begin{proof}
By the preceding lemma, for any $t\neq0$,
\begin{align*}
\Mn{\chi f}{s_c}{q(W_1)}2t {}&\sim |t|^{s_c} \hBn{M(-t)\chi f}{s_c}{q(W_1)}2\\
{}&\le C|t|^{s_c} \hBn{M(-t)f}{s_c}{q(W_1)}2 \sim \Mn{f}{s_c}{q(W_1)}2t.
\end{align*}
Take $L^{\rho(W_1),2}$-norm to obtain the desired result.
\end{proof}
The following is an adaptation of \cite[Lemma 3.7]{Ker} (see also \cite[Lemma 2.5]{KV}).
\begin{lemma}\label{lem:cpt_supp_small}
Let  $\mathcal{B}$ be a bounded set of $\R^{1+N}$.
Let $f \in \FHsc$.
Let $1\le \rho,q\le\I$.
Them, for any $\eps>0$, there exists a constant $C_\eps$ such that
\[
	\norm{U(t)|x|^{s_c}f}_{L^2(\mathcal{B})}\le
	\eps \norm{f}_{\FHsc} + C_\eps 
	\norm{U(t) f}_{L^{\rho,\I}(\R_+,L^{p}) }.
\]
\end{lemma}
\begin{proof}
Suppose not. 
Then, there exists a number $\eps_0>0$ and a sequence $\{f_n\}_n \subset \FHsc$
such that
\[
	\norm{U(t)|x|^{s_c}f_n}_{L^2(\mathcal{B})}>
	\eps_0 \norm{f_n}_{\FHsc} + n
	\norm{U(t) f_n}_{L^{\rho,\I}(\R_+,L^{p}) }
\]
Let $g_n:=f_n /\norm{U(t)|x|^{s_c}f_n}_{L^2(\mathcal{B})}$.
Then, $\norm{U(t)|x|^{s_c}g_n}_{L^2(\mathcal{B})}=1$ and
\[
	1>
	\eps_0 \norm{g_n}_{\FHsc} + n
	\norm{g_n}_{L^{\rho,\I}(\R_+,L^{p}) }.
\]
In particular, we obtain
\begin{equation}\label{eq:cptsupp1}
	\norm{g_n}_{\FHsc} \le \eps_0^{-1}
\end{equation}
and
\begin{equation}\label{eq:cptsupp2}
	\norm{g_n}_{L^{\rho,\I}(\R_+,L^{p}) } \to 0
\end{equation}
as $n\to\I$. Now, claim that
\begin{equation}\label{eq:cptsupp3}
	|x|^{s_c}g_n \rightharpoonup 0 \wIN L^2
\end{equation}
as $n\to\I$. Let $\varphi \in \mathcal{S}$ such that $\varphi=0$ on 
$\{|x|\le \delta\}$ for some $\delta>0$. Notice that such functions forms a dense
subset of $L^2$. 
Then, since $\Jbr{|x|^{s_c} g_n, \varphi}_{L^2(\R^N)} = \Jbr{ U(t)g_n, U(t)|x|^{s_c}\varphi}_{L^2(\R^N)}$
for any $t\in \R$,
\[
	\Jbr{|x|^{s_c} g_n, \varphi}_{L^2(\R^N)} = \int_0^1 \Jbr{ U(t)g_n, U(t)|x|^{s_c}\varphi}_{L^2(\R^N)} dt.
\]
Define $\phi(t,x)={\bf 1}_{[0,1]}(t) U(t)|x|^{s_c}\varphi$.
Since $|x|^{s_c} \varphi \in \mathcal{S}$, one sees that
$\phi \in L^{\rho',1}(\R_+,L^{p'})$. Therefore, \eqref{eq:cptsupp2} yields
\begin{align*}
	\abs{\Jbr{|x|^{s_c} g_n, \varphi}_{L^2(\R^N)}}
	&{}= \abs{\Jbr{ U(t)g_n, \phi}_{L^2(\R^{1+N})}} \\
	&{}\le \norm{U(t)g_n}_{L^{\rho,\I}(\R_+,L^q)}
	\norm{\phi}_{L^{\rho',1}(\R_+,L^{p'}} \\
	&{}\le C_\varphi \norm{U(t)g_n}_{L^{\rho,\I}(\R_+,L^q)} \to 0
\end{align*}
as $n\to\I$.
Thus, we obtain \eqref{eq:cptsupp3}.
It is known that $ f \mapsto U(t)f$ is compact from $L^2(\R^N)$
to $L^2(\mathcal{B})$ (see \cite{Sj,Cs,Ve}). Hence, \eqref{eq:cptsupp3} implies
$\norm{U(t)|x|^{s_c}g_n}_{L^2(\mathcal{B})}=0$,
which contradicts with $\norm{U(t)|x|^{s_c}g_n}_{L^2(\mathcal{B})}=1$.
\end{proof}

\section{Local well-posedness and Perturbation arguments}\label{sec:3}
The aim of this section is to establish local well-posedness of \eqref{eq:NLS}-\eqref{eq:IC} in $\FHsc$
and some related results.
Let us first make the notion of a solution clear.
\begin{definition}\label{def:sol}
Let $I \subset \R$ be a nonempty time interval and let $t_0 \in I$.
We say a function $u$: $I\times \R^N \to \C$ is an $\FHsc$-solution to \eqref{eq:NLS}
with data $u_{t_0} \in \dot{X}^{s_c}_2(t_0)$ at $t=t_0$
if $u(t)$ and $U(-t)u(t)$ lie in the class $W(K)$ and
$C_t^0 (K, \FHsc)$, respectively, for all compact $K \subset I$,
and $u$ obeys the Duhamel formula
\begin{equation}\label{eq:INLS}
	U(-t_1)u(t_1) = U(-t_0)u_{t_0} + i \int_{t_0}^{t_1} U(-s) (|u|^{p-1} u)(s) ds
\end{equation}
in the $\FHsc$ sense for all $t_1 \in I$.
We refer to the interval $I$ as the \emph{lifespan} of $u$.
We say that $u$ is a global solution if $I=\R$.
We say a solution $u$ of \eqref{eq:NLS} satisfies \eqref{eq:IC} if $I$ contains $0$ and $u(0)=u_0$.
\end{definition}
\begin{remark}
The second term of the right hand of \eqref{eq:INLS} makes sense.
Indeed, if $u\in W(K)$ then $|u|^{p-1}u \in F(K)$ by Lemma \ref{lem:lpt_nonlinearest1} and embedding \eqref{eq:W2L}.
Then, Strichartz' estimate implies that the term belongs to $C_t^0 (K,\FHsc)$.
The equation \eqref{eq:INLS} equivalent to
\begin{equation}\tag{\ref{eq:INLS}${}'$}
	u(t) = U(t-t_0)u_{t_0} + i \int_{t_0}^{t} U(t-s) (|u|^{p-1} u)(s) ds
\end{equation}
holds in the $\dot{X}^{s_c}_2(t)$ sense or in the $\M{s_c}22{t}$ sense for $t\in I$.
\end{remark}
\subsection{Local existence and uniqueness}
In this subsection, we show a local existence theorem under an additional condition $u_0\in L^2$.
This assumption will be removed after we establish a short-time perturbation.
\begin{theorem}[\cite{NO}]\label{thm:lpt_existence}
Let $t_0\in I \subset \R$. Suppose $u_0 \in  \M{s_c}22{t_0}$ satisfies
\[
	\norm{U(t-t_0)u_0}_{W(I)} \le \eta
\]
for some $0<\eta\le\eta_0$, where $\eta_0$ is a positive constant.
Further, assume that $u_0 \in L^2$.
Then, there exists a unique solution $u(t)\in L^\I(I,\Ms{s_c}22)\cap W(I)$ of \eqref{eq:NLS}
with $u(t_0)=u_{0}$ and
\begin{align}
	\norm{u}_{W(I)} &{}\le C \eta,\label{eq:lpt_existence1}\\
	\norm{u}_{L^\I(I,\Ms{s_c}22)} &{}\le C (\norm{u_{0}}_{\M{s_c}22{t_0}}
	+\eta^p),\label{eq:lpt_existence2}\\
	\norm{|u|^{p-1}u}_{F(I)} &{}\le C \eta^{p},\label{eq:lpt_existence3}\\
	\norm{u}_{L^\I(I,L^2)\cap L^{\rho,2}(I,L^q)} &{}\le C_q \norm{u_0}_{L^2},\label{eq:lpt_existence4}
\end{align}
where $(\rho,q)$ is an admissible pair.
\end{theorem}
\begin{remark}
For any $u_0 \in \M{s_c}22{t_0}$, there exists an interval $I\ni t_0$ in which 
a corresponding solution $u(t)$ exists.
Indeed,
thanks to Strichartz' estimate, we have
$\norm{U(t-t_0)u_0}_{W(\R)} <+\I$ for any $u_0 \in \M{s_c}22{t_0}$.
Hence, we are always able to choose an interval $I \ni t_0$ so that 
$\norm{U(t-t_0)u_0}_{W(I)} \le \eta_0$.
Notice that the interval $I$ depends on the profile of the data, not simply on the size of the data.
\end{remark}
\begin{proof}
We shall put $u^{(0)}(t):=0$, $u^{(1)}(t):=U(t-t_0)u_0$, and
\[
	u^{(m+1)}(t):=U(t-t_0)u_0 +i\int^t_{t_0} U(t-s)(|u^{(m)}|^{p-1}u^{(m)})(s) ds
\]
for $m\ge1$.
Let us first establish a uniform a priori bound on $u^{(m)}$.
Strichartz'  estimates (Proposition \ref{prop:Strichartz}),
Lemma \ref{lem:lpt_nonlinearest1}, and \eqref{eq:W2L}
give us
\begin{align*}
	\norm{u^{(m+1)}}_{W(I)} \le{}&  \norm{U(t-t_0)u_0}_{W(I)}
	+ C\norm{|u|^{p-1}u}_{F(I)} \\
	\le{}&  \eta + C\norm{u^{(m)}}^p_{W(I)}.
\end{align*}
Hence, if $\eta_0$ is sufficiently small, then we can show by induction that
\begin{equation}\label{eq:exist1}
	\norm{u^{(m)}}_{W(I)} \le 2\eta,\quad
	\norm{|u^{(m)}|^{p-1}u^{(m)}}_{F(I)} \le C \eta^p.
\end{equation}
Again by Strichartz' estimate,
\begin{equation}\label{eq:exist2}
	\norm{u^{(m)}}_{L^\I(I,\Ms{s_c}22)} \le C\norm{u_0}_{\M{s_c}22{t_0}} + C \eta^p.
\end{equation}

We next estimate the difference $u^{(m+1)}-u^{(m)}$. We set
\[
	D^0(I) := L^{\I}(I,L^{2}) \cap L^{\rho(W_2),2}(I,L^{q(W_2)}), \quad
	F^0(I) := L^{\rho(F),2}(I,L^{q(F)}).
\]
By Strichartz' estimate and \eqref{eq:exist1},
\begin{equation}\label{eq:exist3}
\begin{aligned}
	\norm{u^{(m+1)}-u^{(m)}}_{D^0(I)}
	&{}\le C \norm{|u^{(m+1)}|^{p-1}u^{(m+1)}-|u^{(m)}|^{p-1}u^{(m)}
	}_{F^0(I)}\\
	&{}\le C\norm{u^{(m)}-u^{(m-1)}}_{D^0(I)}\sum_{i=0,1} \norm{u^{(m-i)}}_{L(I)}^{p-1}\\
	&{}\le C \eta^{p-1} \norm{u^{(m)}-u^{(m-1)}}_{D^0(I)}.
\end{aligned}
\end{equation}
Further, by the assumption $u_0 \in L^2$, we have
$\norm{u^{(1)}}_{D^0(I)} \le C \Lebn{u_0}2$.
One sees from \eqref{eq:exist3} that if $\eta_0$ is small then $\{u^{(m)}\}_m$ is a Cauchy sequence in $D^0(I)$.
Thus, there exists a function $u \in D^0(I)$ such that
\[
	u^{(m)} \to u \IN D^0(I)
\]
as $m\to\I$. The estimate \eqref{eq:exist3} also imply
\[
	|u^{(m)}|^{p-1} u^{(m)} \to |u|^{p-1}u \IN F^0(I)
\]
as $m\to\I$. It therefore holds that
\begin{multline*}
	\norm{ u(t)-U(t-t_0)u_0 - i \int_{t_0}^t U(t-s) (|u|^{p-1}u)(s)ds }_{D^0(I)}\\
	\le \norm{u-u^{(m)}}_{D^0(I)} 
	+ C \norm{|u|^{p-1}u-|u^{(m-1)}|^{p-1}u^{(m-1)}}_{F^0(I)}\to0
\end{multline*}
as $m\to0$, which shows that $u(t)$ is a solution of \eqref{eq:NLS}.
By a priori estimates \eqref{eq:exist1} and \eqref{eq:exist2}, one has desired bounds on $u(t)$.
Uniqueness follows from a standard argument.
\end{proof}

\subsection{Short-time perturbation}
We next establish  a stability estimate what is called a short-time perturbation. 
This is one of the key estimate of our argument.
The local well-posedness, shown in the forthcoming subsection, relies on this estimate.
The result depends on whether $p\ge2$ or not.
When $p\ge2$ we are able to obtain Lipschitz continuity.
However, it becomes merely H\"older if $p<2$.
We follow an argument by \cite{TV}, in which the energy critical equation is considered.
In \cite{TV}, in order to establish a stability estimate for $p<2$,
they introduce an "intermediate space" and
an exotic Strichartz estimate on it.
In our case, $X(I)$ is the intermediate space.
Heart of the matter is that we take not a Lorentz-modified-Besov type space
but a Lorentz-modified-Sobolev space.
Then, the non-admissible Strichartz estimate (Proposition \ref{prop:imprvStr})
plays the role of the exotic Strichartz estimate.
This stability estimate will be upgraded in Theorem \ref{thm:lpt}.
However, the main point lies in this proposition.
\begin{proposition}\label{prop:spt}
Let $I$ be a compact interval contains $t_0$.
Let $\widetilde{u}(t)$ is an approximate solution of \eqref{eq:NLS} in
such a sense that
\[
	i\d_t \widetilde{u} + \Delta\widetilde{u} = -|\widetilde{u}|^{p-1}\widetilde{u}
	+e
\]
holds (in the sense of corresponding integral equation) with some function $e \in F(I)$.
Suppose that 
\[
	\norm{\widetilde{u}}_{L^\I (I,\Ms{s_c}22)} \le A
\]
for some $A>0$. Let $u(t_0)\in X^{s_c}_2(t_0)$ satisfy
\[
	\norm{u(t_0)-\widetilde{u}(t_0)}_{\M{s_c}22{t_0}} \le A'
\]
for some $A'>0$. Then, there exists a positive constant
$\eps_0>0$ such that if 
\[
	\norm{\widetilde{u}}_{W(I)}  \le \eps_0
\]
and if 
\begin{align*}
	\norm{U(t-t_0)(u(t_0)-\widetilde{u}(t_0))}_{W(I)} &{}\le \eps, &
	\norm{e}_{F(I)} &{} \le \eps
\end{align*}
for some $0<\eps \le \eps_0$ then a solution $u(t)$ of \eqref{eq:NLS}
with initial data $u(t_0)$ at $t=t_0$ obeys 
\begin{align*}
	\norm{u-\widetilde{u}}_{W(I)} &{}\le C \eps^{\min(1,p-1)} \\
	\norm{u-\widetilde{u}}_{L^\I(I,\Ms{s_c}22)} &{}\le A'+ C\eps^{\min(1,p-1)} \\
	\norm{u}_{L^\I(I,\Ms{s_c}22)} &{}\le A+ A' + C\eps^{\min(1,p-1)} \\
	\norm{(i \d_t +\Delta)(u-\widetilde{u}) + e}_{F(I)} &{}\le C\eps^{\min(1,p-1)}.
\end{align*}
\end{proposition}
\begin{proof}
Put $w:=u-\widetilde{u}$. Then, $w$ satisfies
\begin{align*}
	w(t) ={}&U(t-t_0)w(t_0) + i \int_{t_0}^t U(t-s) (|u|^{p-1}u-|\widetilde{u}|^{p-1}\widetilde{u})
	(s)ds \\
	&{}- \int_{t_0}^t U(t-s)e(s)ds.
\end{align*}

We first consider the case $p\ge2$.
By the Strichartz estimate,
\eqref{eq:lpt_nlest22}, \eqref{eq:W2L}, and assumption,
one sees that
\begin{align*}
	\norm{w}_{W(I)} \le{}& C \norm{U(t-t_0)w(t_0)}_{W(I)} + C\norm{|u|^{p-1}u-|\widetilde{u}|^{p-1}\widetilde{u}}_{F(I)} + \norm{e}_{F(I)}\\
	\le{}& C\eps + C \norm{w}_{W(I)} (\norm{w+\widetilde{u}}_{W(I)} + \norm{\widetilde{u}}_{W(I)})^{p-1} \\
	\le{}& C\eps + C\eps_0^{p-1} \norm{w}_{W(I)} + C \norm{w}_{W(I)}^p 
\end{align*}
If $\eps_0$ is small, we obtain
\[
	\norm{w}_{W(I)} \le C \eps + C \norm{w}_{W(I)}^p,
\]
which implies $\norm{w}_{W(I)} \le C \eps$ for $\eps \le \eps_0$
if $\eps_0$ is small.
Then, another use of Strichartz' estimate yields
\[
	\norm{w}_{L^\I(I, \Ms{s_c}22)} \le A' + C\eps_0^{p-1} \norm{w}_{W(I)}+ C \norm{w}_{W(I)}^p
	\le A'+ C \eps.
\]
We finally note that
\begin{align*}
\norm{(i \d_t +\Delta)(u-\widetilde{u}) + e}_{F(I)} 
&{}=\norm{|u|^{p-1}u - |\widetilde{u}|^{p-1} \widetilde{u} }_{F(I)} \\
&{}\le  C \norm{w}_{W(I)} (\norm{u}_{W(I)} + \norm{\widetilde{u}}_{W(I)})^{p-1}\\
&{}\le C\eps,
\end{align*}
which completes the proof.

We next consider the case $p<2$.
By the embedding $W_2(I)\hookrightarrow X(I)\hookrightarrow L(I)$,
Strichartz' estimate (Proposition \ref{prop:Strichartz}),
the non-admissible Strichartz estimate (Proposition \ref{prop:imprvStrLS}),
Lemma \ref{lem:lpt_nonlinearest3}, and the assumption,
one sees that
\begin{align}
	\norm{w}_{X(I)} \le{}& C \norm{U(t-t_0)w(t_0)}_{W(I)} + C\norm{|u|^{p-1}u-|\widetilde{u}|^{p-1}\widetilde{u}}_{Y(I)} + \norm{e}_{F(I)}\nonumber\\
	\le{}& C\eps + C \norm{w}_{X(I)} (\norm{u}_{W(I)} + \norm{\widetilde{u}}_{W(I)})^{p-1}.
	\label{eq:spt2_1}
\end{align}

Let us now estimate $\norm{u}_{W(I)}$.
By the assumption, we have
\begin{align*}
	\norm{U(t-t_0)u(t_0)}_{W(I)}
	&{}\le \norm{U(t-t_0)w(t_0)}_{W(I)} + \norm{U(t-t_0)\widetilde{u}(t_0)}_{W(I)} \\
	&{}\le \eps + \norm{\widetilde{u}}_{W(I)}
	+ C\norm{|\widetilde{u}|^{p-1}\widetilde{u}}_{F(I)} + C\norm{e}_{F(I)}\\
	&{}\le C\eps_0 + C \eps_0^p.
\end{align*}
If $\eps_0$ is sufficiently smaller than $\eta_0$, which is a constant given in Theorem \ref{thm:lpt_existence}, then the life span of $u$ contains $I$ and we have
\begin{equation}\label{eq:spt2-2}
	\norm{u}_{W(I)} \le C\eps_0.
\end{equation} 
Substituting this estimate to \eqref{eq:spt2_1}, we obtain
$\norm{w}_{X(I)} \le C \eps + C\eps_0^{p-1} \norm{w}_{X(I)}$,
which implies 
\begin{equation}\label{eq:spt2-3}
	\norm{w}_{X(I)} \le C \eps
\end{equation} if $\eps_0$ is small.

Again by Strichartz' estimate,
\[
	\norm{w}_{W(I)} \le \norm{U(t-t_0)w(t_0)}_{W(I)}
	+ C\norm{|u|^{p-1}u-\widetilde{u}^{p-1}\widetilde{u}}_{F(I)} + C\norm{e}_{F(I)}.
\]
We now use \eqref{eq:lpt_nlest2} to the second term of the right hand side to yield
\begin{equation}\label{eq:spt2-4}
	\norm{u|^{p-1}u-\widetilde{u}^{p-1}\widetilde{u}}_{F(I)}
	\le C\norm{w}_{W(I)} \norm{u}_{L(I)}^{p-1}
	+ C\norm{w}_{L(I)}^{p-1} \norm{\widetilde{u}}_{W(I)}.
\end{equation}
Then, by $W_2(I)\hookrightarrow X(I)\hookrightarrow L(I)$ and by \eqref{eq:spt2-2}
and \eqref{eq:spt2-3},
\begin{align*}
	\norm{w}_{W(I)} &{}\le \eps + C\norm{w}_{W(I)}
	\norm{u}_{W(I)}^{p-1}
	+ C \norm{w}_{X(I)}^{p-1} \norm{\widetilde{u}}_{W(I)} \\
	&{} \le \eps  + C\eps_0^{p-1}\norm{w}_{W(I)} + C\eps_0\eps^{p-1}.
\end{align*}
Hence, if $\eps_0$ is small, then we have
\[
	\norm{w}_{W(I)} \le2\eps + 2C\eps_0 \eps^{p-1} \le C\eps^{p-1}\eps_0^{2-p}.
\]
The other estimates follow as in the $p\ge2$ case.
\end{proof}

\begin{remark}\label{rmk:lbd_of_p2}
The lower bound $p>1+\frac{4}{N+2}$ comes from this proposition.
The point of the above argument for $p<2$
is that we first establish 
 \eqref{eq:spt2-3}
to deal with $\norm{w}_{L(I)}^{p-1}$ in the right hand side of \eqref{eq:spt2-4}.
To obtain \eqref{eq:spt2-3}, it is essential to employ Lemma \ref{lem:lpt_nonlinearest3} which
requires the lower bound $p>1+\frac{4}{N+2}$ (see Remark \ref{rmk:lbd_of_p}).
\end{remark}

\subsection{Local well-posedness and blowup criterion}
We are now in a position to establish local well-posedness in $\FHsc$.
\begin{theorem}\label{thm:lwp}
The Cauchy problem \eqref{eq:NLS}-\eqref{eq:IC} is localy well-posed in $\FHsc$.
Further, all the statement of Theorem \ref{thm:lpt_existence},
except for \eqref{eq:lpt_existence4}, 
hold without the assumption $u_0 \in L^2$.
\end{theorem}
\begin{proof}
Let us first extend the existence result to the case
$u_0 \in \M{s_c}22{t_0}\setminus L^2$.
Let $\eta_0$ be the constant given in 
Theorem \ref{thm:lpt_existence}, and assume
\[
	\norm{U(t-t_0)u_0}_{W(I)} \le \eta
\]
for some $\eta\le \frac12 \eta_0$.
Take a sequence of functions $\{u_{0,n}\}_{n=1}^\I \subset \M{s_c}22{t_0}\cap L^2$
such that $u_{0,n}\to u_0$ in $\M{s_c}22{t_0}$ as $n\to\I$.
Since
\[
	\norm{U(t-t_0)(u_0-u_{0,n})}_{W(I)} \le
	C \norm{u_0-u_{0,n}}_{\M{s_c}22{t_0}} \to 0
\]
as $n\to\I$,
there exists $N=N(\eta)>0$ such that
\[
	\norm{U(t-t_0)u_{0,n}}_{W(I)} \le \norm{U(t-t_0)(u_0-u_{0,n})}_{W(I)} + \norm{U(t-t_0)u_0}_{W(I)} \le 2\eta
\]
for all $n\ge N$. One then deduces from Theorem \ref{thm:lpt_existence} that
there exists a sequence of solutions $\{u_n\}_{n=N}^\I$ with 
the bounds
\begin{align*}
	\norm{u_n}_{W(I)} &{}\le C \eta,\\
	\norm{u_n}_{L^\I(I,\Ms{s_c}22)} &{}\le C (\norm{u_{0,n}}_{\M{s_c}22{t_0}}+\eta^p),\\
	\norm{|u_n|^{p-1}u_n}_{F(I)} &{}\le C \eta^{p}
\end{align*}
(uniformly in $n$).
Then, since $\{u_{0,n}\}_{n=N}^\I$ is a Cauchy sequence in $\M{s_c}22{t_0}$,
by means of Proposition \ref{prop:spt}, $\{u_n\}_{n=N}^\I$ is a Cauchy sequence
in such a sense that
\[
	\norm{u_n-u_m}_{W(I)} \to 0, \qquad
	\norm{|u_n|^{p-1}u_n-|u_m|^{p-1}u_m}_{F(I)} \to 0
\]
as $n,m\to\I$. Hence, there exists a unique limit $u \in W(I)$ such that
\eqref{eq:lpt_existence1}--\eqref{eq:lpt_existence3} and
\[
	\norm{u-u_n}_{W(I)} \to 0, \qquad
	\norm{|u|^{p-1}u-|u_n|^{p-1}u_n}_{F(I)} \to 0
\]
as $n\to\I$. Further, $u$ solves \eqref{eq:NLS}.
Uniqueness and
continuous dependence properties are also immediate consequences of 
Proposition \ref{prop:spt}.
\end{proof}
One also obtains a standard blowup criterion.
It is worth mentioning that that if $u_0 \in L^2$ then blowup never occurs due to
the mass conservation.
\begin{proposition}[Blowup criterion]\label{prop:lpt_criterion}
Let $u_0 \in \M{s_c}22{t_0}$ and let $u$ be a solution to \eqref{eq:NLS}
with initial data $u(t_0)=u_0$.
Let $I_{\mathrm{max}}=(T_-,T_+)$ be the maximal interval.
If $T_+<+\I$ then 
$\lim_{\tau\uparrow T_+} \norm{u}_{S([t_0,\tau])} =+\I$.
\end{proposition}
\begin{proof}
Assume $T_+<\I$ and $\norm{u}_{S([t_0,T_+))} <+\I$ for contradiction.
Then, for any $\delta>0$ there exists $t_1 \in (t_0,T_+)$ such that
$\norm{u}_{S([t_1,T_+))}\le \delta$. 
Then, for $\tau \in [t_1,T_+)$,
\[
	\norm{u}_{W_2([t_1,\tau])} \le C\norm{u(t_1)}_{X^{s_c}_2(t_1)}
	+C\norm{u}_{S([t_1,T_+))}^{p-1} \norm{u}_{W_2([t_1,\tau])}.
\]
We now take $\delta$ (and $t_1$) so that $C\delta^{p-1}\le 1/2$.
Then, we obtain $\norm{u}_{W_2([t_1,\tau])} \le 2C\norm{u(t_1)}_{X^{s_c}_2(t_1)}<+\I$
for any $\tau \in [t_1,T_+)$, showing $\norm{u}_{W_2([t_1,T_+))}<\I$.
Strichartz' estimate then gives us 
$\norm{u}_{L^\I([t_0,T_+),\Ms{s_c}22)}+\norm{u}_{W([t_0,T_+))}<+\I$.
Another use of Strichartz' estimate yields
\[
	\norm{U(t-t_2)u(t_2)}_{W([t_2,T_+))}
	\le \norm{u}_{W([t_2,T_+))} + C \norm{u}_{W([t_2,T_+))}^p
\]
for $t_2<T_+$.
Since $\norm{u}_{W([t_0,T_+))}<\I$, we can take $t_2<T_+$ so that
\[
	\norm{U(t-t_2)u(t_2)}_{W([t_2,T_+))} \le \frac{\eta_0}2,
\]
where $\eta_0$ is given in Theorem \ref{thm:lpt_existence}.
One can take $t_3>T_+$ so that $$\norm{U(t-t_2)u(t_2)}_{W([t_2,t_3])} \le \eta_0.$$
By means of Theorem \ref{thm:lpt_existence},
this implies $u$ is extended beyond $t=T_+$, which is a contradiction.
\end{proof}
\subsection{Scattering criterion}
In this subsection, we give a necessary and sufficient condition for scattering.
It is well-known that a space-time bound of a solution is a sufficient condition for scattering. 
Here, we will prove that this space-time bound is also a necessary condition.
We restrict ourselves to nonnegative time only, but
it is obvious that a similar result holds true for negative time.
\begin{proposition}\label{prop:nscond}
Let $u_0 \in \FHsc$ and let $u(t)$ be a corresponding unique solution to
\eqref{eq:NLS} given in Theorem \ref{thm:lwp}.
Let $(T_{-},T_{+})\ni 0$ be  the maximal interval of $u(t)$.
Then, the following three conditions are equivalent.
\begin{enumerate}
\item $u_0 \in S_+$;
\item $T_{+} =+\I$ and $\norm{u}_{W(\R_+)}<\I$;
\item $T_{+} =+\I$ and $\norm{u}_{S(\R_+)}<\I$.
\end{enumerate}
\end{proposition}
\begin{proof}
By embedding $W_2(\R_+)\hookrightarrow S(\R_+)$,
``(2)$\Rightarrow$(3)'' is obvious. 
The relation ``(2)$\Rightarrow$(1)'' is standard:
Using Strichartz' estimate, one sees that $\norm{U(-t_1)u(t_2)-U(-t_2)u(t_2)}_{\FHsc}$
tends to zero as $t_1,t_2\to\I$ (see \cite{CazBook} for details). 

We shall prove ``(1)$\Rightarrow$(2)''.
By definition of $S_+$, we have $T_{+} =+\I$.
Then, it is easy to see that $\norm{u}_{W((0,\tau))}<+\I$ for any $\tau>0$.
Set $u_+ := \lim_{t\to\I} U(-t)u(t)$.
Then, $u_+\in \FHsc$.
Since $\norm{U(t)u_+}_{W(\R)}<\I$ by Strichartz' estimate,
there exists $T>0$ such that
\[
	\norm{U(t)u_+}_{W([T,\I))} \le \frac{\eta_0}2,
\]
where $\eta_0$ is the number given in Theorem \ref{thm:lwp} (or Theorem \ref{thm:lpt_existence}).
Further, by Strichartz' estimate and the convergence $U(-t)u(t)\to u_+$ as $t\to\I$,
 we have
\[
	\norm{U(t)u_+ - U(t-t_0)u(t_0)}_{W(\R)}
	\le C \norm{u_+ - U(-t_0)u(t_0)}_{\FHsc} 
	\le \frac{\eta_0}2
\]
for sufficiently large $t_0 \ge T$. Fix such $T$ and $t_0$.
Then, one sees that
\begin{multline*}
	\norm{U(t-t_0)u(t_0)}_{W([T,\I))} \\ \le \norm{U(t)u_+}_{W([T,\I))} + \norm{U(t)u_+-U(t-t_0)u(t_0)}_{W([T,\I))}
	\le \eta_0.
\end{multline*}
By Theorem \ref{thm:lwp}, there exists a solution $\widetilde{u}(t)$ of \eqref{eq:NLS} such that
$\widetilde{u}(t_0)=u(t_0)$ and $\norm{\widetilde{u}}_{W([T,\I))}\le C\eta_0$.
By uniqueness, $\widetilde{u}=u$. Therefore, $\norm{u}_{W([T,\I))} \le C\eta_0$,
which yields $\norm{u}_{W(\R_+)}<\I$.

Let us show ``(3)$\Rightarrow$(2)''.
Since $\norm{u}_{S(\R_+)} <\I$,
for any $\eps>0$, there exists $T_0$ such that $\norm{u}_{S([T_0,\I))} \le\eps$.
Now, $u(t)$ satisfies
\[
	u(t) =U(t-T_0) u(T_0) +i \int_{T_0}^t U(t-s)(|u|^{p-1}u(s)) ds.
\]
One sees from Strichartz' estimate that
\begin{equation}\label{eq:nscond}
	\norm{u}_{W([T_0,T])}
	\le C_1 \Mn{u(T_0)}{s_c}{2}2{T_0} + C_2 \norm{u}_{S([T_0,T])}^{p-1}
	\norm{u}_{W([T_0,T])}
\end{equation}
for any $T>T_0$. 
Choose $\eps$ so small that $C_2 \eps^{p-1}\le 1/2$ to yield
$\norm{u}_{W([T_0,T])} \le 2C_1 \Mn{u(T_0)}{s_c}{2}2{T_0}$
uniformly in $T>T_0$. Letting $T\to\I$, we obtain $\norm{u}_{W(\R_+)}<\I$.
\end{proof}
\begin{remark}\label{rmk:nscondFH1}
The above criterion is valid also for $\F H^1$-solutions.
More precisely, if $u_0 \in \F H^1$ then
a corresponding $\F H^1$-solution $u(t)$ of \eqref{eq:NLS}-\eqref{eq:IC},
scatters (in $\F H^1$) for positive time if and only if $\norm{u}_{S(\R_+)} <\I$
(or the other properties in the proposition).
The outline of proof is as follows.
If a solution scatters in $\F H^1$ then it does so also in the $\FHsc$ sense.
Hence the first property of Proposition \ref{prop:nscond} is satisfied.
On the other hand, if the third property of Proposition \ref{prop:nscond} are satisfied,
then a persistence-of-regularity type argument shows that it scatters in $\F H^1$ sense.
The argument is very similar to that in the above proof. 
\end{remark}

A consequence of Proposition \ref{prop:nscond} is the
following small data scattering.
\begin{lemma}[Small data scattering]\label{lem:sds}
Let $u_0 \in \FHsc$ and let $u(t)$ be a corresponding unique solution given
in Theorem \ref{thm:lwp}. Then, we have the following.
\begin{enumerate}
\item There exists $\eta_1>0$ such that if
$\norm{U(t)u_0}_{W(\R)} \le \eta_1$ then $u_0 \in S$.
\item There exists $\eta_2>0$ such that if
$\norm{u_0}_{\FHsc} \le \eta_2$ then $u_0 \in S$.
\end{enumerate}
\end{lemma}
This lemma is proven in \cite{NO} with an additional assumption $u_0 \in L^2$.
Another version is given later (Proposition \ref{prop:sds2})
by using a long-time perturbation, Theorem \ref{thm:lpt}.
\begin{proof}
If $\eta_1$ is sufficiently small, it follows from \eqref{eq:lpt_existence1} that
$\norm{u}_{W(\R)} \le C\eta_1$. Hence, we obtain the former part by Proposition \ref{prop:nscond}.
The latter statement is an immediate consequence of the former and Strichartz' estimate.
\end{proof}
\begin{remark}\label{rmk:lc_low}
The second property of Lemma \ref{lem:sds} implies that
the infimum $\ell_c$ given in \eqref{def:lc} is greater than or equal to $\eta_2>0$.
\end{remark}
\subsection{Long time perturbation}

The following are our second stability estimate. 

\begin{theorem}\label{thm:lpt}
Let $I$ be an interval contains $t_0$.
Let $\widetilde{u}(t)$ is an approximate solution of \eqref{eq:NLS} on $I\times \R^N$ in
such a sense that
\[
	i\d_t \widetilde{u} + \Delta\widetilde{u} = -|\widetilde{u}|^{p-1}\widetilde{u}
	+e
\]
holds (in the sense of corresponding integral equation) with some function $e\in F(I)$.
Suppose that 
\[
	\norm{\widetilde{u}}_{L^\I (I,\Ms{s_c}22)} \le A, \qquad
	\norm{\widetilde{u}}_{W(I)} \le M,
\]
for some $A,M>0$. Let $u(t_0)\in \M{s_c}22{t_0}$ satisfy
\[
	\norm{u(t_0)-\widetilde{u}(t_0)}_{\M{s_c}22{t_0}} \le A'
\]
for some $A'>0$. Then, there exists a positive constant
$\eps_1=\eps_1(A',M)$ such that if 
\begin{align*}
	\norm{U(t-t_0)(u(t_0)-\widetilde{u}(t_0))}_{W(I)} &{}\le \eps, &
	\norm{e}_{F(I)} &{} \le \eps
\end{align*}
for some $0<\eps \le \eps_1$ then the solution $u(t)$ of \eqref{eq:NLS}
with initial data $u(t_0)$ at $t=t_0$ exists in $I$ and 
there exist $\beta = \beta(A',M)\in (0,1]$ and $C=C(A',M)>0$
such that  
\begin{align*}
	\norm{u-\widetilde{u}}_{W(I)} &{}\le C\eps^\beta, &
	\norm{u-\widetilde{u}}_{L^\I(I,\Ms{s_c}22)} &{}\le C
\end{align*}
hold. Further $\beta=1$ if $p\ge2$.
\end{theorem}
\begin{proof}
By the time symmetry, we may assume $t_0=\inf I$ without loss of generality.
Take  the constant $\eps_0$ given in Proposition \ref{prop:spt}.
Replacing $\eps_0$ with a smaller one, if necessary,
we may suppose that $\eps_0\le A'$.
By Proposition \ref{prop:lpt-division},
there exists an integer $N_0=N_0(M,\eps_0)=N_0(A',M)>0$ such that
$I=\bigcup_{j=1}^{N_0} I_j $, $I_j=[t_{j-1},t_j]$ with
\[
	\norm{\widetilde{u}}_{W(I_j)} \le \eps_0.
\] 
We keep the notation $w=u-\widetilde{u}$.
Put 
\[
	\kappa_j:= \norm{|u|^{p-1}u-|\widetilde{u}|^{p-1}\widetilde{u}}_{F(I_j)}.
\]

We see from Proposition \ref{prop:spt} that there exists a positive constant $C_0$
such that if positive constant $\eta_j$ satisfies the
relation $\eta_j \le \eps_0$ and if
\begin{equation}
	\norm{U(t-t_{j-1})w(t_{j-1})}_{W(I_j)} \le \eta_j, \label{eq:lpt2_1}
\end{equation}
holds then we have
\begin{equation}\label{eq:lpt2_2}
	\norm{w}_{W(I_j)}+ \kappa_j \le C_0 \eta_j^{\min(1,p-1)}.
\end{equation}
and
\begin{equation}\label{eq:lpt2_3}
	\norm{w}_{L^\I(I_j,\Ms{s_c}22)} \le \norm{w(t_{j-1})}_{\M{s_c}22{t_{j-1}}} + C_0 \eta_j^{\min(1,p-1)}.
\end{equation}
On the other hand, using Strichartz' estimate twice, one sees that
\begin{align*}
	&\norm{U(t-t_{j-1})w(t_{j-1})}_{W(I_j)}\\
	&{}\le \norm{U(t-t_0)w(t_0)}_{W(I_j)} +
	C \norm{\int_{t_0}^{t_{j-1}} U(t_{j-1}-s)e(s)ds }_{\M{s_c}22{t_{j-1}}} \\
	&{}+ C \norm{\int_{t_0}^{t_{j-1}} U(t_{j-1}-s)(|u|^{p-1}u-|\widetilde{u}|^{p-1}\widetilde{u})ds }_{\M{s_c}22{t_{j-1}}} \\
	&{}\le \eps +C^2 \norm{e }_{F([t_0,t_{j-1}])} + C^2 \norm{|u|^{p-1}u-|\widetilde{u}|^{p-1}\widetilde{u} }_{F([t_0,t_{j-1}])},
\end{align*}
which reads
\begin{equation}\label{eq:lpt2_4}
	\norm{U(t-t_{j-1})w(t_{j-1})}_{W(I_j)} \le C_1 \eps + C_1 \sum_{l=1}^{j-1} \kappa_l
\end{equation}
for some constant $C_1>1$.
Similarly, taking $C_1$ large, if necessary, we obtain
\begin{equation}\label{eq:lpt2_5}
	\norm{w(t_{j-1})}_{\M{s_c}22{t_{j-1}}}
	\le A' + C_1 \eps + C_1 \sum_{l=1}^{j-1} \kappa_l.
\end{equation}

We first suppose $p\ge2$.
Let us take a constant $\al\ge \max(2,C_0C_1)$.
For  $\eps'\le \eps_0$, we set
\begin{equation}\label{eq:lpt1_6}
	\eta_j=\eta_j(\eps'):= \al^{j-N_0-1} \eps'
\end{equation}
for $j\in [1,N_0+1]$. Then, we have
\begin{equation*}
	\eta_1 < \eta_2 <\dots<\eta_{N_0} < \eta_{N_0+1} = \eps'\le \eps_0.
\end{equation*}
We also remark that $\eta_j$ is increasing in $\eps'$.
We now show by induction that
\eqref{eq:lpt2_2} holds for $j\in [1,N_0]$ as long as $\eps'\le \eps_0$ and
$\eps\le\frac1{C_1} \eta_1(\eps')$.
To do so, it suffices to show that \eqref{eq:lpt2_1} 
is satisfied for $j\in[1,N_0]$ under this condition.
When $j=1$, \eqref{eq:lpt2_1} is fulfilled by assumption.
Assume for induction that \eqref{eq:lpt2_1} 
is true for $1\le j \le k$, where $k\in [1,N_0-1]$.
Since \eqref{eq:lpt2_2} is also true for $j\in[1,k]$,
we deduce from \eqref{eq:lpt2_4} that
\begin{align*}
	\norm{U(t-t_{k})w(t_{k})}_{W(I_{k+1})} \le C_1 \eps + C_1 \sum_{l=1}^{k} \kappa_l.
\end{align*}
By the assumptions on $\eps$ and $\al$, \eqref{eq:lpt1_6}, and
\eqref{eq:lpt2_2}, we have 
\[
	C_1 \eps \le  \eta_1 = \al^{-k}\eta_{k+1} 
\]
and
\begin{align*}
	C_1 \kappa_l \le C_1C_0 \eta_l \le \al \eta_{l} = \al^{l-k-1}\eta_{k+1}
\end{align*}
for $l \in [1, k]$.
Combining these estimates, we have
\begin{equation}\label{eq:lpt1_7}
	C_1 \eps + C_1 \sum_{l=1}^{k} \kappa_l
	\le \eta_{k+1} \al^{-k}\(1+\sum_{l=1}^k \al^{l-1}\) \le \eta_{k+1},
\end{equation}
where we have used the assumption $\al\ge2$ to 
yield $1+\sum_{l=1}^k \al^{l-1} \le \al^k$ in the last inequality.
Hence, \eqref{eq:lpt2_1} for $j=k+1$ follows from 
\eqref{eq:lpt2_4} and \eqref{eq:lpt1_7},
and so we see that \eqref{eq:lpt2_2} and \eqref{eq:lpt2_3} hold for $j\in[1,N_0]$.
Then, \eqref{eq:lpt1_7} is also true for $k=[1,N_0]$. 
Set $\eps_1=\eps_1(A',M):= \frac1{C_1} \eta_1(\eps_0)$ and
assume $\eps\le \eps_1$.
We define $\eps'$ by the relation $\eps=\frac1{C_1} \eta_1(\eps')$.
Notice that $\eps'\le \eps_0$ and $\eps'= C_1\al^{N_0+1} \eps$.
By \eqref{eq:lpt2_2},
\begin{align*}
	\norm{w}_{W(I)}
	\le\sum_{j=1}^{N_0} C_0 \eta_j 
	= \frac1{C_1} \sum_{j=1}^{N_0} C_1C_0 \eta_j\le \frac1{C_1} \eps'  = \al^{N_0+1} \eps.
\end{align*}
Further, by means of \eqref{eq:lpt2_3}, \eqref{eq:lpt2_5}, and \eqref{eq:lpt1_7},
we conclude that
\[
	\norm{w}_{L^\I(I_j,\Ms{s_c}22)} \le A' + (C_0+1) \eta_j \le A'+ C_1(C_0+1)\al^{N_0+1} \eps
\]
for all $j\in [1,N_0]$, which completes the proof for $p\ge2$.

Suppose $p<2$.
Let us take a constant $\al\ge2$ so large that
\[
	\al\ge C_1, \quad C_1C_0 \al^{p-1} \le \al,\quad
	\al^{-N_0-1}\eps_0 \le \frac12.
\]
For  $\eps'\le \eps_0$, we set
\[
	\eta_j=\eta_j(\eps'):= \al^j \exp\(\(\frac1{p-1}\)^{N_0+1-j} \log(\al^{-N_0-1}\eps') \)
\]
for $j\in [1,N_0+1]$. Then, we have
\begin{equation}\label{eq:lpt2_6}
	\eta_1 < \eta_2 <\dots<\eta_{N_0} < \eta_{N_0+1} = \eps'\le \eps_0.
\end{equation}
Indeed,
\begin{equation}\label{eq:lpt2_7}
	\frac{\eta_{j+1}}{\eta_j} = \al \(\al^{-N_0-1}\eps'\)^{-(2-p)(\frac1{p-1})^{N_0+1-j}}>\al>1
\end{equation}
We also remark that $\eta_j$ is increasing in $\eps'$.
We shall now show by induction that
\eqref{eq:lpt2_1} holds for $j\in [1,N_0]$ as long as $\eps'\le \eps_0$ and
$\eps\le\frac1\al \eta_1(\eps')$.
When $j=1$, it is obvious by assumption.
Assume for induction that \eqref{eq:lpt2_1} is true for $1\le j \le k$, where $k\in [1,N_0-1]$.
Since \eqref{eq:lpt2_2} and \eqref{eq:lpt2_3} are also true for $j\in[1,k]$ under this assumption,
we deduce from \eqref{eq:lpt2_4} that
\begin{align*}
	\norm{U(t-t_{k})w(t_{k})}_{W(I_{k+1})} \le C_1 \eps + C_1 \sum_{l=1}^{k} \kappa_l
	\le C_1 \eps + C_1 \sum_{l=1}^{k} C_0\eta_l^{p-1}.
\end{align*}
By the assumptions on $\al$ and $\eps$ and by \eqref{eq:lpt2_7}, we have 
\[
	C_1 \eps< \al \eps \le  \eta_1 < \al^{-k}\eta_{k+1} 
\]
and, for $l \in [1, k]$,
\begin{align*}
	C_1C_0 \eta^{p-1}_l
	&{}=C_1C_0 \al^{l(p-1)} \exp \( \(\frac1{p-1}\)^{N_0-l}\log(\al^{-N_0-1}\eps') \)\\
	&{}\le \al^{l} \exp \( \(\frac1{p-1}\)^{N_0-k}\log(\al^{-N_0-1}\eps') \) = \al^{l-k-1}\eta_{k+1},
\end{align*}
where we have used the inequalities $C_1C_0 \al^{p-1}\le \al$ and $\al^{(l-1)(p-1)}
\le \al^{l-1}$.
Combining these estimates, we have
\begin{equation}\label{eq:lpt2_8}
	C_1 \eps + C_1 \sum_{l=1}^{k} \kappa_l
	\le \eta_{k+1} \al^{-k}\(1+\sum_{l=1}^k \al^{l-1}\) \le \eta_{k+1}
\end{equation}
since $\al\ge2$. 
Thus, \eqref{eq:lpt2_1} for $j=k+1$ follows from \eqref{eq:lpt2_4} and \eqref{eq:lpt2_8}.
By induction, we also obtain \eqref{eq:lpt2_2} and \eqref{eq:lpt2_3} 
for $j\in [1,N_0]$.
Set $\eps_1=\eps_1(A',M):= \frac1\al\eta_1(\eps_0)$ and
assume $\eps\le \eps_1$.
We define $\eps'$ by the relation $\eps=\frac1\al\eta_1(\eps')$.
Notice that $\eps'\le \eps_0$ and
\[
	\eps'= \al^{N_0+1} \eps^{(p-1)^{N_0}}.
\]
By \eqref{eq:lpt2_2},
\begin{align*}
	\norm{w}_{W(I)} 
	&{}\le\sum_{j=1}^{N_0} C_0 \eta_j^{p-1} 
	< \sum_{j=1}^{N_0} C_1C_0 \eta_j^{p-1}
	\\&{}\le \sum_{j=1}^{N_0} \al^{j-N_0-1} \eta_{N_0+1} = \frac{1-\al^{-N_0}}{\al-1} \eps' 
	\le\eps' = \al^{N_0+1} \eps^{(p-1)^{N_0}}.
\end{align*}
We apply \eqref{eq:lpt2_8} to \eqref{eq:lpt2_5} to get
$\norm{w(t_{j-1})}_{\M{s_c}22{t_{j-1}}} \le A' + \eta_{j}$.
Then, \eqref{eq:lpt2_3} yields
\[
	\norm{w}_{L^\I(I_j,\Ms{s_c}22)} \le A' + \eta_{j} + C_0 \eta_j^{p-1}
	\le A' + \eps' + C_0 (\eps')^{p-1}
\]
for all $j\in [1,N_0]$.
\end{proof}

\section{Concentration compactness}\label{sec:4}
\subsection{Smallness via quadratic oscillation}
In this subsection, we show that  rapid oscillation of initial data
gives smallness of a solution to linear Schr\"odinger equation.
This is one of the key estimate of profile decomposition.
It will be also a key tool for the proof of Theorem \ref{thm:main3}.
This kind of smallness is given for $\F H^1$ data in \cite{CW1,Ma}.
\begin{proposition}\label{prop:ods}
Let $u_0 \in \FHsc$ and $b\in \R$.
Let $(\rho,q)$ is an acceptable pair
such that $\frac2\rho -\delta(q) + s = s_c$ for some $0\le s<s_c$.
Then,
\begin{equation}\label{eq:ods}
	\lim_{|b|\to\I} \LMn{U(t)e^{ib|x|^2}u_0}\rho{s}q=0.
\end{equation}
\end{proposition}
\begin{proof}
By the symmetry $\overline{(U(t)e^{ib|x|^2}\psi)(x)}=(U(-t)e^{-ib|x|^2}\overline{\psi})(x)$,
it suffices to show the case $b\to\I$.
Further, if we show the limit for some specific acceptable pair
$(\rho,q)$ with $\frac2\rho -\delta(q) + s = s_c$ for
some $s \in [0,s_c)$ then
it holds for all other pairs,
thanks to Lemma \ref{lem:interpolation} and Proposition \ref{prop:LMn_bdd}.
We therefore fix such a pair $(\rho,q)$ and consider the limit $b\to\I$.
Since $(\rho,q)$ is acceptable, we have $0<\delta(q)<1$ and $1<\rho \delta(q)$.
Further, by assumption $s<s_c$, we have $\rho \delta(q)<2$.
Let $\varphi_j \in C_0^\I(\R^N)$ be a cut-off function in the
definition of $\dot{M}^s_{q,r}(t)$.
We deduce from Stichartz' estimate that, for any $\eps>0$,
there exists a number $J=J(u_0,\eps)>0$ such that
\begin{multline*}
	\LMn{U(t)\sum_{|j|> J} \varphi_j e^{ib|x|^2}u_0}\rho{s}q
	\le C \norm{e^{ib|x|^2}\sum_{|j|> J} \varphi_j u_0}_{\FHsc}
	\\= C \norm{\sum_{|j|> J} \varphi_j u_0}_{\FHsc}\le \eps.
\end{multline*}
Since $\sum_{j\in\Z} \varphi_j = 1$ on $\{x\neq0\}$, it suffices to show \eqref{eq:ods}
for $\sum_{|j| \le J} \varphi_j u_0$.
For this, we prove \eqref{eq:ods} for each $\varphi_j u_0$.
Hence, we may suppose by scaling that $\supp u_0 \subset \{ 1/2 \le |x| \le 2\}$.
It holds that
\begin{align*}
	\Mn{U(t) e^{ib|x|^2} u_0}sq2t
	= \(\sum_{j=-1}^1 2^{2js} \Lebn{U(t)\varphi_j u_0}q^2 \)^{1/2}
	\le C \sum_{j=-1}^1\Lebn{U(t) \varphi_j u_0}q.
\end{align*}
Let us restrict our attention to the case $j=0$ since the estimates
for $\varphi_1 u_0$ and $\varphi_{-1}u_0$ are essentially the same.
It is obvious by assumption that $\varphi_0u_0=u_0$.
The pseudo-conformal transform 
\[
	(U(t) e^{ib|x|^2} u_0)(x) =
	e^{i\frac{b}{1+4bt}|x|^2} (1+4bt)^{-N/2} \(U\(\frac{t}{1+4bt}\) u_0 \)\(\frac{x}{1+4bt}\)
\]
gives us
\[
	\Lebn{U(t) e^{ib|x|^2} u_0}q = |1+4bt|^{-\delta(q)} \Lebn{U\(\frac{t}{1+4bt}\) u_0}q=: G_b(t).
\]
As a result, proof of  \eqref{eq:ods} boils down to showing
\begin{equation}\label{eq:ods2}
	\norm{
	G_b
	}_{L^{\rho,2}_t(\R)}
	\to 0
\end{equation}
as $b\to\I$ for any $u_0 \in \FHsc$ with
$\supp u_0 \subset \{1/2\le |x| \le 2\}$.

Let us prove \eqref{eq:ods2}.
Take $a>0$. Let $\beta=\beta(a)>0$  to be chose later.
We divide $\R$ into the following six intervals
\begin{align*}
	I_1 ={}& \(-\I,-\frac1{4b}-\frac{1}{4\beta}\),
	&
	I_2 ={}& \( -\frac1{4b}-\frac{1}{4\beta}, -\frac{a}{4ba-1} \),
	\\
	I_3 ={}& \(-\frac{a}{4ba-1},-\frac1{4b}\),
	&
	I_4 ={}& \(-\frac1{4b}, 0\),
	\\
	I_5 ={}& \(0,\frac1a \),
	&
	I_6 ={}& \(\frac1a,\I\).
\end{align*}
Notice that $I_j \neq \emptyset$ for $j=1,2,\dots,6$
for sufficiently large $b$ as long as $\beta$ is chosen independently of $b$.
By the $L^p$-$L^q$ estimate, we have
\[
	G_b(t) = \Lebn{U(t) e^{ib|x|^2} u_0}q 
	\le C|t|^{-\delta(q)} \Lebn{e^{ib|x|^2} u_0 }{q'}
	= C|t|^{-\delta(q)}\Lebn{ u_0 }{q'}.
\]
It follows from the assumptions $\supp u_0 \subset \{ 1/2 \le |x| \le 2 \}$
and  $0<\delta(q)<1$ that $\Lebn{u_0}{q'}\le C\norm{u_0}_{\FHsc}$.
This yields $G_b(t) \le C|t|^{-\delta(q)}\norm{u_0}_{\FHsc}$. Then,
\begin{equation}\label{eq:ods3}
	\norm{G_b
	{\bf 1}_{I_1}}_{L^{\rho,2}_t(\R)}
	\le C\norm{|t|^{-\delta(q)}{\bf 1}_{I_1}}_{L^{\rho,2}},
\end{equation}
where ${\bf 1}_I={\bf 1}_I(t)$ is a characteristic function of $I$.
Now, recall an equivalent representation of $L^{{\rho},2}$ norm
\[
	\norm{f}_{L^{{\rho},2}} \sim \(
	\frac2{q} \int_0^\I (t^{\frac1{{q}}} f^*(t))^2 \frac{dt}t\)^{\frac12},
\]
where $f^*$ is a non-increasing rearrangement of $f$ (see \cite{BL-Book}).
It then holds that
\begin{align*}
	\norm{|t|^{-\delta(q)}{\bf 1}_{I_1}}_{L^{\rho,2}}^2
	\sim{}& \int_0^\I t^{\frac2\rho-1} \( t+ \frac1{4\beta}+\frac1{4b} \)^{-2\delta(q)}dt\\
	\le {}& \int_0^\I t^{\frac2\rho-1} \( t+ \frac1{4\beta} \)^{-2\delta(q)}dt
	= C \beta^{2\delta(q)-\frac2\rho}.
\end{align*}
Hence, plugging this to \eqref{eq:ods3}, we have
\begin{equation}\label{eq:ods4}
	\norm{G_b
	{\bf 1}_{I_1}}_{L^{\rho,2}_t(\R)}
	\le C \beta^{\delta(q)-\frac1\rho}.
\end{equation}
The same argument shows
\begin{equation}\label{eq:ods5}
	\norm{G_b
	{\bf 1}_{I_6}}_{L^{\rho,2}_t(\R)}
	\le C a^{\delta(q)-\frac1\rho}.
\end{equation}
On the other hand,
by the generalized H\"older inequality,
\[
	\norm{G_b
	{\bf 1}_{I_2}}_{L^{\rho,2}}
	\le \norm{{\bf 1}_{I_2}}_{L^{r, \sigma}}
	\norm{G_b
	{\bf 1}_{I_2}}_{L^{\frac2{\delta(q)},\frac2{\delta(q)}}},	
\]
where $r=(\frac1\rho-\frac{\delta(q)}2)^{-1} =\frac2{s_c-s}\in (1,\I)$ and $\sigma=2/(1-\delta(q))>2$.
Notice that
\[
	\norm{{\bf 1}_{I_2}}_{L^{r, \sigma}}
	\sim |I_2|^{\frac1r}
	= \(\frac1{4\beta}-\frac{1}{4b(4ab-1)}\)^{\frac1r}
	\le C\beta^{-\frac1r}
\]
for large $b$.
Since $L^{\frac2{\delta(q)},\frac2{\delta(q)}}=L^{\frac2{\delta(q)}}$ with equivalent norm,
a change of variable shows
\[
	\norm{G_b
	{\bf 1}_{I_2}}_{L^{\frac2{\delta(q)},\frac2{\delta(q)}}(\R)}
	\sim \norm{U(t)u_0}_{L^{\frac2{\delta(q)}}((\frac1{4b}+\frac{\beta}{4b^2},a),L^q)}
	\le \norm{U(t)u_0}_{L^{\frac2{\delta(q)}}((0,a),L^q)}.
\]
Hence, one obtains
\begin{equation}\label{eq:ods6}
	\norm{G_b
	{\bf 1}_{I_2}}_{L^{\rho,2}}
	\le C\beta^{-\frac1r}\norm{U(t)u_0}_{L^{\frac2{\delta(q)}}((0,a),L^q)}.
\end{equation}
The same argument yields
\[
	\norm{G_b
	{\bf 1}_{I_3}}_{L^{\rho,2}}
	\le C\( 4b(4ab-1)\)^{-\frac1r}\norm{U(t)u_0}_{L^{\frac2{\delta(q)}}((a,\I),L^q)},
\]
\[
	\norm{G_b
	{\bf 1}_{I_4}}_{L^{\rho,2}}
	\le C\( 4b\)^{-\frac1r}\norm{U(t)u_0}_{L^{\frac2{\delta(q)}}((-\I,0),L^q)},
\]
and
\[
	\norm{G_b{\bf 1}_{I_5}}_{L^{\rho,2}}
	\le Ca^{-\frac1r}\norm{U(t)u_0}_{L^{\frac2{\delta(q)}}((0,\frac1{a+4b}),L^q)}.
\]
It then follows from these three estimates that
\begin{equation}\label{eq:ods7}
	\norm{G_b{\bf 1}_{I_3\cup I_4 \cup I_5}}_{L^{\rho,2}}\to 0
\end{equation}
as $b\to\I$ for each $a>0$. We now define $\beta=\beta(a)>0$ by the identity
\[
	\beta^{\delta(q)-\frac1\rho} = \beta^{-\frac1r} \norm{U(t)u_0}_{L^{\frac2{\delta(q)}}((0,a),L^q)}.
\]
More specifically, choose
$\beta= \norm{U(t)u_0}_{L^{2/{\delta(q)}}((0,a),L^q)}^{2/\delta(q)}$.
Summarize \eqref{eq:ods4}, \eqref{eq:ods5}, \eqref{eq:ods6},
and \eqref{eq:ods7} and put this $\beta$ to conclude that
\[
	\limsup_{b\to\I} \norm{G_b}_{L^{\rho,2}(\R)}
	\le C \norm{U(t)u_0}_{L^{\frac2{\delta(q)}}((0,a),L^q)}^{2(1-\frac1{\rho\delta(q))})} + Ca^{\delta(q)-\frac1\rho}
\]
Since $a$ is arbitrary, we let $a\to0$ and finally obtain
\[
	\limsup_{b\to\I} \norm{G_b}_{L^{\rho,2}(\R)}=0
\]
which completes the proof of \eqref{eq:ods2}.
\end{proof}
\subsection{Concentration compactness}
In this subsection,
we show that a bounded sequence of functions in $\FHsc$ is decomposed into a sum of several profiles.
Compared with a sequence bounded in $\ell_{\F H^1}(\cdot)$, which is an $\F H^1$-bounded sequence
up to a scaling normalization, the feature of $\FHsc$-bounded sequence 
is that it admits a sum of functions with extremely distinct scales.
Indeed, take $0\not\equiv \psi \in \mathcal{S}$ with $\supp \psi \subset \{1\le |x| \le 2\}$
and put
\[
	\psi_n(x) = \psi(x) + \psi_{\{n\}}(x) \in \mathcal{S},
\]
where we use the notation \eqref{eq:scaling2}.
Then, $\norm{\psi_n}_{\FHsc} = 2\norm{\psi}_{\FHsc} <\I$ but
$\ell_{\F H^1} (\psi_n) \to\I$ as $n\to\I$.
Thus, in order to extract some specific profiles from an $\FHsc$ bounded sequence,
we have to take this kind of scale decomposition into account.
To do so, we employ results in \cite{BG}, as in \cite{Ker}.
\begin{remark}\label{rmk:FHscFH1}
Let $\{\psi_n\}\subset \mathcal{S}$ be as above. Then, for sufficiently large $n$, we have
$\norm{2\psi_1}_{\FHsc}>\norm{\psi_n}_{\FHsc}$ and $\ell_c(2\psi_1)<\ell_c(\psi_n)$.
\end{remark}

The result is the following.
\begin{proposition}\label{prop:pd}
Let $\{\phi_n\}$ be a bounded sequence in $\FHsc$.
Then, there exist a subsequence of $\{\phi_n\}$,
which is denoted again by $\{\phi_n\}$, and sequences
$\psi^j \in \FHsc$, $h_n^j>0$, $\xi_n^j$,
and $W^l_n$ such that
\begin{equation}\label{eq:pd1}
	\phi_n = \sum_{j=1}^l e^{ix\xi_n^j} \psi^j_{\{ h_n^j \}}
	+ W^l_n
\end{equation}
for all $l\ge1$ and $j\ge1$.
If $i\neq j$ then
\begin{equation}\label{eq:pd2}
	\frac{h_n^i}{h_n^j} + \frac{h_n^j}{h_n^i} + \frac{|\xi^i_n-\xi^j_n|}{h_n^j} \to \I
\end{equation}
as $n\to\I$. Moreover, it holds that
\begin{equation}\label{eq:pd3}
	\norm{\phi_n}_{\FHsc}^2 =
	\sum_{j=1}^l \norm{\psi^j}_{\FHsc}^2  + \norm{W_n^l}_{\FHsc}^2 + o(1)
\end{equation}
as $n\to\I$. Furthermore, for any $1<\rho,q<\I$ with
$\delta(q)-s_c\in (0,\min(1,N/2))$ and $2/\rho-\delta(q)=s_c$, it holds that
\begin{equation}\label{eq:pd4}
	\liminf_{l\to\I} \limsup_{n\to\I} \norm{U(t)W^l_n}_{L^{\rho,\I}(\R,L^q)} =0.
\end{equation}
\end{proposition}
\begin{remark}
Following an argument in \cite{Ker}, one can obtain 
a refined version of the error estimate
\begin{equation}\tag{\ref{eq:pd4}${}'$}
	\lim_{l\to\I} \limsup_{n\to\I} \norm{U(t)W^l_n}_{L^{\rho,\I}(\R,L^q)} =0.
\end{equation}
However, just in order to simplify our argument, we work with \eqref{eq:pd4} since it is sufficient for later use.
\end{remark}
\subsection*{Scale decomposition}
We first show that an $\FHsc$-bounded sequence is decomposed
into several portions which have distinct scales.
\begin{definition}
Let $\{f_n \}$ be a bounded sequence in $L^2$.
Let ${\bf h} = \{ h_n \} $ be a sequence of positive numbers.
\begin{itemize}
\item We say $\{f_n\}$ is ${\bf h}$-scaled if
\[
	\limsup_{n\to\I} \( \int_{h_n|x|\le 1/R} | {f_n}(x)|^2 dx + 
	\int_{h_n|x| \ge R} |f_n(x)|^2 dx\) \to 0
\]
as $R \to \I$.
\item We say $\{f_n\}$ is ${\bf h}$-singular if, for any 
$0<R_1<R_2 <\I$, we have
\[
	\lim_{n\to\I} \int_{R_1<h_n|x|<R_2} | f_n(x)|^2 dx =0.
\]
\end{itemize}
\end{definition}
For a given sequence $\{f_n\}$ bounded in $\FHsc$,
$\{\F (|x|^{s_c} f_n)\}$ is a bounded sequence of $L^2$.
Applying decomposition of $L^2$-bounded sequence in \cite{BG} 
to this sequence, 
we obtain the following.
\begin{proposition}[\cite{BG}]\label{prop:pd1}
Let $\{f_n\}_n$ be a bounded sequence in $\FHsc$.
Then, there exist a subsequence of $\{f_n\}_n$, which we denote again by $\{f_n\}_n$,
a sequence of positive number $\{ h_n^j \}_{n,j}$, and a sequence
$\{g^j_n\}_{n,j} \subset \FHsc$ such that 
\begin{enumerate}
\item $\{h^{j}_n\}_n$ are pairwise orthogonal, i.e.
\[
	\frac{h^j_n}{h^{j'}_n} + \frac{h^{j'}_n}{h^j_n} \to \I
\]
as $n\to\I$, for any $j\neq j'$.
\item For each $j\ge1$, $\{|x|^{s_c} g^j_n\}_n$ is $\{h_n^j\}_n$-scaled.
\item For each $l\ge1$, define a sequence $\{R^l_n\}_n \subset \FHsc$ by 
\[
	f_n = \sum_{j=1}^l g^{j}_n + R^l_n.
\]
Then, $\{|x|^{s_c} R_n^l\}_n$ is $\{h_n^j\}_n$-singular for all $j \in [1,l]$.
Further,
\begin{equation}\label{eq:pds1_1}
	\limsup_{n\to\I} \norm{\F R_n^l}_{\dot{B}^{s_c}_{2,\I}} \to 0
\end{equation}
as $l\to \I$.
\item For any $l\ge1$, we have
\[
	\norm{f_n}_{\FHsc}^2 = \sum_{j=1}^l \norm{g_n^j}_{\FHsc}^2 + \norm{R_n^l}_{\FHsc}^2 + o(1)
\]
as $n\to\I$.
\end{enumerate}
\end{proposition}
The error estimate \eqref{eq:pds1_1} also yields
smallness of linear evolution of the error term.
To see this, we recall a refinement of the Sobolev embedding.
\begin{lemma}[\cite{BG,GMO}]\label{lem:pd1}
Let $1<p<q<\I$ and $s>0$ satisfy
\[
	\frac{N}{q} = \frac{N}p - s.
\]
Then, there exists a constant $C=C(N,p,q,s)$ such that
\[
	\norm{f}_{L^q} \le C \norm{f}_{\dot{H}^s_p}^{\frac{p}q}
	\norm{f}_{\dot{B}^s_{p,\I}}^{1-\frac{p}q}. 
\]
\end{lemma}
This lemma yields the following.
\begin{lemma}\label{lem:pd2}
Let $\{R_n^l\}$ be a bounded sequence in $\FHsc$ such that
\[
	\limsup_{n\to\I} \norm{\F R_n^l}_{\dot{B}^{s_c}_{2,\I}} \to 0
\]
as $l\to \I$. Then, for any $1<\rho,q<\I$ with
$\delta(q)-s_c\in [0,\min(1,N/2))$ and $2/\rho-\delta(q)=s_c$, it holds that
\begin{equation}
	\lim_{l\to\I} \limsup_{n\to\I} \norm{U(t)R^l_n}_{L^{\rho,\I}(\R,L^q)} \to 0
\end{equation}
as $l\to 0$.
\end{lemma}
\begin{proof}
First, we consider the special case $\delta(q)-s_c=0$, that is,
\[
	q=q_0:= \frac{N(p-1)}{N(p-1)-2} \in (2,2^*).
\]
In this case, $\rho=\rho_0:=s_c^{-1}$.
By means of Lemma \ref{lem:pd1}, it holds for $t\neq0$ that
\begin{align*}
	\Lebn{U(t)R_n^l}{q_0} ={}& \Lebn{M(-t)U(t)R_n^l}{q_0} \\
	\le {}& C \norm{M(-t)U(t)R_n^l}_{\dot{H}^{s_c}}^{\frac2{q_0}}
	\norm{M(-t)U(t)R_n^l}_{\dot{B}^{s_c}}^{1-\frac2{q_0}}\\
	\le{}& C|t|^{-\frac1{\rho_0}}\norm{U(t)R_n^l}_{\M{s_c}22t}^{\frac2{q_0}} 
	\norm{U(t)R_n^l}_{\M{s_c}2{\I}t}^{1-\frac2{q_0}}.
\end{align*}
Since $U(t)$ is bounded map $\M{s_c}220 \to \M{s_c}22t$ and
$\M{s_c}2{\I}0 \to \M{s_c}2{\I}t$, we have
\[
	\norm{U(t)R_n^l}_{L^{\rho_0,\I}(\R,L^{q_0})}
	\le C \norm{R_n^l}_{\FHsc}^{\frac2{q_0}} 
	\norm{\F R_n^l}_{\dot{B}^0_{2,\I}}^{1-\frac2{q_0}}.
\]
Taking limit supremum in $n$ and then letting $l\to\I$,
we obtain the result.

The general case follows by interpolation lemma (Lemma \ref{lem:interpolation})
since $U(t)R_n^l$ is bounded in $L^{\rho,\I}(\R,L^q)$
for any $(\rho,q)$ satisfying the assumption
in light of Proposition \ref{prop:LMn_bdd}.
\end{proof}

\subsection*{Decomposition of each scale}
Proposition \ref{prop:pd1} gives us a procedure to decompose
an $\FHsc$-bounded sequence $\{f_n\}_n$
into some pieces $\{g_n^j\}_{n,j}$ of which scales are pairwise orthogonal.
Let us next decompose each pieces $\{g_n^j\}_n$ into sums of
functions of the form $e^{is_n^j|x|^2} e^{i y_n^j \cdot x} \psi^j$.
For this, we consider a ${\bf 1}$-scaled sequence,
where we denote by ${\bf 1}$ a sequence of positive numbers
of which all component is equal to one. 
For a sequence $\{\phi_n\}_n\subset \FHsc$, we set
\[
	\nu(\{\phi_n\}) := \left\{
	\phi \in \FHsc \Bigm| 
	\begin{aligned} 
	&\exists s_n \in \R, \exists y_n \in \R^N 
	\text{ s.t. } e^{-is_n|x|^2} e^{-iy_nx} \phi_n \rightharpoonup \phi \\
	&\text{ in }\FHsc\text{ as }n\to\I,\text{ up to subsequence}.
	\end{aligned}\right\}
\]
and
\[
	\eta (\{\phi_n\}) := \sup_{\phi \in \nu (\{\phi_n\})} \norm{\phi}_{\FHsc} .
\]
It is obvious by definition that
\[
	\eta (\{\phi_n\}) \le \limsup_{n\to\I} \norm{\phi_n}_{\FHsc}.
\]
\begin{proposition}\label{prop:pd2}
Let $\{\phi_n\}_n \subset \FHsc$ be a bounded sequence such that 
$\{|x|^{s_c}\phi_n\}$ is ${\bf 1}$-scaled.
Then, there exist a subsequence of $\{\phi_n\}_n$, which is denoted again by $\{\phi_n\}_n$,
and sequences $\{\psi^j\}_j \subset \nu (\{\phi_n\})$, $\{w_n^j\}_{j,n} \subset \FHsc$
$\{s_n^j \}_{j,n} \subset \R$, and $\{y_n^j \}_{j,n} \subset \R^N$ such that
if $j\neq k$ then
\begin{equation}\label{eq:pdtmp21}
	|s_n^j-s_n^k| + |y^j_n-y^k_n| \to \I
\end{equation}
as $n\to\I$ and that, for all $l\ge1$,
\[
	\phi_n = \sum_{j=1}^l e^{is_n^j|x|^2} e^{iy_n x} \psi^j + w_n^l
\]
for all $n\ge 1$ and
\[
	\norm{\phi_n}_{\FHsc}^2 = \sum_{j=1}^l \norm{\psi^j}_{\FHsc}^2 + \norm{w^l_n}_{\FHsc}^2
	+o(1)
\]
as $n\to\I$.
Moreover,
\begin{equation}\label{eq:pds2_1}
	\eta(\{w^l_n\}_n)\to 0
\end{equation}
as $l\to \I$ and $\{|x|^{s_c}w^l_n\}_n$ is ${\bf 1}$-scaled uniformly in $l$, i.e.
there exists a function $\zeta:\R_+\to\R_+$ such that $\zeta(R)\to0$ as $R\to\I$
and
\[
	\limsup_{n\to \I} \int_{ \{ |x|\le \frac1R\} \cup \{ |x|\ge R \} }
	|x|^{2s_c} |w^l_n(x)|^2 dx \le \zeta(R)
\]
for all $l\ge1$.
\end{proposition}
\begin{proof}
If $\eta(\{\phi_n\})=0$ then the result follows by choosing $\psi^j\equiv 0$
and $w^l_n = \phi_n$. 
Otherwise, there exists $\psi^1 \in \nu(\{\phi_n\})$ satisfying
\[
	0< \frac{\eta(\{\phi_n\})}2 \le \norm{\psi^1}_{\FHsc}.
\]
By the definition of $\nu$, there also exist sequences $\{s^1_n\} \subset \R$ 
and $\{y^1_n\}\subset \R^N$ such that
\[
	e^{-is_n^1|x|^2} e^{-i y^1_n x} \phi_n \rightharpoonup \psi^1 \IN \FHsc.
\]
Define a sequence $\{w^1_n\}$ by $w^1_n=\phi_n- e^{is_n^1|x|^2} e^{i y^1_n x} \psi^1$.
It then follows that
\begin{equation}\label{eq:pdlem1}
	e^{-is_n^1|x|^2} e^{-i y^1_n x} w_n^1 \rightharpoonup 0 \IN \FHsc.
\end{equation}
Further, we obtain
\[
	\norm{\phi_n}_{\FHsc}^2 = \norm{\psi^1}_{\FHsc}^2 + \norm{w_n^1}_{\FHsc}^2 +o(1).
\]
Similarly, for any bounded function $\chi$,
\[
	\norm{\chi\phi_n}_{\FHsc}^2 = \norm{\chi\psi^1}_{\FHsc}^2 + \norm{\chi w_n^1}_{\FHsc}^2 +o(1).
\]

We repeat the above argument by replacing $\{\phi_n\}$ by $\{w^1_n\}$.
If $\eta(\{w_n^1\})>0$ then we can chose $\psi^2\neq0$, $\{s^2_n\}$, and $\{y^2_n\}$ such that
\begin{equation}\label{eq:pdlem2}
	w^2_n := w^1_n - e^{is_n^2|x|^2} e^{i y^2_n x} \psi^2 \rightharpoonup 0 \IN \FHsc
\end{equation}
and
\[
	\norm{w_n^1}_{\FHsc}^2 = \norm{\psi^2}_{\FHsc}^2 + \norm{w_n^2}_{\FHsc}^2 +o(1).
\]
as $n\to\I$ (up to subsequence). We claim that
$|s^1_n-s^2_n|+|y^1_n-y^2_n|\to \I$ as $n\to\I$. 
Indeed, otherwise we see from \eqref{eq:pdlem1} and \eqref{eq:pdlem2} that
$\psi^2=0$, a contradiction.

By this procedure,
we inductively define $\psi^j$, $\{s_n^j\}$, $\{y_n^j\}$, and $\{w^l_n\}$.
Then, for any $l\ge 1$,
\[
	\sum_{j=1}^l \norm{\psi^j}_{\FHsc}^2 \le \limsup_{n\to\I} \norm{\phi_n}_{\FHsc}^2<\I
\]
holds.
This implies $\norm{\psi^j}_{\FHsc}\to 0$ as $j\to\I$, and so
$$\eta(\{w^l_n\}) \le 2\norm{\psi^{l+1}}_{\FHsc}\to 0$$ as $l\to\I$.
A similar argument shows 
\[
	\limsup_{n\to\I} \norm{\chi w_n^l}_{\FHsc}^2 \le \limsup_{n\to\I} \norm{\chi \phi_n}_{\FHsc}^2
\]
for any bounded function $\chi$ and $l\ge1$.
Take $\chi=\chi_{\{|x|\le 1/R\}\cup\{|x|\ge R\}}$ and set
\[
	\zeta(R) = \limsup_{n\to\I} \norm{\chi \phi_n}_{\FHsc}^2.
\]
Then, since $\{|x|^{s_c} \phi_n\}$ is ${\bf 1}$-scaled, $\zeta(R)\to0$ as $R\to\I$.
\end{proof}
It is possible to upgrade the smallness property \eqref{eq:pds2_1} as follows.
\begin{lemma}\label{lem:pd3}
Let $\{ w^l_n \} \subset \FHsc$ be a bounded sequence with
$\eta(\{w^l_n\}_n)\to 0$ as $l\to \I$.
Further, suppose that $\{ |x|^{s_c}w^l_n \} \subset \FHsc$ is
 ${\bf 1}$-scaled uniformly in $l$. Then, for any 
 $1<\rho,q<\I $ with $\delta(q)-s_c \in (0,\min(1,N/2))$ and $\frac2\rho-\delta(q) =s_c$,
it holds that
\[
	\limsup_{n\to\I} \norm{U(t) w^l_n }_{L^{\rho,2}(\R,\Ms{0}q2)} \to 0
\]
as $l\to\I$.
\end{lemma}
\begin{proof}
Set $\sigma_R(x)=\chi_{\{|x|\le 1/R\}\cup\{|x|\ge R\}}(x)$.
Take $1<\rho,q<\I $ so that $\delta(q) \in (0,\min(1,N/2))$ and $\frac2\rho-\delta(q) =s_c$.
Then,
\begin{equation}\label{eq:pdlem21}
\begin{aligned}
	\norm{U(t) w^l_n }_{L^{\rho,2}(\R,\Ms{0}q2)}
	\le{}& \norm{U(t) \sigma_R w^l_n }_{L^{\rho,2}(\R,\Ms{0}q2)} \\
	&{}+ \norm{U(t) (1-\sigma_R)w^l_n }_{L^{\rho,2}(\R,\Ms{0}q2)}.
\end{aligned}
\end{equation}
Since $\{ |x|^{s_c}w^l_n \}$ is uniformly ${\bf 1}$-scaled,
we see from Lemma \ref{lem:LMn_inclusion} and Strichartz' estimate
that there exists $\zeta(R):\R_+\to\R_+$ with $\zeta(R)\to0$ as $R\to\I$ such that
\begin{equation}\label{eq:pdlem22}
	\norm{U(t) (1-\sigma_R)w^l_n }_{L^{\rho,2}(\R,\Ms{0}q2)}
	\le C \norm{ (1-\sigma_R)w^l_n }_{\FHsc} \le C\zeta(R).
\end{equation}
On the other hand, for $R\gg 1$,
\[
	\norm{U(t)\sigma_R w^l_n}_{\M0q2t}^2 = \sum_{k=-C\log R}^{C\log R}
	\norm{U(t)\chi_k \sigma_R w^l_n}_{L^q}^2,
\]
where $C$ is a positive constant and $\chi_j\in C_0^\I$ is the function
defined in the definition of $\M{s}qrt$.
Let $\theta \in (1/2,1)$ to be chosen later.
Using H\"older's inequality twice, we see that
\begin{align*}
	\sum_{k=-C\log R}^{C\log R}
	\norm{U(t)\chi_k \sigma_R w^l_n}_{L^q}^2
	\le {}&\( \sum_{k=-C\log R}^{C\log R}
	\norm{U(t)\chi_k \sigma_R w^l_n}_{L^\I}^2\)^{1-\theta} \\
	&{} \times \norm{U(t)\sigma_R w^l_n}_{\M0{q\theta}2t}^{2\theta} \\
	\le {}& \( \sum_{k=-C\log R}^{C\log R}
	 \norm{|t|^{\frac{N}2} U(t)\chi_k \sigma_R w^l_n}_{L^\I_{t,x}}  \)^{2(1-\theta)}\\
	&{} \times |t|^{-{N}(1-\theta)}\norm{U(t)\sigma_R w^l_n}_{\M0{q\theta}2t}^{2\theta}
\end{align*}
for $t\neq0$. Generalized H\"odler's inequality gives us
\begin{equation}\label{eq:pdlem3}
\begin{aligned}
	\LMn{U(t)\sigma_R w^l_n}{\rho}0q \le {}&\( \sum_{k=-C\log R}^{C\log R}
	 \norm{|t|^{\frac{N}2} U(t)\chi_k \sigma_R w^l_n}_{L^\I_{t,x}}  \)^{1-\theta}\\
	&{} \times \norm{|t|^{-{N}(1-\theta)}\norm{U(t)\sigma_R w^l_n}_{\M0{q\theta}2t}^{\theta}}_{L^{\rho,2}}.
\end{aligned}
\end{equation}
Here, we have
\begin{equation}\label{eq:pdlem35}
	\norm{|t|^{\frac{N}2} U(t)\chi_k \sigma_R w^l_n}_{L^\I_{t,x}}
	=\sup
	\(\limsup_{n\to\I} |t_n|^{\frac{N}2} \abs{(U(t_n)(\chi_k \sigma_R w^l_n))(x_n)}\),
\end{equation}
where the supremum is taken over all sequences $\{t_n\}\subset \R\setminus\{0\}$
and $\{x_n\}\subset \R^N$.
Further, it follows from the well-known integral representation of $U(t)$ that
\[
	|t_n|^{\frac{N}2} \abs{U(t_n)(\chi_k \sigma_R w^l_n)(x_n)}=
	C\abs{\int_{\R^N} \(e^{-is_n|x|^2}e^{-iy_n x} w^l_n(x)\) (\chi_k\sigma_R)(x)dx},
\]
where $s_n=-1/4t_n$ and $y_n=x_n/2t_n$.
The limit supremum in the right hand side of \eqref{eq:pdlem35}
is hence bounded by
\begin{align*}
	\sup_{V\in \nu(\{w^l_n\}_n)} C\abs{\int_{\R^N} V(x) (\chi_k\sigma_R)(x)dx}
	&{}\le C \eta(\{w^l_n\}) \norm{|x|^{-s_c}\chi_k}_{L^2} \\
	&{}\le C2^{(\frac{N}2-s_c) k} \eta(\{w^l_n\}).
\end{align*}
Thus,
\begin{equation}\label{eq:pdlem4}
\begin{aligned}
	\limsup_{n\to\I}\sum_{k=-C\log R}^{C\log R}
	 \norm{|t|^{\frac{N}2} U(t)\chi_k \sigma_R w^l_n}_{L^\I_{t,x}} 
	 &{}\le C \eta(\{w^l_n\}) \( \sum_{k=-C\log R}^{C\log R} 2^{(\frac{N}2-s_c) k} \)\\
	 &{}\le C(R) \eta(\{w^l_n\}).
\end{aligned}
\end{equation}
On the other hand, one deduces from generalized H\"odler's inequality that
\begin{equation}\label{eq:pdlem5}
\begin{aligned}
	&\norm{|t|^{-{N}(1-\theta)}\norm{U(t)\sigma_R w^l_n}_{\M0{q\theta}2t}^{\theta}}_{L^{\rho,2}} \\
	&{}\le \norm{ |t|^{-{N}(1-\theta)} }_{L^{\frac2{N(1-\theta)},\I}}
	\norm{\norm{U(t)\sigma_R w^l_n}_{\M0{q\theta}2t}^{\theta}}_{L^{\widetilde{\rho},2}} \\
	&{}\le C \norm{U(t)\sigma_R w^l_n}_{L^{\widetilde{\rho}\theta,2\theta}(\Ms0{q\theta}2)}^{\theta},
\end{aligned}
\end{equation}
where $\widetilde{\rho}$ is defined by $(\widetilde{\rho})^{-1}=\rho^{-1}-\frac{N(1-\theta)}{2}$.
Notice that $\widetilde{\rho}$ is well-defined if $\theta$ is sufficiently close to one.
Define $\rho_\pm$ by
\[
	\frac1{\rho_\pm} = \frac1{\widetilde{\rho}\theta} \pm \eps
\]
with $\eps>0$ to be chosen later. $\rho_\pm$ is well-defined if $\eps$ is sufficiently small.
Then, by Lemma \ref{lem:LMn_inclusion} and Strichartz' estimate, we have
\begin{align*}
	&\norm{U(t)\sigma_R w^l_n}_{L^{\widetilde{\rho}\theta,2\theta}(\Ms0{q\theta}2)}\\
	&{}\le
	C\(\norm{U(t)\sigma_R w^l_n}_{L^{\rho_+,2}(\Ms0{q\theta}2)}+
	\norm{U(t)\sigma_R w^l_n}_{L^{\rho_-,2}(\Ms0{q\theta}2)}\)\\
	&{}\le C\(
	\norm{\sigma_R w^l_n}_{\F\dot{H}^{s_+}}+
	\norm{\sigma_R w^l_n}_{\F\dot{H}^{s_-}}\)
\end{align*}
as long as $\delta(q\theta)-s_\pm \in [0,\min(1,\frac{N}2))$, where
\[
	s_\pm:= \frac2{\rho_\pm} - \delta(q\theta)  = s_c + \delta(q)-\delta(q\theta) \pm \eps.
\]
Notice that if $\theta$ is sufficiently close to one
and $\eps$ is sufficiently small then 
\[
	|s_\pm-s_c| + |\delta(q)-\delta(q\theta)| \ll 1.
\]
Since $\delta(q)-s_c \in (0,\min(1,\frac{N}2))$, we can chose $\theta$ and 
$\eps$ so that $\delta(q\theta)-s_\pm \in [0,\min(1,\frac{N}2))$.
For such $\theta$ and $\eps$, we have
\begin{equation}\label{eq:pdlem6}
	\norm{U(t)\sigma_R w^l_n}_{L^{\widetilde{\rho}\theta,2\theta}(\Ms0{q\theta}2)}
	\le CR^{s_+-s_c} \norm{ w^l_n}_{\FHsc}.
\end{equation}
Combining \eqref{eq:pdlem4}, \eqref{eq:pdlem5}, and \eqref{eq:pdlem6} to \eqref{eq:pdlem3},
we obtain
\[
	\limsup_{n\to\I} 
	\norm{U(t)\sigma_R w^l_n}_{L^{\rho,2}(\Ms0{q}2)}
	\le C(R) \eta(\{w_n^l\})^\theta \(\limsup_{n\to\I} \norm{ w^l_n}_{\FHsc} \)^{1-\theta}. 
\]
By this inequality, \eqref{eq:pdlem21}, and \eqref{eq:pdlem22}, we have
\[
	\limsup_{n\to\I} 
	\norm{U(t)w^l_n}_{L^{\rho,2}(\Ms0{q}2)}
	\le C(R) \eta(\{w_n^l\})^\theta \(\limsup_{n\to\I} \norm{ w^l_n}_{\FHsc} \)^{1-\theta}
	+ C \zeta(R).
\]
Take limit supremum in $l$ and then let $R\to\I$ to obtain the desired result.
\end{proof}
\subsection*{Completion of the proof}
We are now ready to prove Proposition \ref{prop:pd}. 
\begin{proof}[Proof of Proposition \ref{prop:pd}]
In what follows, we denote various subsequences of
$\{\phi_n\}$ by $\{\phi_n\}$.
By Proposition \ref{prop:pd1}, there exist $\{R^j_n\} \subset \FHsc$, 
$\{h_n^j\}\subset \R_+$, and $\{q_n^j\}$ such that
\[
	\phi_n = \sum_{j=1}^J g_n^j + R_n^J,
\]
where $|x|^{s_c}g_n^j$ is $\{h_n^j\}$-scaled and $R_n^J$ is $\{h_n^j\}$-singular for all $j\in[1,J]$.
Set a sequence $\{P_n^j\}$ by
$g_n^j:=(P_{n}^j)_{\{h_n^j\}}$ or, equivalently, by $P_n^j:=(g_{n}^j)_{\{1/h_n^j\}}$.
Then, one sees that $|x|^{s_c} P_n^j$ is ${\bf 1}$-scaled.
Moreover,
\begin{align*}
	\norm{\phi_n}_{\FHsc} 
	&{}= \sum_{j=1}^J \norm{g_n^j}_{\FHsc}^2 + \norm{R_n^j}_{\FHsc}^2 +o(1) \\
	&{}= \sum_{j=1}^J \norm{P_n^j}_{\FHsc}^2 + \norm{R_n^j}_{\FHsc}^2 +o(1),
\end{align*}
which in particular implies a uniform bound on $\{P_n^j\}$;
\[
	\sup_{j} \(\limsup_{n\to\I} \norm{P_n^j}_{\FHsc} \) \le 
	\limsup_{n\to\I} \norm{\phi_n}_{\FHsc}.
\]
Now, apply Proposition \ref{prop:pd2} to $\{P_n^j\}_n$ to infer that, for each $j,K\ge1$,
\[
	P_n^j = \sum_{k=1}^K \varphi^{(j,k)} e^{is_n^{(j,k)}|x|^2} e^{iy_n^{(j,k)}x}
	+ w_n^{(j,K)}
\]
for some $\{\varphi^{(j,k)}\} \subset \FHsc$, $\{w_n^{(j,k)}\}\subset \FHsc$,
$\{s_n^{(j,k)}\} \subset \R$, and $\{y_n^{(j,k)}\}\subset \R^N$
with properties stated in Proposition \ref{prop:pd2}. By extracting a subsequence if necessary,
we assume that $s_n^{(j,k)} \to \overline{s}^{(j,k)} \in [-\I,\I]$ as $n\to \I$.
Combining these expansions, we obtain
\[
	\phi_n = \sum_{j=1}^{J} \sum_{k=1}^{K_j}
	\(\varphi^{(j,k)} e^{is_n^{(j,k)}|x|^2} e^{iy_n^{(j,k)}x}\)_{\{h_n^j\}}
	+ \sum_{j=1}^J {w_n^{(j,K_j)}}_{\{h_n^j\}} + R_n^J.
\]

Take an integer $l\ge1$. 
Using Lemma \ref{lem:pd2},
one can choose an integer $J_0(l)\ge1$ such that
\begin{equation}\label{eq:pdproof1}
	\limsup_{n\to\I} \norm{U(t) R_n^J}_{L^{\rho,\I}(\R,L^q)} \le 2^{-l}
\end{equation}
holds as long as $J\ge J_0(l)$. We may assume $J_0(l+1)\ge J_0(l)$ without 
loss of generality.
Further, for each $j\in [1,J_0(l)]$,
take a number $K_{0,j}=K_{0,j}(l)$ so that
\begin{equation}\label{eq:pdproof2}
	\limsup_{n\to\I} \norm{w_n^{(j,K)}}_{L^{\rho,2}(\R, \Ms{0}q2)}
	\le 2^{-l} J_0(l)^{-1}
\end{equation}
as long as $K\ge K_{0,j}(l)$. 
This is possible because of Lemma \ref{lem:pd3}.
Replacing with larger one if necessary, we let $K_{0,j}(l+1)\ge K_{0,j}(l)$.
We also set $K_{0,j}(l)=0$ if $j> J_0(l)$.
We define a set $A_l \subset \Z_+^2$ by 
\[
	A_l := \bigcup_{j=1}^{J_0(l)} \bigcup_{k=K_{0,j}(l-1)+1}^{K_{0,j}(l)} \{(j,k)\}.
\]
It is obvious that $A_{l_1}\cap A_{l_2}=\emptyset$ if $l_1\neq l_2$.
Introduce two subsets of $A_l$ by
\[
	B_l := \{(j,k)\in A_l \ |\ |\overline{s}^{(j,k)}|<\I \}, \qquad
	C_l := A_l \setminus B_l.
\] 

We enumerate indices $(j,k)$ belonging to $ \cup_{l\ge1} B_l$ in the following manner.
Let $Z=\{k \in \Z_+\ |\ k \le \#(\cup_{l\ge1} B_l)\}$ if $ \cup_{l\ge1} B_l$ is 
a finite set, otherwise $Z=\Z_+$.
Let $m_l = \# (\cup_{j=1}^l B_j)$.
Then, take a bijection ${\bf m}: Z \to \cup_{l\ge1} B_l$ so that
\[
	{\bf m}^{-1} (B_l) = \left\{m\in Z\ |\ 1+ m_{l-1} \le m
\le  m_l \right\}
\]
and ${\bf m}^{-1}(j,k_1) < {\bf m}^{-1}(j,k_2)$ for all 
$(j,k_i) \in \cup_{l\ge1} B_l$ with $k_1<k_2$.
We identify $(j,k)\in \cup_{l\ge1} B_l$ with $m\in Z$
by means of ${\bf m}$, in what follows.

Let us go back to the expansion.
Set $h_n^m=h_n^{j({\bf m}(m))}$, where $j({\bf m}(m))$ is the first component of ${\bf m}(m)$.
Then, we have
\begin{align*}
	\phi_n &{}= \sum_{j=1}^{J_0(l)} \sum_{k=1}^{K_{0,j}(l)}
	\(\varphi^{(j,k)} e^{is_n^{(j,k)}|x|^2} e^{iy_n^{(j,k)}x}\)_{\{h_n^j\}}
	+ \sum_{j=1}^{J_0(l)} {w_n^{(j,K_{0,j}(l))}}_{\{h_n^j\}} + R_n^{J_0(l)} \\
	&{}= \sum_{(j,k)\in \cup_{j=1}^l A_j}
	\(\varphi^{(j,k)} e^{is_n^{(j,k)}|x|^2} e^{iy_n^{(j,k)}x}\)_{\{h_n^j\}}
	+ \sum_{j=1}^{J_0(l)} {w_n^{(j,K_{0,j}(l))}}_{\{h_n^j\}} + R_n^{J_0(l)} \\
	&{} = \sum_{m=1}^{m_l} e^{i \xi_n^m x} \psi^m_{\{h^m_n\}}
	+ W^{m_l}_n,
\end{align*}
where $\psi^m= \varphi^m e^{i\overline{s}^{m}|x|^2}$, $\xi^m_n = h_n^m y^m_n$, and 
\begin{equation}\label{eq:pdproof}
\begin{aligned}
	W_n^{m_l} ={}& 
	\sum_{m=1}^{m_l} (\varphi^m (e^{i{s}_n^{m}|x|^2}- e^{i\overline{s}^{m}|x|^2})e^{iy^m_n x} )_{\{h^m_n\}} \\
	&{} + \sum_{(j,k)\in \cup_{j=1}^l C_j} 
	(\varphi^{(j,k)} e^{i{s}_n^{(j,k)}|x|^2}e^{iy^{(j,k)}_n x} )_{\{h^j_n\}}\\
	&{} + \sum_{j=1}^{J_0(l)} {w_n^{(j,K_{0,j}(l))}}_{\{h_n^j\}} + R_n^{J_0(l)}=:I_1+I_2
	+I_3 + I_4.
\end{aligned}
\end{equation}
For $m'\in (m_l,m_{l+1})$, we set
\[
	W_n^{m'} = W_n^{m_{l+1}} - \sum_{m=m'+1}^{m_{l+1}}
	e^{i \xi_n^m x} \psi^m_{\{ h^m_n\}}.
\]
Then, \eqref{eq:pd1} holds.

We shall prove the sequences $\{\psi^m\}$, $\{\xi^m_n\}$, and $\{W_n^l\}$
possess the desired properties \eqref{eq:pd2}--\eqref{eq:pd4}.

Let $m_1,m_2 \in Z$, $m_1\neq m_2$.
If $j({\bf m}(m_1))\neq j({\bf m}(m_2))$ then the scales $\{h_n^{m_1}\}_n$ and
$\{h_n^{m_2}\}_n$ are orthogonal;
\[
	\frac{h_n^{m_1}}{h_n^{m_2}} + \frac{h_n^{m_2}}{h_n^{m_1}} \to \I
\]
as $n\to \I$. Hence, \eqref{eq:pd2} holds. 
If $j({\bf m}(m_1))= j({\bf m}(m_2))$ then 
$|s_n^{m_1} - s_n^{m_2}| + |y_n^{m_1} - y_n^{m_2}|$ tends to infinity as $n\to\I$
by means of \eqref{eq:pdtmp21}.
Recall that $s_n^{m_i}$ converges to a number $\overline{s}^{m_i} \in \R$ by the definition of $B_l$.
Therefore, we have
\[
	\frac{|\xi^{m_1}_n-\xi^{m_2}_n|}{h_n^{m_1}} = |y_n^{m_1} - y_n^{m_2}| \to \I
\]
as $n\to\I$. The limit \eqref{eq:pd2} is true also in this case.

The equality \eqref{eq:pd3} is rather trivial by
Propositions \ref{prop:pd1} and \ref{prop:pd2}.
We only note that
\begin{equation}\label{eq:pdproof3}
\begin{aligned}
	\norm{(\varphi^m (e^{i{s}_n^{m}|x|^2}- e^{i\overline{s}^{m}|x|^2})e^{iy^m_n x} )_{\{h^m_n\}}}_{\FHsc} &{}= \norm{\varphi^m (e^{i({s}_n^{m}-\overline{s}^m)|x|^2}- 1)}_{\FHsc}\\
	&{}=o(1)
\end{aligned}
\end{equation}
 as $n\to\I$ because ${s}_n^{m}$ converges to $\overline{s}^m\in \R$.

Let us show \eqref{eq:pd4}. Let $m=m_l$.
One verifies from Proposition \ref{prop:LMn_bdd} and \eqref{eq:pdproof3} that
\begin{equation}\label{eq:pdproof4}
	\lim_{n\to \I} \LMn{U(t) I_1}{\rho}{0}q = 0.
\end{equation}
Next we consider $I_2$. By Galilean transform,
\[
	(U(t)(\varphi e^{is|x|^2}e^{iyx}))(x)
	= e^{-i|y|^2t} e^{iyx}(U(t)(\varphi e^{is|x|^2}))(x-2yt)
\]
and so
\[
	\norm{U(t)(\varphi e^{is|x|^2}e^{iyx})}_{L^q}
	=\norm{U(t)(\varphi e^{is|x|^2})}_{L^q}.
\]
Since $(j,k)\in \cup_{l\ge1}C_l$ implies $\overline{s}^{(j,k)}$ equals to $+\I$ or $-\I$,
it follows from Proposition \ref{prop:ods} that
\begin{equation}\label{eq:pdproof5}
	\lim_{n\to\I} \norm{U(t) I_2}_{L^{\rho,2}(\R,L^q)} =0.
\end{equation}
Therefore, plugging \eqref{eq:pdproof1}, \eqref{eq:pdproof2}, \eqref{eq:pdproof4},
and \eqref{eq:pdproof5} to \eqref{eq:pdproof}, we conclude that
\begin{align*}
	\limsup_{n\to\I} \norm{U(t) W_n^{m_l}}_{L^{\rho,\I}(\R,L^q)}
	\le{}& \sum_{j=1}^{J_0(l)} \limsup_{n\to\I} \norm{ U(t) w_n^{(j,K_{0,j}(l))}}_{L^{\rho,\I}(\R,L^q)}
	\\&{}+ \limsup_{n\to\I} \norm{ U(t) R_n^{J_0(l)}}_{L^{\rho,\I}(\R,L^q)}\\
	\le{}& 2^{-l+1},
\end{align*}
which proves \eqref{eq:pd4} since $l>0$ is arbitrary.
\end{proof}

\section{Proof of Theorem \ref{thm:main2}}\label{sec:6}
In what follows, we prove our main theorems.
Let us begin with the proof of Theorem \ref{thm:main2}.
\begin{proof}[Proof of Theorem \ref{thm:main2}]
Suppose $u_0 \in \dot{H}^1 \cap \FHsc$ satisfies $E[u_0]<0$.
We first note that Hardy's inequality implies $\FHsc \hookrightarrow \dot{H}^{-s_c}$.
Therefore, $u_0 \in \dot{H}^1 \cap \dot{H}^{-s_c} \subset H^1$.
It is known that \eqref{eq:NLS}-\eqref{eq:IC} is globally well-posed in $H^1$ and we see that
the solution $u$ belongs to $C(\R,H^1)$ and that the mass and the energy is conserved,
\[
	\Lebn{u(t)}2 = \Lebn{u_0}2, \qquad E(u(t))=E(u_0)
\]
for any $t\in \R$. By uniqueness, this solution coincides with the one given by Theorem \ref{thm:lwp}.
By assumption,
\begin{equation}\label{eq:pf2_1}
	\Lebn{u(t)}{p+1} \le -(p+1) E(u(t)) = -(p+1)E(u_0)>0.
\end{equation}
Since $q(L)>p+1>2$, by means of H\"older's inequality,
there exists $\theta \in (0,1)$ independent of $t$ such that
\[
	\Lebn{u(t)}{p+1} \le \Lebn{u(t)}2^{1-\theta} \Lebn{u(t)}{q(L)}^\theta.
\]
By the mass conservation and \eqref{eq:pf2_1}, one sees that
$\Lebn{u(t)}{q(L)}$ is bounded by some positive constant from below.
Hence, $\norm{u}_{S(\R_+)}=\I$. 
Thanks to Proposition \ref{prop:nscond},
this implies $u_0 \not\in S_+$.

Furthermore, if $0<d<1$ is sufficiently close to one then
\[
	E(du_0) = E(u_0) + \frac{1-d^2}{2}\Lebn{\nabla u_0}2^2 -\frac{1-d^{p+1}}{p+1}\Lebn{u_0}{p+1}^{p+1} 
	<0.
\]
Then, $du_0 \not\in S_+$ as shown above, and so
\[
	\norm{u_0}_{\FHsc}> \norm{du_0}_{\FHsc} \ge \ell_c,
\]
which completes the proof.
\end{proof}

\section{Proof of Theorem \ref{thm:main3}}\label{sec:7}
For the proof of Theorem \ref{thm:main3}, let us establish a version of small data scattering.
This follows from the long-time perturbation.
Theorem \ref{thm:main3} immediately follows from this proposition and Proposition \ref{prop:ods}.
\begin{proposition}\label{prop:sds2}
Let $u_0 \in \FHsc$ and let $u(t)$ be a corresponding unique solution given
in Theorem \ref{thm:lwp}. Then, for any $M>0$, there exists a constant $\eta=\eta(M)>0$
such that if $\norm{u_0}_{\FHsc} \le M$ and 
$\norm{U(t)u_0}_{L(\R)} \le \eta$ then $u_0 \in S$.
\end{proposition}
\begin{proof}
We apply Theorem \ref{thm:lpt} with 
$t_0=0$, $I=\R$, and $\widetilde{u}(t)=U(t)u_0$.
Then, $e(t)=|U(t)u_0|^{p-1} U(t)u_0$.
By Strichartz' estimate,
\[
	\norm{\widetilde{u}}_{L^\I(\R,\dot{M}^{s_c}_{2,2})} + \norm{\widetilde{u}}_{W(\R)}
	\le C \norm{u_0}_{\FHsc} \le CM.
\]
Since $\widetilde{u}(0)=u_0$, we have $\norm{u(0)-\widetilde{u}_0}_{\FHsc} = 0$
and $\norm{U(t)(u(0)-\widetilde{u}_0)}_{W(\R)}=0$.
One verifies from  Lemma \ref{lem:lpt_nonlinearest1}, Strichartz' estimate, and the assumption
that
\[
	\norm{e}_{F(\R)}
	\le C \norm{U(t)u_0}_{L(\R)}^{p-1} \norm{U(t)u_0}_{W_1(\R)}
	\le C \eta^{p-1} M.
\]
Hence, if we take $\eta$ small, we have $\norm{e}_{F(\R)} \le \eps_1$,
where $\eps_1=\eps_1(M)$ is the number given in Theorem \ref{thm:lpt}.
Then, Theorem \ref{thm:lpt} yields
\[
	\norm{u-\widetilde{u}}_{W(\R)} \le C \eps_1^\beta <\I,
\]
which implies $\norm{u}_{W(\R)}<\I$.
Thus, we conclude from Proposition \ref{prop:nscond} that $u_0\in S$  
\end{proof}
\begin{remark}\label{rmk:odsFH1}
Since scattering in $\FHsc$ and scattering in $\F H^1$ is equivalent for 
$\F H^1$-solutions (see Remark \ref{rmk:nscondFH1}),
if $u_0 \in \F H^1$ is added to the assumption of the above proposition
then the result holds true with replacing the meaning of scattering by $\F H^1$ sense.
This is the reason why Theorem \ref{thm:main3} holds for  $\F H^1$ solutions.
\end{remark}

\section{Proof of Theorem \ref{thm:main1}}\label{sec:5}
\begin{proof}[Proof of Theorem \ref{thm:main1}]
By the time symmetry, we note that
\[
	\ell_{c}
	= \inf\{ \norm{f}_{\FHsc} \ |\ f \in \FHsc \setminus S_+ \}.
\]
Let $\{u_{0,n}\}_n \subset \FHsc$ be a sequence satisfying
$u_{0,n} \not\in S_+$ and $\norm{u_{0,n}}_{\FHsc} \le \ell_c + \frac1n$.
We shall show that there exist a subsequence of $\{u_{0,n}\}_n$,
which is denoted again by $\{u_{0,n}\}_n$,
a function $\psi \in  \FHsc$ with $\norm{\psi}_{\FHsc}=\ell_c$,
and sequences $\{W_n\}_{n}\subset \FHsc$, $\{ h_n \} \subset \R_+$,
and $\{\xi_n\}_{n} \subset \R^N$ such that
\begin{equation}\label{eq:cc1}
	u_{0,n} = e^{i\xi_n\cdot x} \psi_{\{h_n\}} + W_n
\end{equation}
and
\begin{equation}\label{eq:cc2}
	\norm{W_n}_{\FHsc} \to 0
\end{equation}
as $n\to\I$.
We first apply the profile decomposition (Proposition \ref{prop:pd}) to $\{u_{0,n}\}$.
Then, there exist a subsequence of $\{u_{0,n}\}$, which is denoted again by $\{u_{0,n}\}$,
and sequences
$\psi^j \in \FHsc$, $h_n^j>0$, $\xi_n^j$, and $W^l_n$
 such that for every $l\ge1$
\begin{equation}\label{eq:critpf1}
	u_{0,n} = \sum_{j=1}^l e^{ix\xi_n^j} \psi^j_{\{ h_n^j \}}
	+ W^l_n
\end{equation}
in $\FHsc$ and
\begin{equation}\label{eq:critpf3}
	\norm{u_{0,n}}_{\FHsc}^2 -\sum_{j=1}^l \norm{\psi^j}_{\FHsc}^2
	- \norm{W^j_n}_{\FHsc}^2 \to 0
\end{equation}
as $n\to\I$.
Moreover, for all $l\ge1$ and $i\neq j$,
\begin{equation}\label{eq:critpf4}
	\lim_{n\to\I}\(
	\frac{h_n^i}{h_n^j} + \frac{h_n^j}{h_n^i} + \frac{|\xi^i_n-\xi^j_n|}{h_n^j}\)= \I.
\end{equation}
Furthermore, there exists a sequence $l_k$ with $l_k\to\I$ as $k\to\I$ and
\begin{equation}\label{eq:critpf5}
	\lim_{k\to\I} \limsup_{n\to\I} \norm{U(t)W^{l_k}_n}_{L^{\rho,\I}(\R,L^q)} =0
\end{equation}
for any $1<\rho,q<\I$ with
$\delta(q)-s_c\in (0,\min(1,N/2))$ and $2/\rho-\delta(q)=s_c$.
Property \eqref{eq:critpf3} implies that
\begin{equation}\label{eq:critpf35}
	\sum_{j=1}^\I \norm{\psi^j}_{\FHsc}^2 \le \ell_c,
\end{equation}
from which $\norm{\psi^j}_{\FHsc}\le \ell_{c}$ holds for all $j\ge1$. 
We now claim that there exists $j_0$ such that 
\begin{equation}\label{eq:pf_claim}
	\norm{\psi^{j_0}}_{\FHsc}=\ell_c.
\end{equation}
This claim completes the proof of \eqref{eq:cc1} and \eqref{eq:cc2}.
Indeed, if such $j_0$ exists then we deduce from
 \eqref{eq:critpf3} that $\psi^j\equiv0$ for all $j\neq j_0$.
We hence obtain \eqref{eq:cc1} with $\psi = \psi^{j_0}$, $\xi = \xi^{j_0}$, $h_n = h_n^{j_0}$ and $W_n=W_n^{j_0}$.
Further, the property \eqref{eq:cc2} immediately follows from \eqref{eq:critpf3}.

Let us show the claim \eqref{eq:pf_claim}.
Assume for contradiction that 
$\norm{\psi^j}_{\FHsc} <\ell_c$ for all $j$.
Let $\Psi_j$ be a unique solution of \eqref{eq:NLS} with $\Psi_j|_{t=0}=\psi^j$.
Since $\psi^j \in S_+$ by definition of $\ell_{c}$, the maximal interval of $\Psi_j$ 
contains $[0,\I)$ and $\norm{\Psi_j}_{W(\R_+)}<\I$ by Proposition \ref{prop:nscond}.
Further $\norm{\Psi_j}_{L^\I(\R_+,\Ms{s_c}22)}<\I$ easily follows from the convergence of $U(-t)\Psi_j(t)$ 
in $\FHsc$ as $t\to\I$.
Let
\[
	\widetilde{v}_n^l(t,x) := \sum_{j=1}^l (\Psi_j)_{[h_n^j,\xi_n^j]}(t, x) 
\]
and
\[
	\widetilde{u}_n^l(t,x) := \widetilde{v}_n^l(t,x) + U(t)W_n^l,
\]
where 
\[
	f_{[h,\xi]}(t,x) = h^{\frac2{p-1}} f(h^2t,h(x-2t\xi)) e^{i\xi\cdot x} e^{-it|\xi|^2}.
\]
Remark that  $\Psi_j$ are defined on $[0,\I)\times \R^N$ and so are
$\widetilde{v}_n^l$ and $\widetilde{u}_n^l$. 
It then follows that
\begin{align*}
	(i\d_t + \Delta) \widetilde{u}^{l}_n 
	={}& \sum_{j=1}^l (i\d_t + \Delta)(\Psi_j)_{[h_n^j,\xi_n^j]}  \\
	={}& -\sum_{j=1}^l |(\Psi_j)_{[h_n^j,\xi_n^j]}|^{p-1} (\Psi_j)_{[h_n^j,\xi_n^j]}  .
\end{align*}
We also let
\[
	\widetilde{e}^l_n := (i\d_t + \Delta) \widetilde{u}_n^l
	+ |\widetilde{u}_n^l|^{p-1} \widetilde{u}_n^l.
\]
We need the following two estimates;
\begin{lemma}\label{lem:pf_bdd_v}
For any $\eps>0$, there exists $l_0=l_0(\eps)$ such that
\[
	\limsup_{n\to\I} \norm{\widetilde{v}^l_n -\widetilde{v}^{l_0}_n}_{W(\R_+)\cap L^\I(\R_+,\Ms{s_c}22)} \le \eps
\]
for any $l> l_0$.
\end{lemma}
\begin{lemma}\label{lem:pf_error2}
Let $l_k$ be a sequence such that \eqref{eq:critpf5} holds.
It holds that
\[
	\limsup_{n\to\I} \norm{ \widetilde{e}_n^{l_k} }_{F(\R_+)} \to 0
\]
as $k\to\I$.
\end{lemma}
The proofs of these two lemmas will be given later.
We shall show that they give the desired conclusion.
Using Lemma \ref{lem:pf_bdd_v} with $\eps=1$, we see that
there exists $l_0$ such that
\begin{equation}\label{eq:critpf9}
\begin{aligned}
	\norm{\widetilde{u}_n^l}_{L^\I(\R_+,\Ms{s_c}22)}
	\le{}&\norm{\widetilde{u}_n^{l_0}}_{L^\I(\R_+,\Ms{s_c}22)}
	+ \norm{\widetilde{v}_n^{l}-\widetilde{v}_n^{l_0}}_{L^\I(\R_+,\Ms{s_c}22)} \\
	\le{}&\sum_{j=1}^{l_0} \norm{\Psi_j}_{L^\I(\R_+,\Ms{s_c}22)}+
	\norm{U(t)W_n^l}_{L^\I(\R_+,\Ms{s_c}22)} + 1\\
	\le{}&\sum_{j=1}^{l_0} \norm{\Psi_j}_{L^\I(\R_+,\Ms{s_c}22)}+
	\ell_c + 1=:A
\end{aligned}
\end{equation}
for any $l>l_0$ and $n\ge1$. Remark that $A$ is independent of $n$ and $l$.
Similarly, by \ref{lem:pf_bdd_v} and Strichartz' estimate,
\begin{equation}\label{eq:critpf10}
	\norm{\widetilde{u}_n^l}_{W(\R_+)} \le \sum_{j=1}^{l_0} \norm{\Psi_j}_{W(\R_+)}+
	C\ell_c + 1=: M
\end{equation}
for any $l>l_0$ and $n\ge 1$.
Further, since $u_{0,n}=\widetilde{u}_n^l(0)$, 
\begin{equation}\label{eq:critpf11}
	\norm{u_{0,n}-\widetilde{u}_n^l}_{\Ms{s_c}22} =0 \le A'
\end{equation}
for any $A'\ge0$. In particular, we let $A'=1$.
Let $\eps_1=\eps_1(A',M)$ be a number given by the long-time perturbation theory
(Theorem \ref{thm:lpt}).
Then,
\begin{equation}\label{eq:critpf12}
	\norm{e^{it\Delta} (u_{0,n} - \widetilde{u}_n^l(0))}_{W(\R_+)}
	= 0 \le \eps_1.
\end{equation}
We deduce from Lemma \ref{lem:pf_error2} that there exists $k_0$ such that
\[
	\limsup_{n\to\I} \norm{\widetilde{e}^{l_k}_n}_{F(\R_+)} \le \frac{\eps_1}2.
\]
for $k\ge k_0$. Choose $k\ge k_0$ so that $l_{k}\ge l_0$.
There exists $n_0$ such that
\begin{equation}\label{eq:critpf13}
	\norm{\widetilde{e}^{l_k}_{n_0}}_{F(\R_+)} \le {\eps_1}.
\end{equation}
By \eqref{eq:critpf9}, \eqref{eq:critpf10}, \eqref{eq:critpf11}, \eqref{eq:critpf12},
\eqref{eq:critpf13}, we deduce from Theorem \ref{thm:lpt} that
a solution $u_{n_0}$ of \eqref{eq:NLS} with $u_{n_0}=u_{0,n_0}$ satisfies
\[
	\norm{u_{n_0}}_{W(\R_+)} < \I,
\]
which implies $u_{0,n} \in S_+$.
However, this contradicts with the definition of $u_{0,n}$. 
Hence, the claim follows and \eqref{eq:cc1} and \eqref{eq:cc2} are 
established.

We next show $\psi \not\in S_+$. Suppose $\psi\in S_+$ for contradiction.
It then follows that
$\norm{\Psi}_{L^\I(\R_+,\Ms{s_c}22)} + \norm{\Psi}_{W(\R_+)}<\I$, where $\Psi$ is a solution of 
\eqref{eq:NLS} with $\Psi(0)=\psi$.
Then, $\Psi_{[h_n,\xi_n]}$ is a solution of \eqref{eq:NLS}
with $\Psi_{[h_n,\xi_n]}(0)=e^{ix\xi}\psi_{\{ h_n \}}$.
Now, we shall apply Theorem \ref{thm:lpt} with $\widetilde{u}_n(t,x)=
 \Psi_{[h_n,\xi_n]} + U(t)W_{n}$.
Notice that 
\begin{equation}\label{eq:critpf14}
\begin{aligned}
	&\norm{\Psi_{[h_n,\xi_n]}}_{L^\I(\R_+,\Ms{s_c}22)} =
	\norm{\Psi}_{L^\I(\R_+,\Ms{s_c}22)} <\I,\\
	&\norm{\Psi_{[h_n,\xi_n]}}_{W(\R_+)}= \norm{\Psi}_{W(\R_+)}<\I.
\end{aligned}
\end{equation}
for any $n$. Since $\norm{U(t)W_n}_{L^\I(\R_+,\Ms{s_c}22)}+
\norm{U(t)W_n}_{W(\R_+)}\le C\norm{W_n}_{\FHsc}$ is bounded,
there exists $\overline{A}$ and $\overline{M}$ 
such that
\begin{equation}
	\norm{\widetilde{u}_n}_{L^\I(\R_+,\Ms{s_c}22)} \le \overline{A},\quad
	\norm{\widetilde{u}_n}_{W(\R_+)} \le \overline{M}.
\end{equation}
Set $e:=(i\d_t + \Delta)\widetilde{u}+|\widetilde{u}|^{p-1}\widetilde{u}e=|\widetilde{u}|^{p-1}\widetilde{u}
-|\Psi_{[h_n,\xi_n]}|^{p-1} \Psi_{[h_n,\xi_n]}$. 
Then, by \eqref{eq:lpt_nlest2}, Strichartz' estimate, \eqref{eq:cc2}, and \eqref{eq:critpf14},
\[
	\norm{e}_{F(\R_+)}
	\le C \norm{W_n}_{\FHsc}(\norm{W_n}_{\FHsc}+\norm{\Psi}_{W(\R_+)})^{p-1}\to0
\]
as $n\to\I$.
Further $u_{0,n}=\widetilde{u}_n(0)$.
Hence, applying Theorem \ref{thm:lpt},
we see $u_{0,n} \in S_+$ for large $n$,
which is a contradiction.
Thus, $\psi\in \FHsc \setminus S_+$.
We therefore see that
the function $\psi$ possesses the all desired properties.
\end{proof}
To prove Lemmas \ref{lem:pf_bdd_v} and \ref{lem:pf_error2},
we first establish the following.
\begin{lemma}\label{lem:pf_error1}
Let $\{h_n^j\}_{(j,n) \in [1,l]\times \Z_+} \subset \R_+$
and $\{\xi_n^j\}_{(j,n) \in [1,l]\times \Z_+} \subset \R^N$ satisfy \eqref{eq:pd2}.
Let 
 $\Psi_j(t,x) \in W(\R_+)$. 
Set a complex valued function $F(z)=|z|^{p-1}z$
and
\[
	e(t,x) := F\(\sum_{j=1}^l (\Psi_j)_{[h_n^j,\xi_n^j]}\)- \sum_{j=1}^l F\( (\Psi_j)_{[h_n^j,\xi_n^j]} \)
\]
Then,
\[
	\norm{e}_{F(\R_+)} \to 0
\]
as $n\to \I$.
\end{lemma}
\begin{proof}
By Proposition \ref{prop:LM_approx} and \eqref{eq:lpt_nlest2},
we only have to consider the case where $\supp \Psi_j \subset [m,M] \times B_R(0)$
holds for some $m,M,R>0$ and for all $j\in [1,l]$, where $B_r(y)$ denotes 
a ball in $\R^N$ with radius $r>0$ and center $y\in \R^N$.
We use the difference form of the $F(I)$-norm, i.e.\ \eqref{eq:Mn_alt} and  Lemma \ref{lem:B_equivalence}.
Observe that
\begin{equation}\label{eq:lem_e_1}
	\abs{\delta (M(-t)e)} 
	\le \sum_{1\le j,k\le l,\,j\neq k}\abs{\delta (M(-t)(\Psi_j)_{[h_n^j,\xi_n^j]})}
	\left[ (\Psi_k)_{[h_n^k,\xi_n^k]} \right]^{p-1},
\end{equation}
where $\delta f(x)=f(x+a)-f(x)$ and $[f](x)=|f(x)|+|f(x+a)|$ for some $a\in \R^N$.
If $\frac{h_n^j}{h_n^k}+ \frac{h_n^k}{h_n^j} \ge \frac{m}M+\frac{M}m $ then
time supports of ${\delta (\Psi_j)_{[h_n^j,\xi_n^j]}}$ and 
$[ (\Psi_k)_{[h_n^k,\xi_n^k]} ]$ do not intersect and hence
\[
	\sup_{a\in \R^N} \norm{\abs{\delta (M(-t)(\Psi_j)_{[h_n^j,\xi_n^j]})}
	\left[ (\Psi_k)_{[h_n^k,\xi_n^k]} \right]^{p-1}}_{L^{q(F)}} =0
\]
holds for such $n$ and $(j,k)$.
Therefore, it suffices to prove under an additional
assumption $\sup_n (\frac{h_n^j}{h_n^k}+ \frac{h_n^k}{h_n^j}) \le \frac{m}M+\frac{M}m$
for all $1\le j,k\le l$
since the whole estimate can be decomposed into finite number of such estimates.
To do so, we may let, by changing notations if necessary,
$h_n^j\equiv h_n$ and $\supp \Psi_j \subset [m',M']\times B_R(0)$ for all $j\in[1,l]$,
where $m'=m^2/M$ and $M'=M^2/m$.

Let $\iota_0\in \N$ to be chosen later.
If $\iota\in \Z$ satisfies $-\iota+\log_2 h_n \le \iota_0$
and if $t(h_n)^2 \in [m',M']$,
then it follows from \eqref{eq:pd2} that
\begin{align*}
	&\abs{h_n(x-2t\xi_n^j) - h_n(x-2t\xi_n^k)} \ge 2m' \abs{\frac{\xi^j_n-\xi^k_n}{h_n}}
	\to \I,\\
	&\abs{h_n(x+a-2t\xi_n^j) - h_n(x-2t\xi_n^k)} \ge 2m' \abs{\frac{\xi^j_n-\xi^k_n}{h_n}}-C2^{\iota_0}
	\to \I
\end{align*}
as $n\to\I$ for any $|a|\le 2^{-\iota}$.
Since the spatial support of $\Psi_j$ is contained in $B_R(0)$,
there exists $n_0=n_0(\iota_0)$ that
if $-\iota+\log_2 h_n\le \iota_0$ and $n\ge n_0$ then
\[
	\sum_{1\le j,k\le l,\,j\neq k}\abs{\delta (M(-t)(\Psi_j)_{[h_n^j,\xi_n^j]})}
	\left[ (\Psi_k)_{[h_n^k,\xi_n^k]} \right]^{p-1} =0
\]
for any $(t,x) \in \R_+ \times \R^N$ and
for any $a \in \R^N$ with $|a|\le 2^{-\iota}$.
We conclude that
\begin{equation}\label{eq:lem_e_2}
	\norm{ 2^{\iota s_c} \sup_{|a| \le 2^{-\iota}} \norm{\delta (M(-t)e(t))}_{L^{q(F)}} }_{\ell^2_\iota(\{\iota-\log_2 h_n\ge -\iota_0\})} =0
\end{equation}
for $n\ge n_0$.
Consider the opposite case $-\iota +\log_2h_n\ge \iota_0$.
By scaling, 
$$\norm{(\Psi_j)_{[h_n,\xi_n^j]}(t)}_{L^{q(L)}}=
h_n^{\frac2{p-1}-\frac{N}{q(L)}} \norm{\Psi_j(h_n^2t)}_{L^{q(L)}}$$
and
\begin{multline*}
	\norm{\delta_a (M(-t)(\Psi_j)_{[h_n,\xi_n^j]}(t))}_{L^{q(W_1)}}\\
	=h_n^{\frac2{p-1}-\frac{N}{q(W_1)}} 
	\norm{\delta_{h_na} (M(-h_n^2 t)\Psi_j(h_n^2t))}_{L^{q(W_1)}}
\end{multline*}
hold. 
It therefore follows that
\begin{multline}\label{eq:lem_e_3}
	\norm{ 2^{\iota s_c} \sup_{|a| \le 2^{-\iota}} \norm{\delta (M(-t)e(t))}_{L^{q(F)}} }_{\ell^2_\iota(\{\iota-\log_2 h_n\le -\iota_0\})}\\
	\le Ch_n^{\frac{2p}{p-1}-\frac{N(p-1)}{q(L)}- \frac{N}{q(W_1)}} \(\sum_{j=1}^l \Lebn{\Psi_j(h_n^2t) }{q(L)} \)^{p-1}\\
	\times \sum_{j=1}^l \norm{ 2^{\iota s_c} \sup_{|b| \le h_n 2^{-\iota}}
	\norm{\delta_b (M(-h_n^2 t)\Psi_j(h_n^2t))}_{L^{q(W_1)}} }_{\ell^2_\iota(\{\iota-\log_2 h_n\le -\iota_0\})}.
\end{multline}
Introduce a new variable $\iota'=\iota-k$,
where $k$ is a unique integer such that $k\le \log_2 h_n <k+1$.
Then, $\iota-\log_2 h_n\le -\iota_0$ is equivalent to $\iota' \le -\iota_0$
and so
\begin{align*}
	&{}\norm{ 2^{\iota s_c} \sup_{|b| \le h_n 2^{-\iota}}
	\norm{\delta_b (M(-h_n^2 t)\Psi_1(h_n^2t))}_{L^{q(W_1)}} }_{\ell^2_\iota(\{\iota-\log_2 h_n\le -\iota_0\})}\\
	={}& \norm{ 2^{(\iota'+k) s_c} \sup_{|b| \le h_n 2^{-\iota'-k}}
	\norm{\delta_b (M(-h_n^2 t)\Psi_1(h_n^2t))}_{L^{q(W_1)}} }_{\ell^2_{\iota'}(\{\iota'\le-\iota_0\})}\\
	\le{}&
	h_n^{s_c}\norm{ 2^{\iota's_c} \sup_{|b| \le 2^{-\iota'+1}}
	\norm{\delta_b (M(-h_n^2 t)\Psi_1(h_n^2t))}_{L^{q(W_1)}} }_{\ell^2_{\iota'}(\{\iota'\le-\iota_0\})},
\end{align*}
where we have used inequalities $2^{ks_c}\le h_n^{s_c}$ and $h_n 2^{-k} \le 2$ to
deduce the last line. Denoting $\iota'-1$ again by $\iota$, we obtain
\begin{equation}\label{eq:lem_e_4}
\begin{aligned}
	&{}\norm{ 2^{\iota s_c} \sup_{|b| \le h_n 2^{-\iota}}
	\norm{\delta_b (M(-h_n^2 t)\Psi_1(h_n^2t))}_{L^{q(W_1)}} }_{\ell^2_\iota(\{\iota-\log_2 h_n\le -\iota_0\})}\\
	\le{}&
	Ch_n^{s_c}\norm{ 2^{\iota s_c} \sup_{|b| \le 2^{-\iota}}
	\norm{\delta_b (M(-h_n^2 t)\Psi_1(h_n^2t))}_{L^{q(W_1)}} }_{\ell^2_{\iota}(\{\iota\le-\iota_0+1\})}.
\end{aligned}
\end{equation}
The same estimate holds for $\Psi_j$ ($2\le j \le l$).
Combining \eqref{eq:lem_e_1}, \eqref{eq:lem_e_2}, \eqref{eq:lem_e_3}, and
\eqref{eq:lem_e_4}, and using the identity
\[
	\frac{2p}{p-1}-\frac{N(p-1)}{q(L)}- \frac{N}{q(W_1)}=\frac{2}{\rho(F)}+s_c,
\]
one deduces that
\[
	\Mn{e}{s_c}{q(F)}2t \le C h_n^{\frac{2}{\rho(F)}}|h_n^2t|^{s_c} \Gamma(h_n^2 t)
	\(\sum_{j=1}^l \Lebn{\Psi_j(h_n^2t) }{q(L)} \)^{p-1} 
\]
for $n\ge n_0$, where
\begin{equation}\label{eq:lem_e_6}
	\Gamma(t) := \sum_{j=1}^l
	\norm{ 2^{\iota s_c} \sup_{|a| \le 2^{-\iota}}
	\norm{\delta (M(- t)\Psi_j(t))}_{L^{q(W_1)}} }_{\ell^2_{\iota}(\{\iota\le-\iota_0+1\})}.
\end{equation}
Take $L^{\rho(F),2}(\R_+)$-norm of the both sides.
By scaling and by
the generalized H\"older inequality, we conclude that
\[
	\limsup_{n\to\I} \norm{e}_{F(\R_+)}
	\le C\norm{|t|^{s_c} \Gamma}_{L^{\rho(W_1),2}(\R_+)}
\(\sum_{j=1}^l \norm{\Psi_j}_{L(\R_+)} \)^{p-1}.
\]
Since $\norm{\Psi_j}_{L(\R_+)}<\I$ by embedding $W_2(\R_+) \hookrightarrow L(\R_+)$,
the proof is now reduced to showing that
\begin{equation}\label{eq:lem_e_7}
	\norm{|t|^{s_c} \Gamma}_{L^{\rho(W_1),2}(\R_+)} \to 0
\end{equation}
as $\iota_0\to\I$.

Let us prove \eqref{eq:lem_e_7}.
By assumption on $\Psi_j$, the support of $\Gamma(t)$ is included in $[m',M']$.
Moreover, thanks to \eqref{eq:Mn_alt}, and Lemma \ref{lem:B_equivalence}, for almost all $t\in[m',M']$,
\begin{align*}
	|t|^{s_c}\norm{ 2^{\iota s_c} \sup_{|a| \le 2^{-\iota}}
	\norm{\delta (M(-t)\Psi_1(t))}_{L^{q(W_1)}} }_{\ell^2(\Z)}
	\le C \Mn{\Psi_1(t)}{s_c}{q(W_1)}2t <\I
\end{align*}
since $\Psi_1 \in W(\R_+)$.
Hence, for almost all $t\in \R_+$, one deduces
\begin{align*}
	|t|^{s_c}\norm{ 2^{\iota s_c} \sup_{|a| \le 2^{-\iota}}
	\norm{\delta (M(-t)\Psi_1(t))}_{L^{q(W_1)}} }_{\ell^2(\{\iota\le -\iota_0+1\})}
	\to 0
\end{align*}
as $\iota_0\to\I$. The same holds for $\Psi_j$ ($2\le j \le l$).
We hence obtain
\[
		|t|^{s_c} \Gamma(t) \to 0
\]
as $\iota_0\to\I$ for almost all $t\in \R_+$.
Further, $\Mn{\Psi_j(t)}{s_c}{q(W_1)}2t \in L^{q(W_1)}(\R_+)$
follows from the fact that $\Psi_j \in W_1(\R_+)$ and $\supp \Psi_j \in [m',M']$.
Therefore,
\[
	|t|^{s_c} \Gamma(t) \le C \sum_{j=1}^l \Mn{\Psi_j(t)}{s_c}{q(W_1)}2t \in L^{q(W_1)}(\R_+).
\]
Lebesgue's convergence theorem then gives us the desired limit \eqref{eq:lem_e_7}.
\end{proof}

\begin{proof}[Proof of Lemma \ref{lem:pf_bdd_v}]
Set $z_n^{l,l_0}:=\widetilde{v}^l_n -\widetilde{v}^{l_0}_n=\sum_{j=l_0+1}^l (\Psi_j)_{[h_n^j,\xi_n^j]}$.
We first note that
\[
	z_n^{l,l_0}(0) = \sum_{j=l_0+1}^l e^{ix\xi_n^j} \psi^j_{\{ h_n^j \}}.
\]
By \eqref{eq:critpf4}, for any $l>l_0\ge1$,
\[
	\lim_{n\to\I} \norm{z_n^{l,l_0}(0)}
	= \( \sum_{j=l_0+1}^l \norm{\psi^j}_{\FHsc}^2 \)^{1/2}.
\]
Here, $( \sum_{j=1}^\I \norm{\psi^j}_{\FHsc}^2 )^{1/2}<\I$ by \eqref{eq:critpf3}. Hence, for any $\eps>0$ there exists $l_0$ such that
\begin{equation}\label{eq:pf_bdd_v1}
	\lim_{n\to\I} \norm{z_n^{l,l_0}(0)} \le \eps
\end{equation}
for any $l>l_0$.
Now,
\[
	(i\d_t + \Delta)z_n^{l,l_0} = -\sum_{j=l_0+1}^l 
	F((\Psi_j)_{[h_n^j,\xi_n^j]}).
\]
Set $e_n^{l,l_0}:=F(\sum_{j=l_0+1}^l (\Psi_j)_{[h_n^j,\xi_n^j]})- \sum_{j=l_0+1}^l F((\Psi_j)_{[h_n^j,\xi_n^j]})$. Then,
$z_n^{l,l_0}$ satisfies the following equation
\[
	(i\d_t + \Delta)z_n^{l,l_0}+ F(z_n^{l,l_0}) = e_n^{l,l_0}.
\]
Write in an integral form
\[
	z_n^{l,l_0}(t) = U(t) z_n^{l,l_0}(0)
	- i \int_0^t U(t-s) F(z_n^{l,l_0}(s)) ds
	+ i \int_0^t U(t-s) e_n^{l,l_0}(s) ds.
\]
Using Strichartz" estimate and nonlinear estimate \eqref{eq:lpt_nlest1},
\[
	\norm{z_n^{l,l_0}}_{W(\R_+)}
	\le C \norm{z_n^{l,l_0}(0)}_{\FHsc} + C \norm{z_n^{l,l_0}}_{W(\R_+)}^p
	+ C\norm{e_n^{l,l_0}}_{F(\R_+)}.
\]
We now use Lemma \ref{lem:pf_error1} to obtain
$\lim_{n\to\I}\norm{e_n^{l,l_0}}_{F(\R_+)}=0$ for any $l>l_0$.
Then, together with \eqref{eq:pf_bdd_v1},
\begin{equation}\label{eq:pf_bdd_v2}
	\norm{z_n^{l,l_0}}_{W(\R_+)}
	\le C\eps+ C \norm{z_n^{l,l_0}}_{W(\R_+)}^p
\end{equation}
for any $l>l_0$ and $n\ge n_0(l,l_0,\eps)$.
One easily sees that there exists $\eps_0>0$ such that 
if $0<\eps \le \eps_0$ then the inequality
\eqref{eq:pf_bdd_v2} implies
\[
	\norm{z_n^{l,l_0}}_{W(\R_+)}
	\le 2C\eps,
\]
which completes the proof.
\end{proof}
\begin{proof}[Proof of Lemma \ref{lem:pf_error2}]
Set $F(z)=|z|^{p-1}z$.
By definition of $\widetilde{e}_n^l$,
\begin{align*}
	\norm{ \widetilde{e}_n^l }_{F(\R_+)}
	={}& \norm{F(\widetilde{v}^{l}_n+ U(t) W_n^l)-  \sum_{j=1}^l F((\Psi_j)_{[h_n^j,\xi_n^j]}) }_{F(\R_+)} \\
	\le{}& \norm{F(\widetilde{v}^{l_0}_n+ U(t) W_n^l)-  F(\widetilde{v}^{l_0}_n) }_{F(\R_+)} \\ 
	&{}+\norm{F(\widetilde{v}^{l}_n+ U(t) W_n^l) - F(\widetilde{v}^{l_0}_n+ U(t) W_n^l) }_{F(\R_+)}\\ 
	&{}+\norm{ F(\widetilde{v}^l_n) - F(\widetilde{v}^{l_0}_n) }_{F(\R_+)}\\ 
	&{}+\norm{ F(\widetilde{v}^l_n) - \sum_{j=1}^l F((\Psi_j)_{[h_n^j,\xi_n^j]})}_{F(\R_+)}.
\end{align*}
By Lemma \ref{lem:pf_error1}, the last term of the right hand side
tends to zero as $n\to\I$ for all $l$.
Moreover, the second and the third terms become small if we take $l_0$ sufficiently large
in light of Lemma \ref{lem:pf_bdd_v} and \eqref{eq:lpt_nlest2}. 

Thus we shall estimate the first term.
Without loss of generality, we may assume that
$\supp \Psi_j \subset [m,M] \times B_0(R)$
holds for some $m,M,R>0$ and for all $j\in [1,l_0]$.
Set
\[
	I_n^j = [(h_n^j)^{-2}m,(h_n^j)^{-2}M].
\]
Note that if $t\not\in I_n^j$ then $(\Psi_j)_{[h_n^j,\xi_n^j]}(t,x)=0$.
Hence,
\[
	\norm{F(\widetilde{v}^{l_0}_n+ U(t) W_n^l)-  F(\widetilde{v}^{l_0}_n) }_{
	F(\R_+\setminus (\bigcup_{j=1}^{l_0}I_n^j) )}
	=\norm{F(U(t) W_n^l)}_{
	F(\R_+\setminus (\bigcup_{j=1}^{l_0}I_n^j) )}.
\]
By \eqref{eq:lpt_nlest1} and \eqref{eq:critpf5},
\[
	\limsup_{n\to\I}\norm{F(U(t) W_n^l)}_{
	F(\R_+\setminus (\bigcup_{j=1}^{l_0}I_n^j) )}
	\le
	\limsup_{n\to\I}\norm{F(U(t) W_n^l)}_{
	F(\R_+)}
	\to0
\]
as $l\to\I$.
We shall consider
\[
	\norm{F(\widetilde{v}^{l_0}_n+ U(t) W_n^l)-  F(\widetilde{v}^{l_0}_n) }_{
	F(\bigcup_{j=1}^{l_0}I_n^j )}.
\]
However, we only have treat the case where 
$\frac{h_n^i}{h_n^j} + \frac{h_n^j}{h_n^i}$ is bounded for any
$n$ and $1\le i,j \le l_0$.
This is because we can decompose
the whole estimate into finite number of such estimates
by using the fact that $\frac{h_n^i}{h_n^j} + \frac{h_n^j}{h_n^i} \ge \frac{m}M+\frac{M}m$ implies
$I^i_n\cap I_n^j=\emptyset$.
Changing scales (and notations if necessary), we may further let $h_n^i\equiv1$.
Then, orthogonality \eqref{eq:critpf4} is simply
 $|\xi_n^i-\xi_n^j|\to\I$ as $n\to\I$ for any $i\neq j$.
What we want to estimate is
\begin{equation*}
	\norm{F(\widetilde{v}^{l_0}_n+ U(t) W_n^l)-  F(\widetilde{v}^{l_0}_n) }_{
	F([m',M'])},
\end{equation*}
where $m'=m^2/M$ and $M'=M^2/m$.
 Let $\chi(x) \in C^\I_0(\R^N)$ be a nonnegative smooth radial cut-off function such that
$\chi\equiv 1$ on $B_R(0)$ and $\supp \chi \subset B_{2R}(0)$.
Let $\chi_n^j(t,x) = {\bf 1}_{[m',M']}(t) \chi(x-t\xi_n^j)$.
Recall that
\[
	{\Psi_j}_{[1,\xi_n^j]}(t,x) = \Psi_j(t,x-t\xi_n^jx) 
	e^{i\xi_n^j \cdot x} e^{-it|\xi_n^j|^2}.
\]
Hence,
\[
	\supp {\Psi_j}_{[1,\xi_n^j]}(t,x) \subset  \supp \chi_n^j (t,x) 
	\subset  \bigcup_{m'\le t \le M'} \(\{t\} \times B_{2R} (t\xi_n^j) \)=: \Sigma_n^j.
\]
By the orthogonality condition, $\Sigma_n^j$ ($1\le j \le l_0$) are mutually disjoint for large $n$.
For such $n$, we have
\[
	\chi_n^k{\Psi_j}_{[1,\xi_n^j]}=
	\begin{cases}
	{\Psi_j}_{[1,\xi_n^j]} & k=j\\
	0 & k\neq j.
	\end{cases}
\]
Set $\widetilde{\chi}_n^j=(\chi_n^j)^p$. 
Then, 
\begin{align*}
	\widetilde{\chi}_n^j F(\widetilde{v}^{l_0}_n+ U(t) W_n^l)
	&{}= F(\chi_n^j \widetilde{v}^{l_0}_n+ \chi_n^j U(t) W_n^l)\\
	&{}= F({\Psi_j}_{[1,\xi_n^j]} + \chi_n^j U(t) W_n^l)
\end{align*}
for any $j\in [1,l_0]$, provided $n$ is sufficiently large.
Similarly, $\widetilde{\chi}_n^j F(\widetilde{v}^{l_0}_n)
	= F({\Psi_j}_{[1,\xi_n^j]})$ for large $n$.
Further, one easily verifies that
$1-\sum_{j=1}^l \widetilde{\chi}_n^j \equiv 0$
on $\cup_{j=1}^{l_0} \supp {\Psi_j}_{[1,\xi_n^j]}(t,x)$.
Therefore,
\[
	\(1-\sum_{j=1}^l \widetilde{\chi}_n^j\)
	F(\widetilde{v}^{l_0}_n+  U(t) W_n^l)
	= \(1-\sum_{j=1}^l \widetilde{\chi}_n^j\) F( U(t) W_n^l )
\]
and $(1-\sum_{j=1}^l \widetilde{\chi}_n^j)F(\widetilde{v}^{l_0}_n)\equiv0$.
Thus, for large $n$, we have
\begin{equation*}
\begin{aligned}
	&{}\norm{F(\widetilde{v}^{l_0}_n+ U(t) W_n^l)-  F(\widetilde{v}^{l_0}_n) }_{
	F([m',M'])} \\
	\le{}& \sum_{j=1}^{l_0} \norm{
	F({\Psi_j}_{[1,\xi_n^j]} + \chi_n^j U(t) W_n^l) - F({\Psi_j}_{[1,\xi_n^j]})}_{
	F([m',M'])}  \\
	&{}+ \norm{F(U(t) W_n^l)}_{F([m',M'])}  +\sum_{j=1}^{l_0} \norm{F(\chi_n^j U(t) W_n^l)}_{F([m',M'])}\\
	=:{}& I+II+III.
\end{aligned}
\end{equation*}
We first estimate $II$. By \eqref{eq:lpt_nlest1}, we obtain
\[
	II \le C \norm{U(t) W_n^l}_{L^{\rho(L),\I}([m',M'],L^{q(L)})}^{p-1}\norm{U(t) W_n^l}_{W_1([m',M'])}
\]
Since $W_n^l$ is uniformly bounded in $\FHsc$,
we deduce from Strichartz' estimate and \eqref{eq:critpf5} that
\[
	\lim_{k\to\I} \limsup_{n\to\I} \norm{F(U(t) W_n^{l_k})}_{F([m',M'])} =0.
\]
The estimate of $III$ is done in essentially the same way. 
We only note that changing of variable $x-t\xi^j_n =y$ and
application of the Galilean transform give us
\[
	\norm{F(\chi_n^j U(t) W_n^l)}_{F([m',M'])} = 
	\norm{F(\chi U(t) e^{-i\xi_n^jx}W_n^l)}_{F([m',M'])},
\]
and that Corollary \ref{cor:W_cutoff} and Strichartz' estimate imply
$\norm{\chi U(t) W_n^l}_{W_1([m',M'])}$ is uniformly bounded.

Let us proceed to the estimate of $I$. 
We consider only $j=1$, i.e.\ we treat
\[
	\norm{
	F({\Psi_1}_{[1,\xi_n^1]} + \chi_n^1 U(t) W_n^l) - F({\Psi_1}_{[1,\xi_n^1]})}_{
	F([m',M'])} .
\]
Arguing as in the proof of \eqref{eq:lpt_nlest2}, we have
\begin{equation}\label{eq:pfe2_1}
\begin{aligned}
	&\norm{F({\Psi_1}_{[1,\xi_n^1]} + \chi_n^1 U(t) W_n^l) - F({\Psi_1}_{[1,\xi_n^1]})}_{F([m',M'])}\\
	\le{}& C \norm{\chi_n^1 U(t) W_n^l}_{L^{\frac{4-N(p-1)}{4}}([m',M'],\Ms{s_c}22)}\norm{{\Psi_1}_{[1,\xi_n^1]}}_{L^{\rho_0,\I}([m',M'],L^{q_0})}^{p-1} \\&{} +
	C \norm{\chi_n^1 U(t) W_n^l}_{L([m',M'])}^{p-1}\norm{{\Psi_1}_{[1,\xi_n^1]}+\chi_n^1 U(t) W_n^l}_{W_1([m',M'])},
\end{aligned}
\end{equation}
where $\frac1{\rho_0}= \frac1{\rho(F)} - \frac{4-N(p-1)}{4}$
and 
\[
	q_0= \left\{
	\begin{aligned}
	&\frac{N(p^2-1)}{N(p-1)-2} && p\ge2,\\
	&\frac{4N}{N-2} &&p<2.
	\end{aligned}
	\right.
\]
One sees that $\norm{{\Psi_1}_{[1,\xi_n^1]}}_{W_1([m',M'])}$ is bounded uniformly in $n$.
Changing the variable by $x-t\xi^1_n =y$ and
applying the Galilean transform, we see
$ \norm{\chi_n^1 U(t) W_n^l}_{W_1([m',M'])} =\norm{\chi U(t) e^{-i\xi^1_n x}W_n^l}_{W_1([m',M'])} $.
By Corollary \ref{cor:W_cutoff} and Strichartz' estimate,
\begin{align*}
	\norm{ \chi U(t) e^{-i\xi^1_n x} W_n^l}_{W_1([m',M'])} &{}\le C\norm{U(t) e^{-i\xi^1_n x}W_n^l}_{W_1([m',M'])} \\ &{}\le C \norm{e^{-i\xi^1_n x}W_n^l}_{\FHsc} \le C.
\end{align*}
Now, by \eqref{eq:critpf5},
\begin{multline*}
	\limsup_{n\to\I} \norm{\chi_n^1 U(t) W_n^{l_k}}_{L^{\rho(L),\I}([m',M'],L^{q(L)})}\\
	\le \limsup_{n\to\I} \norm{U(t) W_n^{l_k}}_{L^{\rho(L),\I}([m',M'],L^{q(L)})}
	\to 0
\end{multline*}
as $k\to \I$. 
Let us next bound $\norm{{\Psi_1}_{[1,\xi_n^1]}}_{L^{\rho_0,\I}([m',M'],L^{q_0})}$.
By the generalized H\"older inequality,
\begin{multline*}
	\norm{{\Psi_1}_{[1,\xi_n^1]}}_{L^{\rho_0,\I}([m',M'],L^{q_0})}\\
	\le C\norm{{\Psi_1}_{[1,\xi_n^1]}}_{L^{(p-1)\rho(F),\I}([m',M'],L^{q_0})}^{\frac12}
	\norm{{\Psi_1}_{[1,\xi_n^1]}}_{L^{\widetilde{\rho},\I}([m',M'],L^{q_0})}^{\frac12},
\end{multline*}
where
\[
	\widetilde{\rho}= \left\{
	\begin{aligned}
	&\frac{2(p^2-1)}{p(N(p-1)-2)} && p\ge2,\\
	&\frac{8(p-1)}{(3N+2)(p-1)-8} &&p<2.
	\end{aligned}
	\right.
\]
Now, by the embedding (Lemma \ref{lem:embedding1}), 
\[
	\Lebn{{\Psi_1}_{[1,\xi_n^1]}}{q_0} \le C|t|^{-s_c}  \Mn{{\Psi_1}_{[1,\xi_n^1]}}{s_c}{\widetilde{q}}2t
	\le C(m')^{-s_c} \Mn{{\Psi_1}_{[1,\xi_n^1]}}{s_c}{\widetilde{q}}2t
\]
for $t\ge m'$, where
\[
	\widetilde{q}= \left\{
	\begin{aligned}
	&\frac{2N(p^2-1)}{4p- N(p-1)^2} && p\ge2,\\
	&\frac{4N(p-1)}{8-(N+2)(p-1)} &&p<2.
	\end{aligned}
	\right.
\]
Hence,
\[
	\norm{{\Psi_1}_{[1,\xi_n^1]}}_{L^{\rho_0,\I}([m',M'],L^{q_0})}\\
	\le C(m') \norm{{\Psi_1}_{[1,\xi_n^1]}}_{L^{\widetilde{\rho},2}([m',M'],\Ms{s_c}{\widetilde{q}}2)}.
\]
Here, we have used the embedding 
$$L^{\widetilde{\rho},2}([m',M'],\Ms{s_c}{\widetilde{q}}2) \hookrightarrow L^{(p-1)\rho(F),2}([m',M'],L^{q_0}) $$
which follows from Lemma \ref{lem:LMn_inclusion}.
Since $(\widetilde{\rho},\widetilde{q})$ is an admissible pair, we obtain the desired uniformly bound
for $\norm{{\Psi_1}_{[1,\xi_n^1]}}_{L^{\rho_0,\I}([m',M'],L^{q_0})}$.
Therefore, to give desired estimate on $I$, it suffices to show that
\begin{equation}\label{eq:pfe2_2}
	\limsup_{n\to\I}\norm{\chi U(t) e^{-i\xi_n^1 x}W_n^{l_k}}_{L^{2}([m',M'],\Ms{s_c}22)}
	\to 0
\end{equation}
as $k\to \I$. Indeed, since multiplication by $\chi$ is a bounded operator on $L^\I([m',M'], \Ms{s_c}22)$
as in Corollary \ref{cor:W_cutoff},
if \eqref{eq:pfe2_2} is established then
\begin{align*}
	&\norm{\chi U(t) e^{-i\xi_n^1 x}W_n^{l_k}}_{L^{\frac{4-N(p-1)}{4}}([m',M'],\Ms{s_c}22)}\\
	&{}\le C \norm{\chi U(t) e^{-i\xi_n^1 x}W_n^{l_k}}_{L^2([m',M'],\Ms{s_c}22)}^{\frac{4-N(p-1)}{2}}
	 \norm{U(t) e^{-i\xi_n^1 x}W_n^{l_k}}_{L^\I([m',M'],\Ms{s_c}22)}^{\frac{N(p-1)-2}{2}} \\
	&{}\le C \norm{W_n^{l_k}}_{\FHsc}^{\frac{N(p-1)-2}{2}} \norm{\chi U(t) e^{-i\xi_n^1 x}W_n^{l_k}}_{L^2([m',M'],\Ms{s_c}22)}^{\frac{4-N(p-1)}{2}}\\
	&{}\to 0
\end{align*}
as $k\to\I$. Putting this estimate and above estimates to \eqref{eq:pfe2_1},
the estimate of $I$ is completed.

Let us show \eqref{eq:pfe2_2}.
It holds for $t\neq0$ that
\begin{align*}
	\Mn{\chi U(t) e^{-i\xi_n^1 x}W_n^{l_k}}{s_c}22t
	&{}\sim |t|^{s_c} \hBn{M(-t)\chi U(t) e^{-i\xi_n^1 x}W_n^{l_k}}{s_c}22\\
	&{}\sim  |t|^{s_c} \hSobn{M(-t)\chi U(t) e^{-i\xi_n^1 x}W_n^{l_k}}{s_c}.
\end{align*}
Since $\rho(W_2) >\max(2,\rho(L))$, for some $\delta\in( \max(0,\frac12-\frac1{\rho(L)}),\frac12-\frac1{\rho(W_2)})$,
there exists $\theta \in (0,1)$ such that
\[
	\frac12 - \delta = \frac{1-\theta}{\rho(L)} + \frac{\theta}{\rho(W_2)}.
\]
For this $\theta$, we set $q_1$ by the relation
\[
	\frac1{q_1} = \frac{1-\theta}{q(L)} + \frac{\theta}{q(W_2)}.
\]
Now, we use the following Lemma by \cite{KPV};
\begin{lemma}
For any $1<p,q,r<\I$ with $\frac1p=\frac1{q}+\frac{1}{r}$
and $0< \alpha,\alpha_1,\alpha_2<1$ with $\alpha=\alpha_1+\alpha_2$,
we have
\[
	\Lebn{|\nabla|^\alpha (fg) - f|\nabla|^\alpha g - g |\nabla|^\alpha f}p
	\le C \Lebn{|\nabla|^{\alpha_1}f}q \Lebn{|\nabla|^{\alpha_2} g}r  
\]
whenever the right hand side is bounded.
\end{lemma}
By this lemma, 
\begin{align*}
	\hSobn{M(-t)\chi U(t) e^{-i\xi_n^1 x}W_n^{l_k}}{s_c}
	\le{}& \Lebn{\chi  |\nabla|^{s_c} M(-t) U(t) e^{-i\xi_n^1 x}W_n^{l_k}}2 \\
	&{}+ C_\chi \Lebn{ M(-t) U(t) e^{-i\xi_n^1 x}W_n^{l_k}}{q(L)}\\
	&{}+ C_\chi \Lebn{|\nabla|^{\theta s_c} M(-t) U(t) e^{-i\xi_n^1 x}W_n^{l_k}}{q_1}.
\end{align*}
It holds from the embedding $\dot{B}^{\theta s_c}_{q_1,2} \hookrightarrow
\dot{H}^{\theta s_c,q_1}$ that
\[
	|t|^{s_c}\Lebn{|\nabla|^{\theta{s_c}} M(-t) U(t) e^{-i\xi_n^1 x}W_n^{l_k}}{q_1}
	\le C{M'}^{(1-\theta)s_c} \Mn{U(t) e^{-i\xi_n^1 x}W_n^{l_k}}{\theta s_c}{q_1}2t
\]
for $t\le M'$.
Similarly, by embedding $\dot{B}^{\theta s_c}_{q_1,2} \hookrightarrow L^{q(L)}$,
\[
	|t|^{s_c}\Lebn{M(-t) U(t) e^{-i\xi_n^1 x}W_n^{l_k}}{q(L)}
	\le C{M'}^{(1-\theta)s_c} \Mn{U(t) e^{-i\xi_n^1 x}W_n^{l_k}}{\theta s_c}{q_1}2t
\]
for $t\le M'$.
Now, the generalized H\"odler inequality, 
Lemma \ref{lem:interpolation}, and Strichartz' estimate yield
\begin{align*}
	&\norm{ U(t) e^{-i\xi_n^1 x}W_n^{l_k} }_{L^2([m',M'], \Ms{\theta s_c}{q_1}2)}  \\
	&{} \le C \norm{1}_{L^{\delta^{-1},2}([m',M'])}
	\norm{ U(t) e^{-i\xi_n^1 x}W_n^{l_k} }_{L^{\frac2{1-2\delta},\I}([m',M'], \Ms{\theta s_c}{q_1}2)}\\
	&{} \le C{M'}^\delta \norm{ U(t) e^{-i\xi_n^1 x}W_n^{l_k}}_{L([m',M'])}^{1-\theta}
	\norm{ U(t) e^{-i\xi_n^1 x}W_n^{l_k}}_{W_2([m',M'])}^{\theta} \\
	&{} \le C{M'}^\delta \norm{ U(t) e^{-i\xi_n^1 x}W_n^{l_k}}_{L([m',M'])}^{1-\theta}
	 \norm{W_n^{l_k}}_{\FHsc}^{\theta}.
\end{align*}
We obtain desired smallness of this term by \eqref{eq:critpf5}.
On the other hand, it follows from $|2t\nabla|^{s_c}M(-t)U(t)=M(-t)U(t)|x|^{s_c}$ that
\begin{multline*}
	\norm{\chi |t|^{s_c} |\nabla|^{s_c} M(-t) U(t) e^{-i\xi_n^1 x}W_n^{l_k}}_{L^2([m',M']\times \R^N)} \\
	= C \norm{\chi U(t) |x|^{s_c}e^{-i\xi_n^1 x}W_n^{l_k}}_{L^2([m',M']\times \R^N)}.
\end{multline*}
By means of Lemma \ref{lem:cpt_supp_small},
for any $\eps$ there exists $C_\eps$ such that
\begin{multline*}
	\norm{\chi U(t) |x|^{s_c}e^{-i\xi_n^1 x}W_n^{l_k}}_{L^2([m',M']\times \R^N)} \\
	\le \eps \norm{W_n^{l_k}}_{\FHsc} + C_\eps 
	\norm{U(t) e^{-i\xi_n^1 x}W_n^{l_k}}_{L^{\rho(L),\I}(\R_+,L^{q(L)}) }.
\end{multline*}
Therefore we conclude from \eqref{eq:critpf5} that
\[
	\limsup_{n\to\I} \norm{\chi U(t) |x|^{s_c}e^{-i\xi_n^1 x}W_n^{l_k}}_{L^2([m',M']\times \R^N)} =0
\]
as $k\to\I$.
\end{proof}

\begin{remark}\label{rmk:FH1extension}
The minimizing problem
\[
	\ell_{c,\F H^1} := \inf \{  \ell_{\F H^1} (u_0)\ |\  u_0 \in \F H^1 \setminus S \} 
\]
can be treated in a similar way, where $\F H^1$ is given in \eqref{def:lfh1}.
Existence of the minimizer to this problem is shown in \cite{Ma} under $\pst<p<1+4/N$.
Here, we extend it as follows: Under the assumption \eqref{cond:p},
 there exists $\widetilde{u}_{0,c} \in \F H^1$ such that
$\widetilde{u}_{0,c} \not\in S_+$ and $\ell_{\F H^1}(\widetilde{u}_{0,c}) = \ell_{c,\F H^1}$.
Further, a solution $\widetilde{u}_c(t)$ to \eqref{eq:NLS} with data $\widetilde{u}_{0,c}$ is not a standing wave.

We give a sketch of proof. 
The strategy is the same as in the proof of Theorem \ref{thm:main1}
We first take a minimizing sequence for $\ell_{c,\F H^1}$ (we replace $S$ with $S_+$
without loss of generality).
By scaling, we may further assume that each function has a unit mass
and so that they are uniformly bounded in $\F H^1$.
Then, apply the profile decomposition (Proposition \ref{prop:pd}) to the sequence.
Uniform boundedness in $\F H^1$ enables us to establish
the Pythagoras decomposition \eqref{eq:pd3} with $\F H^1$ norm (see \cite{Ma}).
The rest of the argument is the same.
Recall that scattering in $\F H^1$ is equivalent to scattering in $\FHsc$, as noted in Remark \ref{rmk:nscondFH1}.
We prove that only one profile is involved in the decomposition 
and that the profile must not belong to $S_+$.
\end{remark}
\begin{remark}
A naive conjecture is that the above $\widetilde{u}_{0,c}$ is one of the function
satisfying the properties of Theorem \ref{thm:main1}.
However, it is not clear at least by the following two reasons.
First is that we do not know whether $u_{0,c}$ given in Theorem \ref{thm:main1} belongs to $\F H^1$ or not.
Even when $u_{0,c} \in \F H^1$ is true, minimality with respect to $\ell(\cdot)$ and 
that to $\ell_{\F H^1}(\cdot)$ are different, which is the second reason.
This fact is easily checked by the example given in Remark \ref{rmk:FHscFH1}.
Notice that $u_c(t)$ may blow up in finite time while $\widetilde{u}_c(t)$ is global in time.
\end{remark}
\subsection*{Acknowledgments}
This research is supported by Japan Society for the Promotion of Science(JSPS)
Grant-in-Aid for Young Scientists (B) 24740108.

\providecommand{\bysame}{\leavevmode\hbox to3em{\hrulefill}\thinspace}
\providecommand{\MR}{\relax\ifhmode\unskip\space\fi MR }
\providecommand{\MRhref}[2]{%
  \href{http://www.ams.org/mathscinet-getitem?mr=#1}{#2}
}
\providecommand{\href}[2]{#2}

\end{document}